\documentclass[11pt]{article}

 \usepackage[utf8]{inputenc}
 \usepackage[T1]{fontenc}  
\usepackage[english]{babel}

\usepackage{amsfonts}
\usepackage[bbgreekl]{mathbbol}
\DeclareSymbolFontAlphabet{\mathbbm}{bbold}
\DeclareSymbolFontAlphabet{\mathbb}{AMSb}%

\DeclareMathAlphabet{\mathmybb}{U}{bbold}{m}{n}
\newcommand*{\bbOmega}{\mathmybb{\Omega}}

\usepackage{amsfonts,amsmath,amsxtra,amsthm,latexsym,mathabx}
\usepackage{amssymb} 
\usepackage{accents}
\newtheorem{thm}{Theorem}[section]
 \newtheorem{cor}[thm]{Corollary}
 \newtheorem{lem}[thm]{Lemma}
 \newtheorem{prop}[thm]{Proposition}

 \theoremstyle{definition}
 \newtheorem{defn}{Definition}[section]
 \theoremstyle{remark}
 \newtheorem{rem}{Remark}[section]

  \newtheorem*{trem*}{Draft remark}

 \newtheorem{ex}[thm]{Example}
 
 \numberwithin{equation}{section}

\def\al{\alpha}
\def\th{\theta}
\def\Th{\Theta}
\def\vphi{\varphi}
\def\la{\lambda}

\def\ga{\gamma}
\def\Ga{\Gamma}
\def\ka{\ka}
\def\Om{\Omega}

\def\si{\sigma}
\def\De{\Delta}

\def\ka{\kappa}

\def\om{\omega}

\def\epb{\bbespilon}
\def\mub{\bbmu}

\DeclareMathOperator{\im}{Im}
\DeclareMathOperator{\re}{Re}
\DeclareMathOperator{\dom}{dom}
\DeclareMathOperator{\ran}{ran}

\DeclareMathOperator{\Mul}{\mathfrak{m}}
\DeclareMathOperator{\Gr}{Gr}
\DeclareMathOperator{\Ncone}{\bbomega}
\DeclareMathOperator{\curl}{curl}
\DeclareMathOperator{\curlm}{\mathbf{curl}}

\DeclareMathOperator{\Div}{div}
\DeclareMathOperator{\grad}{\mathbf{grad}}

\DeclareMathOperator{\curln}{\curlm_0}

\DeclareMathOperator{\Grad}{\mathbf{Grad}}
\DeclareMathOperator{\Curl}{\mathbf{Curl}}
\DeclareMathOperator{\Dev}{\mathbf{\De}_\mathrm{t}^{\pa \Om}}
\DeclareMathOperator{\Var}{\mathrm{Var}}

  \def\CC{\mathbb{C}}  
\def\FF{\mathbb{F}} \def\GG{\mathbb{G}} \def\HH{\mathbb{H}} \def\II{\mathbb{I}}
\def\KK{\mathbb{K}} \def\LL{\mathbb{L}} \def\NN{\mathbb N}
\def\PP{\mathbb{P}}  \def\RR{\mathbb{R}} \def\SS{\mathbb{S}}
 \def\VV{\mathbb{V}}\def\ZZ{\mathbb{Z}}

\def\Ff{\mathfrak{F}}

\newcommand{\x}{\mathbf{x}}

\def\pa{\partial}

\def\ii{\mathrm{i}}

\def\uph{\upharpoonright}
\def\wh{\widehat}
\def\wt{\widetilde}

\def\<{\langle}
\def\>{\rangle}
\def\imb{\hookrightarrow}
\def\ver{\big\arrowvert\!}

\def\A{\mathcal{A}}
\def\E{\mathbf{E}}
\def\H{\mathbf{H}}
\def\n{\mathbf{n}}
\def\af{\mathfrak{a}}
\def\bfr{\mathfrak{b}}
\def\ef{\mathfrak{e}}
\def\hf{\mathfrak{h}}

\def\Kc{\mathcal{K}}
\def\Tc{\mathcal{T}}

\def\uu{\mathbf{u}}
\def\vv{\mathbf{v}}
\def\ww{\mathbf{w}}
\def\Hc{\mathfrak{X}}

\def\Z{z}
\def\Zc{\mathcal{Z}}
\def\M{\mathcal{M}}

\def\Hs{\mathfrak{H}}
\def\A{\mathcal{A}}

\def\LLt{\LL^2_{\mathrm{t}} (\pa \Om)}
\def\LLT{\LL^2_{\mathrm{t}}}

\def\loc{\mathrm{loc}}

\def\sym{\mathrm{sym}}

\def\comp{\mathrm{comp}}

\def\imp{\mathrm{imp}}

\def\epr{\epb}

\def\Gb{G}
\def\G{\mathcal{G}}

\def\Yc{\mathfrak{Y}}

\def\Sop{S_{-,+}}
\def\Soin{S_{+,-}}
\def\Sdiv{\mathcal{S}_\ga}
\def\U{\mathcal{U}}
\def\Upi{\U_\pi}
\def\Uop{U}

\def\Himp{\HH_{\imp}}

\def\Fr{\mathrm{F}}
\def\KN{\mathrm{K}}
\def\Es{\mathfrak{E}}

\def\Lc{\mathcal{L}}
\def\Cay{\mathrm{Cay}}
\def\pOmp{[\pa \Om]_{\mathrm{p}}}
\def\pOma{[\pa \Om]_{\mathrm{a}}}

\begin{document}
\title{M-dissipative boundary conditions and boundary tuples for Maxwell operators} 
\author{}
\date{}
\maketitle

\vspace{-8ex}
{ \center{ \large
Matthias Eller$^\text{ a}$ and Illya M. Karabash$^\text{ b,c,*}$\\[3ex]
  }
}  
  
  {\small \noindent
$^{\text{a}}$ Department of Mathematics and Statistics, Georgetown University, Washington, DC 20057 \\[1mm]  
$^{\text{b}}$ Fakultät für Mathematik, TU Dortmund, Vogelpothsweg 87, 44227 Dortmund, Germany\\[1mm]
 $^{\text{c}}$ Institute of Applied Mathematics and Mechanics of NAS of Ukraine,
Dobrovolskogo st. 1, Slovyans'k 84100, Ukraine\\[0.2ex]
$^{\text{*}}$ Corresponding author: i.m.karabash@gmail.com\\[.2ex]
E-mails: mme4@georgetown.edu, i.m.karabash@gmail.com
}

\vspace{2ex}

\begin{abstract}
For Maxwell operators $(\E,\H) \to (\ii \epb^{-1} \nabla\times \H, -\ii \mub^{-1} \nabla \times \E)$ in Lipschitz domains, we describe all m-dissipative boundary conditions and apply this result to generalized impedance and Leontovich boundary conditions including the cases of singular, degenerate, and randomized impedance coefficients. 
To this end we construct Riesz bases in the trace spaces associated with the $\curlm$-operator and introduce a modified version of boundary triple adapted to the specifics of Maxwell equations, namely,
to the mixed-order duality of the related trace spaces.  This provides a translation of the problem to operator-theoretic settings of abstract Maxwell operators. In particular, we show that Calkin reduction operators are naturally connected with Leontovich boundary conditions and provide an abstract version of impedance boundary condition applicable to other types of wave equations. Taking  
Friedrichs and Krein-von Neumann extensions of related boundary operators, it is possible to associate m-dissipative Maxwell operators to arbitrary non-negative measurable impedance coefficients. 
\end{abstract}

\vspace{1ex}
\noindent
MSC-classes: 
35Q61 
47B44 
35F45   
35J56   
78A25  
58J90 
34G10 
\\[0.5ex]

\tableofcontents

\section{\label{s:Int} Introduction}


The evolution of the electromagnetic field in a domain $\Omega \subset \RR^3$ is governed by Maxwell's equations
\[ \quad
  \epr e_t = \nabla \times h, \qquad \mub h_t =- \nabla\times e,
\]
where the electric field $e=e(t,\x)$ and the magnetic field $h=h(t,\x)$ are 
$\CC^3$-valued functions. The $3\times 3$-matrix-functions $\mub=\mub(\x)$ and $\epr=\epr(\x)$ are the magnetic permeability and electric permittivity, respectively, which we suppose to be  time-independent. 
Assuming that $e(t,\x)=e^{-i\kappa t}\E(\x)$, $h(t,\x) = e^{-i\kappa t}\H(\x)$, the vector fields $\E$ and $\H$ satisfy the time-harmonic Maxwell equations \[
  \kappa \epr \E =\mathrm{i}\nabla \times \H, \quad \kappa  \mub \H=- \mathrm{i} \nabla \times \E.
\]
These equations can be understood as an eigenvalue problem for the Maxwell operator $\{\E,\H\} \to \{\ii \epb^{-1} \nabla\times \H, -\ii \mub^{-1} \nabla \times \E\}$. 

The primary goal of this work is to determine all boundary conditions imposed on the boundary $\pa \Omega$ of $\Om$ which make this operator $m$-dissipative. The second goals is to minimize the assumptions on the regularity and topology of $\pa \Omega$ and on the regularity of specific coefficients in boundary conditions and in Maxwell's equations. We apply our results to various classes of impedance, Leontovich, and mixed boundary conditions. 

In particular, the only our assumption on $\Omega$ is that it is a Lipschitz domain, i.e., $\Omega$ is an nonempty open connected bounded set with a Lipschitz continuous boundary $\pa \Om$ \cite{M03} (in particular, 
$\pa \Om$ may consist of finitely many connected components, which are not necessarily simply-connected). The constitutive tensor-fields $\mub (\x)$ and $\epr (\x)$ are represented by arbitrary uniformly positive definite and essentially bounded $3\times3$-matrix-functions.

To achieve the above aims, we modify  various tools of the operator extension theory and introduce abstract versions of Maxwell operators. With the use of the theory of differential operators on Lipschitz manifolds and corresponding Hodge decompositions, the abstract results written in terms of generalized boundary tuples are translated then to the results on concrete boundary conditions.

M-dissipative boundary conditions naturally appear in the theory of linear evolution equations because they ensure well-posedness and uniform stability. More precisely, an operator $T$ in a Hilbert space $\Hc$ is m-dissipative exactly when $(-\ii)T$ is a generator of strongly continuous semigroup of contractions \cite{P59}. This provides an elegant semigroup approach to
dynamic Maxwell equations and to modeling of leaky optical cavities. Note that we use the Physics conventions \cite{E12} that put the numerical range of a dissipative operator into the closed lower complex half-plane $\overline{\CC}_-$, and that in the PDE theory a related result is well-known in a more general context of Banach spaces under the name of the Lumer-Phillips theorem.  

While a general description m-dissipative boundary conditions for Maxwell systems has not been available even in the case of smooth $\pa \Om$, some classes of selfadjoint and m-dissipative boundary conditions for the time-harmonic Maxwell equations are well known,  for example, the  condition $\n \times \E =0$ which, roughly speaking, models the boundary as a perfect conductor.  Here $\n$ is the exterior normal vector field of $\Omega$ along $\pa \Om$. This boundary condition leads to a self-adjoint operator \cite{BS87,M03,ACL17}.
(Rigorously, $\n \times \E =0$ have to be supplemented by $\n \cdot (\mub \H  )=0$ to obtain the perfect conductor condition). 

 Another physical relevant example is the \emph{Leontovich boundary condition} \cite{LL84,K94,YI18}  $\n\times \E = \Z (\x) \pi_\top (\H)$ where the \emph{impedance coefficient} $\Z$ is a complex valued function on $\pa \Om$ with non-negative real part and $\pi_\top (\H)= \n\times (\H\times \n)$ is the tangential component of the magnetic field $\H$ on $\partial\Omega$.  Kapitonov   \cite{K94} showed how to assign naturally an m-dissipative Maxwell operator to this condition as long as $\pa \Om$ is smooth and $\Z \in C^1(\partial\Omega)$. On the other hand Lagnese and Leugering \cite[Lemma 7.2.2.1]{LL04} using the density result of Belgacem et al. \cite{BBCD97} proved that as long as $\Z \in L^\infty(\partial\Omega,\mathbb{R})$ is strictly positive on $\partial\Omega$, the condition $\n\times \E = \Z (\x) \pi_\top (\H)$ produces an m-dissipative operator  (formally, $\pa \Om$ in \cite{LL04} is piecewise smooth, but actually the proof  can be easily adapted to Lipschitz domains).
The case of constant impedance $\Z \equiv \alpha>0$ is sometimes referred to as the transparent or absorbing boundary condition since it is an approximation of the Silver-Müller radiation condition.   

Of particular interest are \emph{mixed boundary conditions} that in the context of this paper always understood as the mix of the conservative condition $\n \times \E =0$ on a part $\pOmp$ of the boundary with an  absorbing or Leontovich boundary condition on the rest $\pOma = \pa \Om \setminus \pOmp$
(in other papers, the same name is often used also for other types of mixed conditions).

If Kapitonov's  regularity assumptions \cite{K94} are not fulfilled, the question of m-dissipativity of mixed boundary conditions and well-posedness of associated Maxwell systems become substantially more difficult \cite{M03,ACL17}. Certain results are obtained under the assumptions that the impedance coefficient on $\pOma$ is uniformly strictly positive and bounded and, additionally,  $\Om$ and $\pOma$  are `good enough' (in particular, the boundary of $\pOma$ is piecewise smooth) and have no certain undesirable interplays with possible ``pathologies'' of $\pa \Om$ (see the discussion in the recent monograph \cite[Section 5.1.2]{ACL17}).
As a particular applications of our general results we show how to approach the m-dissipativity for arbitrary measurable subsets $\pOma$ of the Lipschitz manifold $\pa \Om$ and arbitrary measurable nonnegative $\Z (\cdot)$, see Section \ref{s:discussion}.
 
 Our interest is motivated by contemporary studies in radiophysics, photonics, and optical engineering that involve singular structures (like fractal antennas \cite{LM12}) or modeling of leaky optical cavities \cite{M03,MSG17}, which is needed, in particular, for the rigorous formulation for Maxwell systems of nonselfadjoint eigenvalue optimization problems   \cite{CZ94,K13a,AK17,MSG17,EK21}.
In photonics context, leaky cavities are
 surrounded by a medium with an uncertain structure, e.g., due to uncertain coupling or random fabrication errors \cite{MSG17}. The latter rises a natural question of an accurate randomization of dissipative boundary conditions (see Section \ref{s:discussion}).


Additional conditions of type  $\n \cdot (\mub \H ) =0$ are closely connected with the `divergence-free' conditions $\Div (\epb(\x) \E (\x)) = 0$,  $\Div (\mub(\x) \H (\x)) = 0$, and corresponding Weyl decompositions, which in the m-dissipative case are treated briefly by Remark \ref{r:div=0} (cf. the selfadjoint case in \cite{BS87}). However, the condition $\n \cdot (\mub \H  )=0$ on $\pa \Om$ (or on a part of $\pa \Om$) remains outside  the scope of the paper.

With the `divergence-free' conditions the Maxwell system
becomes elliptic, which makes it possible, in the case where $\pa \Om$ and the coefficients  are smooth enough, to apply the theory of boundary value problems for elliptic equations \cite{G09}. However, photonics and optical engineering applications, as well as numerical methods, often involve  polyhedral domains and discontinuous coefficients, which makes the use of the elliptic theory difficult.       

\textbf{Notation and terminology.} 
If a linear space $\wt \Hc$ is a subset of a linear space $\Hc$, we say that $\wt \Hc$ is a subspace of $\Hc$. If 
a subspace $\wt \Hs$ of a Hilbert space $\Hs$ is closed in $\Hs$ and we want to emphasize this, we say that $\wt \Hs$ is a closed or Hilbert subspace of $\Hs$; in this case, $\Hs \ominus \wt \Hs$ is the orthogonal complement of $\wt \Hs$ in $\Hs$.  The orthogonal sum of Hilbert spaces $\Hs_{1,2}$,
is denoted by $\Hs_1 \oplus \Hs_2$.
A norm $\| \cdot \|_{\wt \Hs}$ on a linear space $\wt \Hs$ is called Hilbertian if there is an inner product $(\cdot|\star)_{\wt \Hs}$ defined on $\wt \Hs \times \wt \Hs$ such that $\| \cdot \|_{\wt \Hs}^2 = (\cdot|\cdot)_{\wt \Hs}$. 

Let $\Hc_j$, $j=1, \dots, 4$, be linear spaces.
An operator $A: \dom A \subseteq \Hc_1 \to \Hc_2$ from $\Hc_1$ to $\Hc_2$ is understood as a linear mapping defined on a subspace $\dom A \subseteq \Hc_1$ ($\dom A$ is called the domain of $A$) and having the range $\ran A =\{ A f : f \in \dom A\}$ in $\Hc_2$.  If $\dom A = \Hc_1$ and we want to emphasize this, we write $A: \Hc_1  \to \Hc_2$. 
A restriction $B=A \uph_{\wt \Hs}$ of an operator $A$ to $\wt \Hs \subseteq \dom A$ is an operator $B: \wt \Hs \subseteq \Hc_3 \to \Hc_4$ defined by $Bf =Af$, $f \in \wt \Hs$, where $\wt \Hs = \dom B$ is a subspace of $\Hc_3$ (and so $\dom B \subseteq \Hc_1 \cap \Hc_3$) and  the image $\ran B$ is a subspace of $\Hc_4$ (and so $\ran B \subseteq \Hc_2 \cap \Hc_4$). If $\Hc_3 = \Hc_1$ and $\Hc_4 =\Hc_2$, we  do not specify explicitly the spaces $\Hc_{3,4}$;  in this case an extension of $B$ is an operator $A$ such that $B$ is a restriction of $A$ (so $A$ is simultaneously an extension and a restriction of itself).  An extension $\wt A$ of a symmetric densely defined operator $A$  is called \emph{admissible} if $\Gr A \subseteq \Gr \wt A \subseteq \Gr A^*$.
By $\ker A := \{f \in \dom A \ : \ Af=0 \}$, we denote the kernel of $T$. 

Operators $A$ from a Hilbers space $\Hc_1$ to a Hilbert space $\Hc_2$ are sometimes identified with their graphs $\Gr A := \{ \{f,A f\} \ : \ f \in \dom A\} \subset \Hc_1 \oplus \Hc_2$. 
While such identification allows one to consider an operator as a particular case of a linear relation from $\Hc_1$ to $\Hc_2$, we avoid the use of this operator theory identification since it is not standard for the PDE theory. 
A linear relation $\Th$ from $\Hc_1$ to $\Hc_2$ is a subspace of the orthogonal sum $\Hc_1 \oplus \Hc_2$ (for related basic definitions see \cite{BHdS20,DM95,DM17,HW12} and Section \ref{s:MBT-BT}). 

A nonnegative symmetric operator $A$ in $\Hs$ is called \emph{positive} if $\ker A = 0$, and is called \emph{uniformly positive} if $A \ge c I_\Hs$ for a certain constant $c >0$, where $\ge$ is the standard partial ordering of  nonnegative symmetric operators and $I_\Hs f = f$, $f \in \Hs$.

\section{Main results and methods of the paper}

\subsection{Main results on m-dissipative Maxwell operators}
\label{s:Mresults}

Let $T:\dom T \subseteq \Hc \to \Hc$ be an operator in a Hilbert space  $\Hc$. The following  definitions fix our particular choice of basic conventions.

\begin{defn}[cf. \cite{Kato13,E12}] \label{d:dis}
An operator $T$ is called \emph{dissipative} 
if $\im (Tf|f)_{\Hc} \le 0$ 
for all $f \in \dom T$. 
A dissipative operator is \emph{maximal dissipative} if it is not a proper restriction of another dissipative operator.
An operator $T$ is called  \emph{m-dissipative} 
if $\CC_+ := \{\la \in \CC : \im \la >0\}$ 
is a subset of its resolvent set 
and 
$\| (T- \la I_\Hc)^{-1}\| \le (\im \la)^{-1}$ for all $\la \in \CC_+$.
An operator $T$ is called \emph{accretive} or \emph{m-accretive} if $(-\ii) T$ is dissipative or, resp., m-dissipative. 
A closeable operator $T$ is \emph{essentially m-dissipative} (\emph{essentially m-accretive}) if its closure $\overline{T}$ is m-dissipative (resp., m-accretive).
\end{defn}

\begin{defn}[cf. \cite{P59,E12,Kato13}] \label{d:contraction}
An operator $T:\dom T \subseteq \Hc \to \Hc$ is called \emph{contractive in $\Hc$} if 
$\| T h \|_{\Hc} \le \|h\|_\Hc$ for all $h \in \dom T$ (contractive operators are not necessarily closed or densely defined). A contractive operator $K$ in $\Hc$ is said to be a \emph{contraction on} $\Hc$ if $\dom K = \Hc$.
\end{defn}

We always assume that $\Om$ is a Lipschitz domain in $\RR^3$, i.e., $\Om$ is an open nonempty bounded connected set with the boundary $\pa \Om$ satisfying Lipschitz regularity condition \cite{M03}. 
The outward unit vector $\n (\x)$ normal to $\partial \Omega$ at $\x$ is defined for almost all (a.a.) 
$\x \in \pa \Om$ (w.r.t. the surface measure of $\pa \Om$). The resulting measurable  and essentially bounded $\RR^3$-vector field $ \n (\cdot)$ belongs to the space $L^\infty  (\pa \Om, \RR^3)$.

By $\LL^2 (\Om)=L^2 (\Om,\CC^3)$ we denote the Hilbert space of complex 3-D vector fields in $\Om$ equipped with the standard sesquilinear inner product 
$
(  \uu | \vv )_{\LL^2 } =  \int_{\Om} \uu \cdot \overline{\vv} =  \int_{\Om} (  \uu  (\x) | \vv (\x) )_{\CC^3} $.

The space 
$ 
\HH (\curlm,\Om) := \{ \uu \in \LL^2 (\Om) \ : \ \nabla \times \uu \in \LL^2 (\Om) \}, 
$
where $\nabla \times \uu$ is understood in the distribution sense,
is the domain (of definition) of the operator $\curlm : \uu \mapsto  \nabla \times \uu$ acting in $\LL^2 (\Om)$. 
The space $\HH (\curlm,\Om)$ and other vector spaces built as domains of operators are assumed to be equipped with the graph norms. 

The space $
\HH_0 (\curlm, \Om) 
$
is the closure in $
\HH (\curlm,\Om) $ of the space $C^\infty_0 (\Om; \CC^3)$ of compactly supported in $\Om$ smooth  $\CC^3$-vector-fields.  
The operator $\curln$ is defined as the closure in $\LL^2 (\Om)$ of the densely defined restriction $\curlm \uph_{C_0^\infty (\Om; \CC^3)}$.
So $\curlm=\curln^*$ and  $\curln=\curlm^*$ are closed operators, and $\curln$ additionally is symmetric. In particular, $\HH (\curlm,\Om) $ and $ \HH_0 (\curlm, \Om) = \dom \curln $
are Hilbert spaces.

Let us consider the material tensor fields $\epb$ (dielectric permittivity), $\mub$ (magnetic permeability), and an abstract tensor field $\bbxi$ that are given by essentially bounded matrix-valued functions $\epb (\cdot), \mub (\cdot), \bbxi (\cdot) \in L^\infty (\Om, \RR_{\sym}^{3\times3})$,
where $\RR_{\sym}^{3\times3}$ is the Banach space of $3\times3$ real-valued symmetric matrices (the choice of a norm in $\RR_{\sym}^{3\times3}$ is not important). 
We always assume  that there exists a constant $c>0$ such that 
$\epb (\x) \ge c \II $, $\mub (\x) \ge c \II$, and $\bbxi (\x) \ge c \II$ for a.a. $\x \in \Om$.

The `weighted' Hilbert space $ \LL^2_{\epb,\mub} = \LL^2_{\epb,\mub} (\Om)$ coincides with the orthogonal sum $\LL^2 (\Om) \oplus \LL^2 (\Om) $ as a linear space, but is equipped with the energy norm $ \| \cdot \|_{\LL^2_{\epb,\mub} }$ defined by
$
 \| \{\E,\H\} \|_{\LL^2_{\epb,\mub}}^2 =  ( \epb \E | \E)_{\LL^2 (\Om) }^2 + 
 ( \mub \H|\H)_{\LL^2 (\Om)}^2 .
$
In $ \LL^2_{\epb,\mub} (\Om)$, we consider the symmetric Maxwell operator 
\begin{gather}\label{e:iM}
\M \begin{pmatrix} \E \\ \H \end{pmatrix} = \begin{pmatrix} 0 & \ii \epb^{-1} \curln \\
-\ii \mub^{-1} \curln & 0 \end{pmatrix} \begin{pmatrix} \E \\ \H \end{pmatrix} , 
\quad 
\{  \E , \H \} \in \dom \M = \HH_0 (\curlm,\Om)^2 , 
\end{gather}
where $\HH_0 (\curlm,\Om)^2 = \HH_0 (\curlm,\Om) \times \HH_0 (\curlm,\Om)$.
This closed symmetric densely defined operator corresponds to a transparent nonhomogeneous (generally, anisotropic) medium \cite{ACL17}. 

The adjoint in $\LL^2_{\epb,\mub} (\Om)$ operator $\M^*$ 
has the same differential expression as $\M$, but the wider domain, $ \dom \M^* = \HH (\curlm,\Om)^2$, which is maximal natural in $\LL^2_{\epb,\mub} (\Om)$ for its differential expression. 
Let $\Hs_0$ be a linear space. Consider operators 
$\G_j:\dom \G_j \subseteq \dom \M^* \to \Hs_0$, $j=0,1$, and 
an 'abstract boundary condition' $  \G_0 f +  \G_1 f = 0 $.

The \emph{Maxwell operator $\wt M$ defined by the condition}  $  \G_0 f +  \G_1 f = 0 $ is the  restriction of the operator $\M^*$ to the set $\dom \wt \M := \{ f \in \dom \G_0 \cap \dom \G_1 \ : \    \G_0 f +  \G_1 f = 0 \}$.  
One of the main results of this paper, the description of all boundary conditions defining m-dissipative Maxwell operators, is given by Theorem \ref{t:m-dMax}. 

The formulation of this result uses the following objects and notions.
The space 
\[
\LLT = \LLt = \{ \vv \in L^2 (\pa \Om,\CC^3) \ : \ \n \cdot \vv = 0 \text{ a.e.}\}
\]
 is the $L^2$-space of the tangential vector fields  on the Lipschitz manifold $\pa \Om$.
The spaces $\HH^{-1/2} (\curl_{\pa \Om})=\HH^{-1/2} (\curl_{\pa \Om}, \pa \Om)$ and $\HH^{-1/2} (\Div_{\pa \Om})=\HH^{-1/2} (\Div_{\pa \Om}, \pa \Om)$ will be called \emph{the trace spaces of $\HH (\curlm,\Om)$}. 
They consist of  tangential vector-fields on $\pa \Om$ (generally, of negative order of regularity)  and can be defined via the surface scalar curl-operator $\curl_{\pa \Om}$  and the surface divergence $\Div_{\pa \Om}$ (see \cite{BCS02,BHPS03,M04,ACL17} and also Section \ref{s:curl}). 
The spaces $\HH^{-1/2} (\curl_{\pa \Om})$ and $\HH^{-1/2} (\Div_{\pa \Om})$ are dual to each other w.r.t. the pivot space $\LLt$ in the generalized `mixed-order norm' sense. We call this generalized duality \emph{m-order duality}, see the explanations in Section \ref{s:OT} and \ref{s:MDAbstract}. The space $\HH^{-1/2} (\curl_{\pa \Om})$ is the image of $\HH (\curlm, \Om)$ under the tangential component projection trace $\pi_\top (\uu) = - \n \times (\n  \times \uu\!\uph_{\pa \Om})$ and  
\begin{gather*}
\text{the norm $|\uu|_\pi  := \| \pi_\top^{-1} \uu \|_{\HH (\curlm, \Om)/\HH_0 (\curlm, \Om)}$ makes $\HH^{-1/2} (\curl_{\pa \Om})$ a Hilbert space;
}
\end{gather*}
similarly, the space 
$\HH^{-1/2} (\Div_{\pa \Om})$ is the image of $\HH (\curlm, \Om)$ under the the tangential trace $\ga_\top (\uu) = - \n  \times \uu\!\uph_{\pa \Om} $ (roughly speaking, $ \HH^{-1/2} (\Div_{\pa \Om}) =$ ``$\n \times$''\,$ \HH^{-1/2} (\curl_{\pa \Om}))$ and  
\begin{gather*}
\text{the norm $|\uu|_\ga  := \| \ga_\top^{-1} \uu \|_{\HH (\curlm, \Om)/\HH_0 (\curlm, \Om)}$ makes $\HH^{-1/2} (\Div_{\pa \Om})$ a Hilbert space.
}
\end{gather*}
These facts and the equivalence of norms $|\cdot|_{\pi(\ga)}$ to the original graph norms of the trace spaces follow from the results of  Buffa et al. \cite{BCS02} and Mitrea \cite{M04}, see (\ref{e:TrTh}) in Section \ref{s:curl}.
The $\LLT$-pairing-adjoint operator $V^\#$ is understood in the sense of Section \ref{s:MDAbstract}. Recall that $I_{\LLT}$ is the identity operator in $\LLt$.

\begin{thm} \label{t:m-dMax} 
Let $V$ be a certain fixed  linear homeomorphism from $\LLt $ to $\HH^{-1/2} (\curl_{\pa \Om})$ and let $V^\#$ be its $\LLt$-pairing-adjoint. Then the two following statements are equivalent:
\item[(i)] An extension $\wh \M$ of $\M$ is an m-dissipative operator in $\LL^2_{\epb,\mub} (\Om)$.
\item[(ii)] There exist a contraction 
$K$ on $\LLt$ such that $\wh \M$ is the Maxwell operator defined by the boundary condition
$(I_{\LLT}+K) V^{-1}  \pi_\top   \H + (I_{\LLT}-K) V^\# \ga_\top  \E  = 0$.

This equivalence establishes a 1-to-1 correspondence between m-dissipative extensions $\wh \M$ of $\M$ and contractions $K$ on $\LLt$. Moreover, $\wh \M$ is selfadjoint if and only if the corresponding contraction $K$ is a unitary operator. 
\end{thm}

Theorem \ref{t:m-dMax} is particular case of more general Theorem \ref{t:absM-dis} (see  Section \ref{s:mdM}).

To make the description of Theorem  \ref{t:m-dMax} concrete, let us give here an explicit example of a linear homeomorphism $V: \LLt \to \HH^{-1/2} (\curl_{\pa \Om})$. Its construction is based on the Hodge decompositions of $\LLt$ and of the trace spaces of $\HH (\curlm,\Om)$. 

The space $\LLt$ admits the orthogonal (Hodge) decomposition \cite[formula (4.6)]{T83} 
\begin{gather} \label{e:HDL2}
\LLt = \grad_{\pa \Om} H^1 (\pa \Om) \oplus \KK_1  (\pa \Om) \oplus \curlm_{\pa \Om} H^1  (\pa \Om),  
\end{gather}
where  $\KK_1 = \KK_1 (\pa \Om) := \{ \vv \in \LLt \; : \; 0=\Div_{\pa \Om } \vv = \curl_{\pa \Om } \vv \}$
is the cohomology space  of $\pa \Om $, $\grad_{\pa \Om}$ is  the tangential gradient,  and $\curlm_{\pa \Om} $ is the  surface vector  curl-operator. The space  $\KK_1 (\pa \Om)$ consists of harmonic tangential vector fields and is finite-dimensional. (Its dimensionality  $\dim \KK_1 (\pa \Om)$ is equal to the 1st Betti number $b_1 (\pa \Om)$ of $\pa \Om$ and is related to the number of cuts necessary to make $\pa \Om$ simply connected.)

Let us denote by $\KK_0 = \KK_0 (\pa \Om)$ the space of locally constant scalar functions on $\pa \Om$. Then 
\[
\text{$\KK_0  = \ker \grad_{\pa \Om} = \ker \curlm_{\pa \Om}  $ and $\dim \KK_0  \in \NN $ is the 0-th Betti number $b_0 (\pa \Om)$ of $\pa \Om$}
\]
(i.e., $\dim \KK_0$ is the number of connected components of $\pa \Om$).
One can consider $\grad_{\pa \Om}$ and $\curlm_{\pa \Om}$ as operators defined on
the Hilbert factor-spaces $H^s_{\pa \Om} := H^s (\pa \Om)/\KK_0$ where $H^s (\pa \Om)$ are standard Hilbertian Sobolev spaces on $\pa \Om$ of regularity $s \in [1/2,1]$. 
In particular, $\grad_{\pa \Om}$ and $\curlm_{\pa \Om}$ map $H^1_{\pa \Om}$ homeomorphically to the closed subspaces $\grad_{\pa \Om} H^1 (\pa \Om)$ and $\curlm_{\pa \Om} H^1 (\pa \Om)$ of $\LLt$. We denote by 
\[
\text{$\Grad_{\pa \Om}^{-1}:\grad_{\pa \Om} H^1 (\pa \Om) \to H^1_{\pa \Om}$ and $\Curl_{\pa \Om}^{-1}:\curlm_{\pa \Om} H^1 (\pa \Om) \to H^1_{\pa \Om}$ }
\] the corresponding inverse homeomorphisms, and by $\De_{\pa \Om}$ the selfadjoint Laplace-Beltrami operator in the factor-space $L^2 (\pa \Om) /\KK_0$ (see Section \ref{s:Upi}).

\begin{thm} \label{t:Upi}
There exists a homeomorphism $\Upi:\HH^{-1/2} (\curl_{\pa \Om}) \to \LLt$ such that its inverse has the form
\begin{gather*} 
\Upi^{-1}  = \grad_{\pa \Om} \De_{\pa \Om}^{1/4} \Grad_{\pa \Om}^{-1} \dot + I_{\KK_1}\dot + \curlm_{\pa \Om} \De_{\pa \Om}^{-1/4} \Curl_{\pa \Om}^{-1} .
\end{gather*} 
Moreover,  the $\LLT$-pairing adjoint $ \Upi^{\#}$ to $\Upi$ is a homeomorphism from $\LLt$ to $\HH^{-1/2} (\Div_{\pa \Om}) $ of the  form  
\begin{gather*} 
 \Upi^{\#}  = \grad_{\pa \Om} \De_{\pa \Om}^{-1/4} \Grad_{\pa \Om}^{-1} \dot +  I_{\KK_1} \dot + \curlm_{\pa \Om} \De_{\pa \Om}^{1/4} \Curl_{\pa \Om}^{-1} .
\end{gather*} 
(Here and below the decompositions of operators defined 
in $\LLt$ are given w.r.t. the Hodge decomposition (\ref{e:HDL2}) of $\LLt$ and the direct Hodge decompositions of the corresponding trace spaces, see \cite{BCS02,BHPS03} and Section \ref{s:Upi} for details.)
\end{thm}

This theorem follows from Corollary \ref{c:Upiga} and Remark \ref{r:Upi}. 
Theorem \ref{t:Upi} is applied to the boundary condition in Theorem \ref{t:m-dMax} by taking $V^{-1} = \Upi$ and $V^\# = (\Upi^\#)^{-1}$.

Let us consider boundary conditions of another form
$\Zc \pi_{\top} \H + \ga_{\top} \E  = 0 $
with an operator $\Zc:\dom \Zc \subseteq \HH^{-1/2} (\curl_{\pa \Om}) \to \HH^{-1/2} (\Div_{\pa \Om})$. Conditions of this class are called in \cite[Sections 1.6.1 and 8.3.3]{ACL17} generalized impedance boundary conditions (see  \cite{ELN,YI18} for other generalizations of impedance and Leontovich boundary conditions). 

In the next result a characterization of all 
m-dissipative generalized impedance boundary conditions is given. We use 
\[
\text{the $\LLT$-duality pairing $\<\cdot|\star\>_{\LLT}$ of the spaces $\HH^{-1/2} (\curl_{\pa \Om})$ and $\HH^{-1/2} (\Div_{\pa \Om}) $}
\] 
constructed on the base of the inner product $(\cdot|\star)_{\LLt}$ of \emph{the pivot space} 
$\LLt$, see Remark \ref{r:<>H} and Corollary \ref{c:DevDuality}.
We say that an operator $\Zc:\dom \Zc \subseteq \HH^{-1/2} (\curl_{\pa \Om}) \to \HH^{-1/2} (\Div_{\pa \Om})$ is \emph{accretive} if $\re \<\Zc \uu | \uu \>_{\LLT} \ge 0$ for $\uu \in \dom \Zc$, and is \emph{maximal accretive} if it has no proper accretive extensions. The following result is a particular case of Corollary \ref{c:I-T}. 

\begin{cor} \label{c:Z}
The Maxwell operator defined by $\Zc \pi_{\top} \H + \ga_{\top} \E  = 0 $ is m-dissipative if and only if 
$\Zc:\dom \Zc \subseteq \HH^{-1/2} (\curl_{\pa \Om}) \to \HH^{-1/2} (\Div_{\pa \Om})$ is closed and maximal accretive.
\end{cor}

\begin{rem} \label{r:2ndCh}
It is easy to see that (unlike Theorem \ref{t:m-dMax})  Corollary \ref{c:Z} does not cover all m-dissipative boundary conditions for Maxwell operators. This can be corrected, roughly speaking, by the replacement of operators $\Zc$ with maximal accretive linear relations from $\HH^{-1/2} (\curl_{\pa \Om})$ to $ \HH^{-1/2} (\Div_{\pa \Om})$, see Theorem \ref{t:MaxwellBT}.
\end{rem}

We use Corollary \ref{c:Z} and the abstract descriptions of Remark \ref{r:2ndCh} for the study of Leontovich boundary conditions. The functional space $\Himp (\curlm, \Om)$ usually associated \cite{M03,LL04} with various surface impedance boundary conditions is defined by 
\begin{equation} \label{e:Himp}
\Himp (\curlm, \Om)  = \{ \uu \in \HH (\curlm, \Om) : \ga_\top \uu \in \LLt  \} = \{ \uu \in \HH (\curlm, \Om)  : \pi_\top \uu \in \LLt \}.
\end{equation}
Consider the restrictions $\pi_{\top,2}(\ga_{\top,2}) \, : \, \Himp (\curlm, \Om) \to \LLt $ of the tangential-component trace and the tangential trace, i.e.,
$\pi_{\top,2} := \pi_\top \uph_{\Himp (\curlm, \Om)}, $ and $\ga_{\top,2} := \ga_\top \uph_{\Himp (\curlm, \Om)} .$

Let $\Z : \pa \Om \to \ii \overline{\CC}_- $ be a measurable function (w.r.t. the surface measure of $\pa \Om$), where $\overline{\CC}_\pm = \{z\in \CC: \pm \im z \ge 0\}$ and $\ii \overline{\CC}_- = \{z\in \CC:\re z \ge 0\}$. Then we say  (cf. \cite{LL84,K94}) that 
\begin{gather} \label{e:GIBC}
\Z (\x) \pi_{\top,2} \H (\x)+ \ga_{\top,2} \E (\x) = 0 \text{ a.e. \ \ is \emph{a Leontovich-type boundary condition}}
\\
\text{and that the Maxwell operator defined by  (\ref{e:GIBC}) is a \emph{Leontovich-type operator} $\Lc_\Z$.} \label{e:Lz}
\end{gather}
The function $\Z (\cdot)$ is called the \emph{impedance coefficient}.
(In the mathematical literature, (\ref{e:GIBC}) with positive $\Z $ is often called impedance boundary condition \cite{M03,ACL17}.)

Conditions of the type (\ref{e:GIBC}) up to our knowledge were independently introduced and studied  in  Radiophysics publications of Shchukin and Leontovich, see \cite{LL84,YI18}. It was important in   Leontovich's settings that the impedance coefficient $\Z (\cdot)$ is allowed to be complex-valued.
A detailed physical explanation can be found, e.g., in the monograph of Landau, Lifshitz, and Pitaevskii \cite[Section 87]{LL84}, where additionally a purely imaginary constant impedance coefficient $\Z \equiv \al \in  \ii \RR_-$ was connected with a superconductivity approximation.  

Leontovich-type operators are always dissipative (which is seen directly by integration by parts), but not necessarily m-dissipative (see Proposition \ref{p:non-m-dis}).

We obtain from Corollary \ref{c:Z} the following characterization of m-dissipativity 
for boundary conditions slightly more general than (\ref{e:GIBC}). 
Let $\Sdiv:\dom \Sdiv \subset \LLt \to \LLt$ be the restriction of the operator $\Upi^{-1} $ 
to $\dom \Sdiv := \HH^{-1/2} (\Div_{\pa \Om}) \cap \LLt $. It is possible to show 
that $\Sdiv$ is a selfadjoint operator in $\LLt$ and that it can be written explicitly w.r.t. the Hodge decomposition of $\LLt$ as
\begin{gather} \label{e:Sga0}
\Sdiv  = \grad_{\pa \Om} \Bigl( \De_{\pa \Om}^{1/4} \uph_{\dom \De_{\pa \Om}^{3/4}} \Bigr) \Grad_{\pa \Om}^{-1} \  \oplus \ I_{\KK_1} \ \oplus \ \curlm_{\pa \Om} \De_{\pa \Om}^{-1/4} \Curl_{\pa \Om}^{-1},  
\end{gather} 
see (\ref{e:Spm}) and Section \ref{s:Upi}.
Let $\Mul_\Z$ be the operator of multiplication on $\Z (\cdot)$ in $\LLt$, i.e.,
\[
\Mul_\Z:\uu (\cdot) \mapsto\Z (\cdot) \uu (\cdot) , \qquad \dom \Mul_\Z :=\{ \uu \in \LLt : \Z (\cdot) \uu (\cdot) \in \LLt \}.
\]

\begin{thm} \label{t:CritGIMC}
(i) Let $\wh \Zc$ be an operator in $\LLt$. Then the Maxwell operator defined by 
$\wh \Zc \pi_{\top,2} \H + \ga_{\top,2} \E  = 0 $
is m-dissipative 
if and only if the operator $\Sdiv \wh \Zc \Sdiv$ in $\LLt$ is m-accretive. 
In particular, the Leontovich-type operator $\Lc_\Z$ is m-dissipative if and only if
$\Sdiv \Mul_\Z \Sdiv$ is m-accretive.

\item[(ii)]If $\Z (\x) \ge 0$ a.e. on $\pa \Om$, then $\Lc_\Z$ is m-dissipative if and only if $\Sdiv \Mul_\Z \Sdiv$ is  selfadjoint .
\end{thm}

This theorem is a particular case of Theorem \ref{t:aIBCm-d}.

\begin{ex} \label{ex:sectorialL}
Assume that there exist constants $c_1$, $c_2 \in (0,+\infty)$ and  $\al_1$, $\al_2  \in [-\pi/2,\pi/2]$ such that $0<c_1 \le |\Z (\x)| \le c_2 $ and 
$-\pi/2 \le \al_1 \le \arg \Z (\x) \le \al_2 < \al_1 + \pi$ a.e. on $\pa \Om$.  Then the  operator $\Lc_\Z$ is m-dissipative (this follows from Corollary \ref{c:sectorialL}). The case $\al_1=\al_2=-\pi/2$ corresponds to superconducting materials on the boundary $\pa \Om$ \cite[Section 87]{LL84}. 
\end{ex}


The case where the set $\pOmp:=\{\x \in \pa \Om : \Z(\x) = 0\}$ is of positive surface measure is connected with mixed boundary conditions in the sense that one  mixes the condition $\n \times \E = 0$ on $\pOmp$ (which is essentially a conservative  perfect conductor condition) and the Leontovich condition (\ref{e:GIBC}) on $\pOma = \pa \Om \setminus \pOmp$ (where $\Z (\x)  \neq 0$). In Proposition \ref{p:non-m-dis}, we show that the Leontovich-type operator itself is not m-dissipative whenever $\pOmp$ contains an open subset of the manifold $\pa \Om$. In the case when $\pa \Om$ is smooth and $\Z \in C^1 (\pa \Om)$, \cite{K94} provides a certain extension of $\Lc_\Z$ to an m-dissipative Maxwell operator. 
From the point of view of well-posedness of dynamical and time-harmonic Maxwell systems, mixed boundary conditions were considered in \cite{M03,ACL17} for polyhedral domains $\Om$ and $\pOmp$ with piecewise smooth boundaries (see the discussion in \cite[Sections 1.6.1 and 5.1.2]{ACL17} and also Section \ref{s:discussion}).

If $\Z (\cdot)$ and/or the boundary of the set $\pOmp$ are not good enough (see Example \ref{ex:FF} with a fat fractal set $\pOma$ and Example  \ref{ex:random} with randomized mixed boundary conditions), a question of m-dissipative extensions of $\Lc_\Z$ is involved. Let us note that electromagnetic properties of fractal structures is a question of interest in contemporary applied Physics (e.g. \cite{LM12}), and that randomized Leontovich-type conditions of Example  \ref{ex:random} is a natural approach for modeling of  the leakage of EM-energy into an uncertain surrounding medium $\RR^3 \setminus \Om$ (cf. \cite{MSG17,EK21}). 

We use Theorem \ref{t:m-dMax} to describe of all m-dissipative extensions  in $\LL^2_{\epb,\mub} (\Om)$ of an arbitrary Leontovich-type operator $\Lc_\Z$. The Cayley transform $\Kc_\Z$ of the dissipative in 
$\LLt$ operator $(-\ii) \Sdiv \Mul_\Z \Sdiv$ is defined by 
$
\Kc_\Z := (-\ii\Sdiv \Mul_\Z \Sdiv + \ii I_{\LLT}) (-\ii\Sdiv \Mul_\Z \Sdiv - \ii I_{\LLT})^{-1} .
$
Note that $\Kc_\Z$ is a contractive operator in $\LLt$ (its domain is not necessarily whole $\LLt$).

\begin{thm} \label{t:Kz} 
Let $\Z : \pa \Om \to \ii \overline{\CC}_- $ be measurable.
Then $\wh \M$ is an m-dissipative extension of the Leontovich-type operator $\Lc_\Z$ if and only if $\wh \M$ is a Maxwell operator defined by 
\begin{gather} \label{e:K+IGa3}
(I_{\LLT}+K) \Upi  \pi_\top   \H +  (I_{\LLT}-K) (\Upi^\#)^{-1}  \ga_\top  \E = 0 
\end{gather}
with a certain contraction $K$ on $\LLt$ such that $K$ is an extension of $\Kc_\Z$.
\end{thm}

This theorem is a particular case of Theorem \ref{t:I-Text}.

If $\Lc_\Z$ is essentially m-dissipative, then its closure $\overline{\Lc_\Z}$ is the only m-dissipative extension. 

\begin{ex} \label{ex:De}
Let $\Dev $ be the selfadjoint Laplace-de Rham operator in $\LLt$ (see Section \ref{s:Upi}). Then the Maxwell operator associated with the boundary condition
$(-\Dev) \pi_{\top,2} \H + \ga_{\top,2} \E = 0 $ is essentially m-dissipative. This statement is valid actually for boundary conditions $f(-\Dev) \pi_{\top,2} \H + \ga_{\top,2} \E  = 0$ with an arbitrary Borel function  $f:[0,+\infty) \to \ii \overline{\CC}_-$ taken of the operator $(-\Dev)$ in the standard sense of \cite{AG}
(see Corollary \ref{c:f(De)}). 
\end{ex}

Theorem \ref{t:aIBCm-d} implies that $\Lc_\Z$ is essentially m-dissipative if and only if the boundary operator $ \Sdiv \Mul_\Z \Sdiv $ is essentially m-accretive. If $\Z$ is `too singular' it is difficult to check essential m-dissipativity of $\Lc_\Z$ and $(-\ii) \Sdiv \Mul_\Z \Sdiv $. Even if $\Z$ is nonnegative, the selfadjoint operators $\Sdiv$ and $\Mul_\Z$ do not generally commute and their product may have very peculiar properties. In such cases, it makes sense to consider specific types of m-dissipative extension of $\Lc_\Z$.

In particular, we introduce in Section \ref{s:discussion} for arbitrary nonnegative impedance coefficients $\Z$ special \emph{F-extensions} $\Lc_{\Z,\Fr}$ and \emph{K-extensions} $\Lc_{\Z,\KN}$
of Leontovich-type operators $\Lc_\Z$ built with the help of the Friedrichs extension $[\Sdiv \Mul_\Z \Sdiv]_{\Fr}$ and the Krein-von Neumann extension $[\Sdiv \Mul_\Z \Sdiv]_{\KN}$ of the boundary operator $\Sdiv \Mul_\Z \Sdiv$. 
The corresponding  m-dissipative Maxwell operators $\Lc_{\Z,\Fr(\KN)}$ are defined by the boundary condition (\ref{e:K+IGa3})
with $K$ equal, roughly speaking,  to  the Cayley transform of $(-\ii)[\Sdiv\Mul_{\Z} \Sdiv]_{\Fr(\KN)}$. 
However,  this procedure works so directly only if the operator $\Sdiv\Mul_{\Z} \Sdiv$ is densely defined. Otherwise, 
$[\Sdiv\Mul_{\Z} \Sdiv]_\Fr$ and, possibly, $[\Sdiv\Mul_{\Z} \Sdiv]_\KN$  are not operators, but selfadjoint linear relations in $\LLt$ 
(it is difficult to exclude this possibility if $\Z$ has no good regularity properties, in particular, in Examples \ref{ex:FF} and \ref{ex:random}). 

The rigorous construction of  F- and K- extensions $\Lc_{\Z,\Fr(\KN)}$ is given in Section \ref{s:discussion}, where we discuss also other approaches. Note that F- and K- extensions cannot be considered as direct analogues of  Friedrichs and Krein-von Neumann extensions
and that Maxwell operators are, generally, not sectorial.

\begin{rem} \label{r:div=0}
The classical Maxwell system includes two (weighted) divergence-free equations
\begin{gather} \label{e:divEH=0}
\Div (\epb(\x) \E (\x)) = 0 \quad \text{ and } \quad \Div (\mub(\x) \H (\x)) = 0  \qquad \text{ a.e. in $\Om$}, 
\end{gather}
which can be taken into account for an arbitrary m-dissipative extension $\wh \M$ of $\M$ by means of orthogonal decomposition of the operator $\wh \M$.
Since Leontovich-type operators require a  decomposition somewhat different from  the decomposition that  \cite{BS87} used for perfect conductor conditions, we consider this procedure in detail.  Let $\bbxi (\cdot)$ be equal to either $\epb (\cdot)$, or $\mub (\cdot)$, and consider
the `weighted' Hilbert space  $ \LL^2_\bbxi (\Om) := (\LL^2 (\Om), \| \cdot \|_{\LL^2_\bbxi})$ with $\| \uu \|_{\LL^2_\bbxi}^2 = (\bbxi \uu , \uu)_{\LL^2 (\Om)}$. Let us consider the closed subspaces of $\LL^2_\bbxi (\Om)$ of gradient and 'weighted' solenoidal vector fields
\[ \text{$ \GG := \{ \uu \in \LL^2 (\Om) \ : \uu = \grad \vphi , \ \vphi \in H^1_\loc (\Om) \}$, \quad
$ \SS^{\bbxi} := \{ \uu \in \LL^2 (\Om) \ : \ \Div (\bbxi \uu) = 0 \} $. }
\]
The space $\GG_0^\bbxi $ is defined as  the closure of $\grad C_0^\infty (\Om)$ in $(\GG , \| \cdot \|_{\LL^2_\bbxi})$.
The orthogonal decomposition \cite{BS87}
$\LL^2_{\bbxi}  (\Om) = \SS^\bbxi   \oplus \GG_0^\bbxi$
is a weighted generalization of one of the versions of the Helmholtz-Weyl decomposition. Then 
$\LL^2_{\epb,\mub} (\Om) =\SS^{\epb,\mub}  \oplus \GG_0^{\epb,\mub}$ with
$\GG_0^{\epb,\mub} := \left\{ \{\E,\H\} \ : \ \E \in \GG_0^\epb ,\  \H \in
\GG_0^\mub \right\} $ and  
$\SS^{\epb,\mub} := \left\{ \{\E,\H\} \ : \ \E \in \SS^\epb ,\  \H \in \SS^\mub \right\}$. 
Recall that, if $\Hc = \Hc_1 \oplus \Hc_2$, where $\Hc_1$ and $\Hc_2$ are reducing subspaces of an operator $T$ \cite{AG}, one says that the decomposition $\Hc = \Hc_1 \oplus \Hc_2$ reduces $T$ and writes $T = T  \ver_{\Hc_1} \oplus T  \ver_{\Hc_2}$,
where the part 
$T_j = T \ver_{\Hc_j}$ of $T$ in $\Hc_j$ is the restriction $T \uph_{\Hc_j \cap \dom T}$ perceived as an operator in $\Hc_j$.
Let $\wt \M $ be an arbitrary dissipative extension of $\M$. 
 Then the Maxwell operators $\wt \M$, $\M$, and $\M^*$ are  reduced by $\LL^2_{\epb,\mub} (\Om) =\SS^{\epb,\mub}  \oplus \GG_0^{\epb,\mub}$ to 
 \begin{gather}
 \text{$\wt \M = \wt \M \ver_{\SS^{\epb,\mub}} \oplus 0$, \quad $\M = \M \ver_{\SS^{\epb,\mub}} \oplus 0$, \quad and $\M^* = (\M \ver_{\SS^{\epb,\mub}})^* \oplus 0$, respectively; } \label{e:whMoplus}
 \\
\text{if $\wt \M$ is m-dissipative, then its part $\wt \M \ver_{\SS^{\epb,\mub}}$ is an m-dissipative operator in $\SS^{\epb,\mub}$.} \label{e:whMS}
\end{gather}
Since, by definition, $\SS^\bbxi$ coincides with $\HH (\Div \bbxi 0,\Om) = \{ \uu \in \LL^2 (\Om) \ : \ \Div (\bbxi \uu) = 0  \}$ as a linear space, we see that in the case of (\ref{e:whMS}) the operator $\wt \M \ver_{\SS^{\epb,\mub}}$ is an m-dissipative operator corresponding to the Maxwell system equipped with the divergence-free conditions (\ref{e:divEH=0}). 
\end{rem}

The proof of (\ref{e:whMoplus})-(\ref{e:whMS}) follows from $\GG_0^{\epb,\mub} \subseteq \ker \M$. Indeed, since $\M$ is a densely defined symmetric operator, one sees that 
$\GG_0^{\epb,\mub}$ and $\ker \M$ are reducing subspaces for $\M$ and  
$\M^*$. This implies for $\M$ and $\M^*$ the decompositions 
(\ref{e:whMoplus}).  Since $\wt M$ is a dissipative extension of $\M$, one has from \cite{K77} that $\Gr \M \subset \Gr \wt \M \subset \Gr \M^*$. This and the decomposition for $\M^*$ imply  $\wt \M = \wt \M \ver_{\SS^{\epb,\mub}} \oplus 0$. If $\wt M$ is m-dissipative, one sees from Definition \ref{d:dis} that $\wt \M \ver_{\SS^{\epb,\mub}} $ is so.

\subsection{Boundary tuples and operator-theoretic tools of the paper}
\label{s:OT}

The PDE nature of the problem of finding of all m-dissipative boundary conditions for Maxwell operators can be separated from its operator theoretic features in several steps.

In the 1st step (see Sections \ref{s:MDAbstract} and \ref{s:Upi}), we put the duality of the trace spaces $\HH^{-1/2} (\curl_{\pa \Om})$ and $\HH^{-1/2} (\Div_{\pa \Om})$ w.r.t. the pivot space $\LLt$ into the abstract framework of  \cite[Appendix to IX.4, Example 3]{RSII}, where a generalization of the notion of rigged Hilbert space is considered from the interpolation point of view. 

We call this generalized type of duality \emph{mixed-order norm duality} (in short, \emph{m-order duality}) to distinguish it from the standard duality of rigged Hilbert spaces.
A standard rigged Hilbert space with the continuous imbeddings $\Hs_+ \imb \Hs \imb \Hs_- $   involves the Hilbert spaces $\Hs_\pm$, which are called the spaces with positive- and negative-order norms, see e.g. \cite{GG91}. 
It is clear from the Hodge decompositions (\ref{e:HDL22})-(\ref{e:HDdiv}) obtained in \cite{C96,BCS02,BHPS03}   that $\HH^{-1/2} (\curl_{\pa \Om})$, $\LLt$, and $\HH^{-1/2} (\Div_{\pa \Om})$ are not ordered by such imbeddings, since, roughly speaking, one part of each of the trace spaces has a norm of positive order, while the other part  a norm of negative order. That is why 
we denote a pair of abstract Hilbert spaces connected by a mixed-order norm duality by $\Hs_{-,+}$ and $\Hs_{+,-}$ and call them mixed-order spaces (in short, \emph{m-order spaces}, see Section \ref{s:MDAbstract}). By 
$\wt \Hs_{\mp,\pm}:= \Hs_{\mp,\pm} \cap \Hs$, their intersections with the pivot space $\Hs$ are denoted. 

We define the $\Hs$-pairing  as  the \emph{duality pairing $\<\cdot|\star\>_{\Hs}$} of the spaces $\Hs_{\mp,\pm}$ that is constructed, roughly speaking, as an extension of the inner product $(\cdot|\star)_{\Hs}$ of $\Hs$  (see  Proposition \ref{p:MixOrder}).
We systematically use $\Hs$-pairing adjoint operators $T^\#$, see  (\ref{e:Am}) and Theorems \ref{t:m-dMax}--\ref{t:Upi}.

In Section \ref{s:MDAbstract}, m-order spaces and their duality is considered from the extrapolation point of view. This process produces various auxiliary operators, which are used  through the rest of the paper as technical tools.
The most important of these auxiliary operators are the homeomorphisms $\Upi$ and $\Upi^\#$ of Theorem \ref{t:Upi} and their abstract versions $\Uop_{\Hs_{-,+} \to \Hs}$ and $\Uop_{\Hs \to \Hs_{+,-}}$, see Proposition \ref{p:MixOrder} and Corollary \ref{c:Upiga}.

In Section \ref{s:Upi}, we construct Riesz bases in the trace spaces $\HH^{-1/2} (\curl_{\pa \Om})$ and $\HH^{-1/2} (\Div_{\pa \Om})$ and introduce in these spaces associated Hilbertian norms $ \| \cdot \|_\pi$ and $ \| \cdot \|_\ga$, which are equivalent to the norms $|\cdot|_{\pi}$ and $|\cdot|_\ga $ of Section \ref{s:Mresults}, but make the homeomorphisms $\Upi$ and $\Upi^\#$ unitary operators and make the Hodge decompositions orthogonal.

In the 2nd step, the integration by parts for the sesquilinear form 
$(\M^* \uu|\vv)_{\LL^2_{\epb,\mub}}$ 
associated with the Maxwell operator $\M^*$
is placed into abstract settings. An abstract approach to the description of boundary value conditions for elliptic partial differential operators (PDOs) traces its origin to works of Calkin, M.~Krein, Birman, Vishik, and Grubb (see the monographs \cite{GG91,G09,P12,DM17,BHdS20}). It is one of the points of the present paper that Calkin's reduction operators (see \cite{C39,HW12} and Section \ref{s:curl}) are especially well suited for writing of an abstract version for Leontovich-type operators, see Section \ref{s:GIBCa} and Corollary \ref{c:curlRT}. (Note that we unite a Calkin reduction operator $G$, its target space $\Hs$, and an associated rotation $W$ into a triple $(G,\Hs,W)$ and call it a \emph{Calkin triple}.) 

For a general description of boundary conditions for PDOs and, in particular, for Maxwell operators, the use of Calkin's reduction operators  based on natural trace maps is difficult  because they usually are not surjective 
onto their target spaces \cite{C39,HW12}. A general description of m-dissipative/selfadjoint boundary conditions require other operator-theoretic notions involving surjective boundary maps.
A powerful and well-developed abstract tool of such type is the notion of boundary triple (or boundary value space), which was introduced by 
Talyush, Kochubei, and Bruk  
for an operator $\A^*$ adjoint to a densely defined symmetric operator $\A$ with equal deficiency indices \cite{K75,GG91} (see also the review \cite{DM95} and the monographs \cite{P12,DM17,BHdS20}). 
Namely, $(\Hs, \wh \Ga_0, \wh \Ga_1)$ is called a \emph{boundary triple for} $\A^*$ if an auxiliary Hilbert space $\Hs$ and 
the maps $\wh \Ga_j: \dom \A^* \to \Hs$, $j=0,1$, are such that 
\begin{multline} \label{e:BTr}
\text{$\wh \Ga : f \mapsto \{ \wh \Ga_0 f ,\wh \Ga_1 f \} $  is a surjective linear operator from $\dom \A^*$ onto $\Hs \oplus \Hs$} \\
\text{and $
(\A^*f|g)_\Hc - (f|\A^*g)_\Hc =  (\wh \Ga_1 f | \wh \Ga_0 g )_\Hs  - (\wh \Ga_0 f | \wh \Ga_1 g )_\Hs $ for all $f,g \in \dom \A^*$.   }
\end{multline}

As soon as a boundary triple is constructed for a differential operator, this  leads to a description of all selfadjoint/m-dissipative extensions via the result of Kochubei \cite{K75}. 
The substantial difficulty of this abstract approach is that boundary triples are often not well adjusted to natural trace maps of the PDO. We refer to \cite{G09,AGW14} and references therein for the review on general linear boundary value problems for  even-order PDOs in domains with smooth boundaries $\pa \Om$ and the use of the techniques of differential and pseudo-differential boundary operators. Substantial efforts aimed on the connection of the PDE and operator-theoretic approaches led to a number of modifications of the notion of boundary triple,
see \cite[Definition 6.1 and Proposition 6.3]{DM95}, \cite{A00}, \cite[Sections 7.4 and 7.6]{DHMdS12}, \cite[Sections 1.2.9 and 3.4]{P12}, \cite[Section 2]{BM14}, and references therein. Up to our understanding, these generalizations of boundary triples are not well suited for the particular case of Maxwell operators. 

We introduce one more modification of the notion of boundary triple, which uses mixed-order duality and is natural from the point of view of the integration by parts for $(\M^* \uu|\vv)_{\LL^2_{\epb,\mub}}$.

\begin{defn} \label{d:MBT}
Let $\A$ be a closed densely defined symmetric operator in a certain Hilbert space $\Hc$.
Let $\Hs_{\mp,\pm}$ be m-order Hilbert spaces dual  to each other w.r.t. a pivot Hilbert space $\Hs$ in the sense of Section \ref{s:MDAbstract}.
 We shall say that $(\Hs_{-,+},\Hs, \Ga_0,\Ga_1)$ is a \emph{mixed-order boundary tuple} (in short, \emph{m-boundary tuple}) for the operator $\A^*$
if the following conditions hold:
\item[(M1)] the map 
$\Ga : f \mapsto \{ \Ga_0 f , \Ga_1 f \} $ is a surjective linear operator from $\dom \A^*$ onto $\Hs_{-,+} \oplus \Hs_{+,-}$; 
\item[(M2)] $
(\A^*f|g)_{\Hc} - (f|\A^*g)_{\Hc} =  \< \Ga_1 f | \Ga_0 g \>_\Hs  - \<\Ga_0 f | \Ga_1 g \>_\Hs $ for all $f,g \in \dom \A^*$.
\end{defn}

\begin{rem}
A  more complete notation for m-boundary tuples is $(\Hs_{-,+},\Hs,\Hs_{+,-},\Gb, W)$ 
because $\Hs_{+,-}$ also participates in the definition.
However, we skip $\Hs_{+,-}$ since its m-order duality to $\Hs_{-,+}$ w.r.t. $\Hs$ defines it essentially  uniquely (uniquely up to the identification of sequences $(u_k)_{k=1}^{\infty} \subset \wt \Hs_{+,-}$ fundamental w.r.t. $\|\cdot\|_{\Hs_{+,-}}$, see Section \ref{s:MDAbstract}).
\end{rem}

\begin{rem} 
A notion of m-boundary tuple is generalization a of the notion of boundary triple. Indeed,  
with every boundary triple $(\Hs,\wh \Ga_0,\wh \Ga_1)$, one can associate a ´trivial´ m-boundary tuple $(\Hs,\Hs,\wh \Ga_0,\wh \Ga_1)$ employing the trivial duality 
$(\Hs_{\mp,\pm}, \| \cdot \|_{\Hs_{\mp,\pm}}) = (\Hs, \| \cdot \|_\Hs)$. On the other hand,
every m-boundary tuple can be regularized to produce a boundary triple, see Proposition \ref{p:BT}. These regularizations are equivalent to a choice of a particular pair of biorthogonal Riesz bases in $\Hs_{\mp,\pm}$. Roughly speaking, m-boundary tuple can be seen as a `coordinate-free' replacement of regularized boundary triples. 
\end{rem}

We show (see Theorem \ref{t:MaxwellBT}) that 
$(\HH^{-1/2} (\curl_{\pa \Om}),\LLt , \pi_\top^{\hf}, \   \ii \ga_\top^\ef )$ with $\pi_\top^{\hf} \{\E,\H\} := \pi_\top \H$ and $\ga_\top^{\ef} \{\E,\H\} := \ga_\top \E$ is an m-boundary tuple for the Maxwell operator $\M^*$.
This m-boundary tuple occurs to be especially convenient for the study of the generalized impedance boundary conditions of \cite[Section 1.6.1]{ACL17}. In particular, Corollary \ref{c:Z} follows from Theorem \ref{t:MaxwellBT}.

It occurs that Leontovich-type boundary conditions (\ref{e:GIBC}) are written in terms of two mutually dual Calkin's reduction operators for the operator $\curlm$, namely, in terms of $\pi_{\top,2}$ and $\ga_{\top,2}$. For the study of associated Leontovich-Maxwell operators, we introduce in Section \ref{s:mrt}  an abstract `$S$-weighted'  Maxwell operators 
$
M  \psi 
:=   
S^{-1} 
\begin{pmatrix}
 0 & \ii A \\  -\ii A & 0 
\end{pmatrix} 
$,
where $A$ is an abstract densely defined symmetric operator, and introduce in Section \ref{s:RT} the notion of a reduction tuple for $A^*$.
A reduction tuple is, roughly speaking, a surjective completion of Calkin's triple. We construct in Proposition \ref{p:M*MBT}
an m-boundary tuple for an abstract Maxwell operator $M^*$ using two mutually dual 
reduction tuples for $A^*$. In Section \ref{s:GIBCa}, this lead to to an abstract version of Leontovich-type boundary conditions and to  complete characterizations of the corresponding m-dissipative and essentially m-dissipative cases in Theorems \ref{t:CritGIMC} and \ref{t:aIBCm-d}.

With minor modifications our description of m-dissipative boundary conditions for abstract Maxwell operators (Theorem \ref{t:absM-dis}) is applicable to other types of wave equations. In this connection, let us note that abstract Maxwell operator $M$ and its adjoint $M^*$ are closely related to abstract Dirac-type operators considered in \cite{GGHT12}.

\section{Trace spaces for $\HH (\curl,\Om)$ and Calkin's triples}\label{s:Mo} 
 
In this section we collect and adapt to our needs the facts concerning trace spaces 
for the space $\HH (\curl,\Om)$ and the notions connected with Calkin's reduction operator.
The Lipschitz boundary $\pa \Om$ of $\Om$ is a 2-D closed surface with bi-Lipschitz  `differentiable-type' structure \cite{ACL17}. 
We use  for $s > 0$ the standard 
Hilbertian Sobolev spaces $\HH^s ( \Om) = W^{s,2} ( \Om, \CC^3)$ and $\HH_0^s ( \Om) = W^{s,2}_0 ( \Om, \CC^3)$, where $\HH_0^s ( \Om)$ is the closure in $\HH^s ( \Om)$ of the subspace $C_0^\infty (\Om,\CC^3)$ of smooth compactly supported in $\Om$ vector-fields. The space of generalized vector-fields $\HH^{-s} ( \Om) $ is dual to $\HH_0^s ( \Om) $ w.r.t. the pivot space $\LL^2 (\Om) = L^2  (\Om, \CC^3)$. 
The analogous spaces of scalar $\CC$-valued functions are denoted  for $s>0$ by  $H^s ( \Om) = W^{s,2} ( \Om; \CC)$ and $H^s_0 ( \Om) = W^{s,2}_0 ( \Om; \CC)$.

\subsection{Integration by parts for $\curlm$ and Calkin's reduction operator}
\label{s:curl}

The continuous imbedding $\HH^1 (\Om) \imb \HH (\curl, \Om) $ is dense (for this and the other basic facts listed below we refer to \cite{C96,BCS02,M03,M04,ACL17}). 
The (vector) trace  $\ga (\uu) = \uu\!\uph_{ \pa \Om}$, the tangential trace $\ga_\top (\uu) = - \n  \times \uu\!\uph_{\pa \Om} $, and the tangential component trace $\pi_\top (\uu) = - \n \times (\n  \times \uu\!\uph_{\pa \Om})$ first defined for $C^\infty (\overline{\Om};\CC^3)$-fields on the closure $\overline{\Om}$ of $\Om$ have unique extensions as continuous operators 
\[
\ga: \HH^1 (\Om) \to \HH^{1/2} (\pa \Om),  \quad
\ga_\top : \HH (\curl, \Om) \to \HH^{-1/2} (\pa \Om) , \quad
\pi_\top : \HH (\curl, \Om) \to \HH^{-1/2} (\pa \Om)  ,
\]
where we use the complex Hilbertian Sobolev spaces $\HH^{s} (\pa \Om) = W^{s,2} ( \pa \Om, \CC^3)$, $s \in (0,1]$, of vector-fields,  and their dual spaces of generalized vector-fields  $\HH^{-s} (\pa \Om)$. The duality is taken w.r.t.
the pivot space $ \LL^2 (\pa \Om)$.
Analogously, $H^s (\pa \Om) = W^{s,2} (\pa \Om; \CC)$, $s \in (0,1]$, are Sobolev spaces of scalar functions on $\pa \Om$, and $H^{-s} (\pa \Om) = W^{-s,2} (\pa \Om; \CC)$ are their dual spaces w.r.t. $L^2 (\pa \Om)$.
The scalar trace 
$\ga_0 (f) = f\!\uph_{ \pa \Om}$ is a continuous operator from 
$H^1 (\Om)$ to $H^{1/2} (\pa \Om)$.

For the boundary spaces $H^{s} (\pa \Om) $ and $\HH^{s} (\pa \Om) $ with $s>0$ and so for their duals, there exist many mutually equivalent  Hilbertian norms generated by various Lipschitz local coordinates. 
Through the paper, we assume that a certain choice of this family of norms is fixed in a consistent way. 

Consider the bounded selfadjoint operator $\uu (\cdot) \mapsto \ii \, \n (\cdot) \times \uu (\cdot)$ in $\LL^2 (\pa \Om)$.  The square of this operator is the orthogonal projection $ P_{\LLT} $ onto 
$ \LLt$. 
The closed subspace $\LLt \subset \LL^2 (\pa \Om)$ is invariant for the operator $\uu (\cdot) \mapsto \n (\cdot) \times \uu (\cdot)$. We denote the restriction of this operator to $\LLt $ by $\n_\times $.
So the operator $\n_\times : \LLt \to \LLt$ defined by $\n_\times \uu = \n  \times \uu $ is a unitary in the Hilbert space $\LLt $ and $(\n_\times)^* = - \n_\times $.

Following \cite{BCS02,M04} (with minor adjustments to the notation of  \cite{M03,ACL17}), let us consider alternative descriptions of the \emph{trace spaces} $\pi_\top \HH (\curl,\Om) := \{ \pi_\top (\uu) \ : \uu \in \HH (\curl,\Om) \}$ and $\ga_\top \HH (\curl,\Om) := \{ \ga_\top (\uu) \ : \uu \in \HH (\curl,\Om) \}$.

The space $H^{3/2} (\pa \Om) := \ga_0 H^2 (\Om)$ is a Hilbert space with the norm 
\[
\| f \|_{H^{3/2} (\pa \Om)} := \inf \{ \| g \|_{H^2 (\Om)} \ : \ f = \ga_0 (g) , \ g \in H^2 (\Om)\}
,\]
 and the space $H^{-3/2} (\pa \Om) $ is the dual of $H^{3/2} (\pa \Om) $ w.r.t. $L^2  (\pa \Om)$.
Equipping the sets 
$\VV_\ga :=  \ga_\top \HH^1 (\Om)$ and $\VV_\pi :=  \pi_\top \HH^1 (\Om)$
 with the norms 
$
\| \vv \|_{\VV_\ga (\VV_\pi)} 
:= \inf \{ \| \uu \|_{\HH^1 (\Om)} \ : \ \uu \in \HH^1 (\Om),  \ga_\top (\uu) = \vv  \text{  (resp., $\pi_\top (\uu) = \vv$) } \}, 
$
one obtains Hilbert spaces. Since $\HH^{1/2} (\pa \Om)$ is dense in $\LL^2 (\pa \Om)$ and $\n_\times \VV_\pi = \VV_\ga $, the spaces $\VV_\ga$ and $\VV_\pi$ are densely and continuously embedded in $\LLt $. 
Let  $\VV'_{\pi (\ga)}$ be the dual Hilbert space to $\VV_{\pi (\ga)}$ w.r.t. $\LLt $.  Note that the spaces $\VV_\pi$ and $\VV_\ga$ are generally different \cite{BCS02} (but in the particular case of smooth $\pa \Om$ they both can be naturally identified to the space of tangential vector fields of order 1/2). 

\begin{rem} \label{r:pairing}
The (sesquilinear) duality pairings $_{\Hs_{\pm}}\< \cdot|\star\>_{\Hs_{\mp}}$ of the rigged Hilbert space $\Hs_+ \imb  \Hs \imb \Hs_-$ will be denoted $\<\cdot|\star\>_\Hs $ independently of a choice of mutually dual spaces $\Hs_\pm$, e.g., the pairing of  $H^{\pm 3/2} (\pa \Om) $ with $H^{\mp 3/2} (\pa \Om) $ is denoted by $\<\cdot|\star\>_{L^2  (\pa \Om) } $, 
the pairing between $\VV_{\pi (\ga)}$ and $\VV'_{\pi (\ga)}$ by  $\<\cdot|\star\>_{\LLT} $. 
 Note that  $\<f|g\>_\Hs $ equals to the scalar product $ (f|g)_\Hs$ if $f,g \in \Hs$.
\end{rem}

The \emph{tangential gradient} $\grad_{\pa \Om}$ and the operator $\curlm_{\pa \Om} $ of \emph{surface vector-curl} can be defined \cite{C96,BCS02,M04} as continuous linear operators from $H^{1/2} (\pa \Om)$ to $\VV'_\ga$ and from $H^{1/2} (\pa \Om)$ to $\VV'_\pi$, respectively, by 
$\grad_{\pa \Om} \uu = \pi_\top (\grad \mathbf{U} ) $ and $ \quad \curlm_{\pa \Om} \uu =  \ga_\top (\grad \mathbf{U} ) $, 
where  $\mathbf{U} \in H^1 (\Om)$ is arbitrary and $\uu = \mathbf{U} \uph_{\pa \Om}$.  
The restrictions of these operators to $H^1 (\pa \Om)$ can be written in local coordinates in the same way \cite{BCS02} as it is done for the case of smooth $\pa \Om$ \cite{C96,M03} (the difference with the smooth case is that matrix-function representations of the metric in local coordinates and corresponding inverses have only $L^\infty$-regularity).
Besides, $\grad_{\pa \Om} \uph_{H^1 (\pa \Om)}$ and  $\curlm_{\pa \Om} \uph_{H^1 (\pa \Om)}  f = - \n \times   \grad_{\pa \Om} \uph_{H^1 (\pa \Om)} f $ are continuous operators from $H^1 (\pa \Om)$ to $\LLt$. Moreover, the norm 
$(\|f \|_{L^2(\pa \Om)}^2 + \| \grad f \|_{\LLT}^2)^{1/2}$ is equivalent to $\|f \|_{H^1 (\pa \Om)}$.
The restriction $\grad_{\pa \Om} \uph_{H^{3/2} (\pa \Om)}$ (restriction $\curlm_{\pa \Om}  \uph_{H^{3/2} (\pa \Om)}$) can \cite{BCS02} and will be considered  as a continuous operator from
$H^{3/2} (\pa \Om)$ to $\VV_\pi$ (resp., to $\VV_\ga$).

The surface divergence $\Div_{\pa \Om}:\VV'_{\pi} \to H^{-3/2} (\pa \Om) $   and the surface scalar curl-operator $\curl_{\pa \Om}: \VV'_{\ga} \to H^{-3/2} (\pa \Om)$ are continuous operators defined by variational formulae 
\begin{align*}
\<\Div_{\pa \Om} \uu | f \>_{L^2 (\pa \Om)}  = & - \< \uu | \grad_{\pa \Om} f \>_{\LLT} \ \text{ and } \  
\<\curl_{\pa \Om} \uu | f \>_{L^2 (\pa \Om)}  = & \< \uu | \curlm_{\pa \Om} f \>_{\LLT} 
\end{align*} 
valid for all  $f \in H^{3/2} (\pa \Om)$
(with the notation of $\#$-adjoint operators of Section \ref{s:MDAbstract}, one can write 
$\Div_{\pa \Om} := - (\grad_{\pa \Om} \uph_{H^{3/2}(\pa \Om)})^\# $ and $\curl_{\pa \Om} := (\curlm_{\pa \Om} \uph_{H^{3/2} (\pa \Om)})^\#$).

One can define the Laplace-Beltrami operator $\De^{\pa \Om}$  as the unique continuous operator from $H^1 (\pa \Om)$  to $H^{-1} (\pa \Om)$  satisfying 
$
\< -\De^{\pa \Om} f   |  g \>_{L^2 (\pa \Om)} 
= 
\< \grad_{\pa \Om} f |  \grad_{\pa \Om} g \> $ 
 for all $g \in 
H^1 (\pa \Om) $
Then the finite-dimensional space of locally constant functions $\KK_0$ is the null space of $\De^{\pa \Om}$. Moreover, $\De^{\pa \Om}$  has a compact resolvent as an operator in $H^{-1} (\pa \Om)$  with the domain $H^1 (\pa \Om)$ since $H^1 (\pa \Om) \imb \imb H^{-1} (\pa \Om)$, where $ \imb \imb$ denotes a compact imbedding.

The paper  \cite{BCS02} defines the spaces 
$ \HH^{-1/2} (\Div_{\pa \Om})  
  :=   \{ \vv \in \VV'_\pi \ :\ \Div_{\pa \Om} \vv \in H^{-1/2} (\pa \Om)\} $ and 
$
\HH^{-1/2} (\curl_{\pa \Om})  
:=   \{ \vv \in \VV'_\ga \ :\ \curl_{\pa \Om} \vv \in H^{-1/2} (\pa \Om)\} $ 
as Hilbert spaces with the  graph norms. 

The following result says that $ \HH^{-1/2} (\Div(\curl)_{\pa \Om})  $ are the trace spaces for $ \HH (\curlm,\Om)$:  
\begin{multline}
\text{$\ga_\top $ and $\pi_\top$ are continuous surjective operators from $ \HH (\curlm,\Om)$} \\ \text{ onto 
$\HH^{-1/2} (\Div_{\pa \Om})$ and
$\HH^{-1/2} (\curl_{\pa \Om})$, respectively (cf. Section \ref{s:Mresults}).}  \label{e:TrTh}
\end{multline} 
This result is well known for smooth boundaries $\pa \Om$ \cite{C96}.  It was obtained by Buffa, Costabel, and Sheen \cite{BCS02} for Lipschitz domains $\Om$ (under an additional restriction that $\pa \Om$ is connected) from a preceding representation theorem of Tartar. The result of Mitrea \cite[Theorem 3.6]{M04} gives the proof of the case $\ga_\top$ for a general Lipschitz domain. 
The case $\pi_\top$ of (\ref{e:TrTh}) for a general Lipschitz domain follows easily from the case $\ga_\top$ proved in \cite{M04}. 



We consider the restrictions 
$\pi_{\top,2} := \pi_\top \uph_{\Himp (\curlm, \Om)}$   and  $\ga_{\top,2} := \ga_\top \uph_{\Himp (\curlm, \Om)}$ 
as operators from $\HH (\curlm, \Om)$ to $\LLt $, see (\ref{e:Himp}).
Since the trace $\pi_\top$ (the trace $\ga_\top$) is bounded from $\HH (\curlm , \Om)$ to $\VV'_{\ga(\pi)}$, one sees from the continuous imbedding $\LLt \imb \VV'_{\ga(\pi)}$ that  $\pi_{\top,2}$  (resp., $\ga_{\top,2}$) is closed as an operator to $\LLt$. 
The domain $\Himp (\curlm, \Om) $ of these two operators equipped with any of the two graph norms (which are equal to each other) becomes a Hilbert space. It follows from \cite{BBCD97} that 
$\HH^1 (\Om) \imb \Himp (\curlm, \Om)$ densely.

The integration by parts for the operator $\curlm$ \cite{C96,BCS02,M03} is given by the following formula  
\begin{gather} \label{e:IntByPpiga}
(\curlm \uu | \vv )_{\LL^2 (\Om)} - ( \uu | \curlm \vv )_{\LL^2 (\Om)}   = \< \pi_\top (\uu) | \ga_\top (\vv) \>_{\LLT} 
=  \< \ga_\top (\overline{\vv}) | \pi_\top (\overline{\uu} \>_{\LLT} , 
\end{gather}
for $\uu \in \HH^1 (\Om)$, $\vv \in \HH (\curlm, \Om)$,
where 
$\< \cdot | \star \>_{\LLT} $ is understood as a pairing of $\VV_\pi $ with $ \VV'_\pi$ in the first case 
and as a pairing of $\VV'_\ga $ with $\VV_\ga$ in the second.

This gives the following description  \cite{C96} of $\dom \curln$:
\begin{align*} 
\HH_0 (\curlm, \Om) = \{ \uu \in \HH (\curlm, \Om) \ :\ \pi_\top (\uu) = 0 \} = \{ \uu \in \HH (\curlm, \Om) \ :\ \ga_\top (\uu)  = 0 \} .
\end{align*}
Combining this with (\ref{e:TrTh}), one sees that the Hilbertian norms 
\begin{gather} \label{e:NoH-1/2}
\text{$|\uu|_\pi := \| \pi_\top^{-1} \uu \|_{\HH (\curlm, \Om) / \HH_0 (\curlm, \Om)} $ and $|\vv|_\ga := \| \ga_\top^{-1} \vv \|_{\HH (\curlm, \Om) / \HH_0 (\curlm, \Om)}$} 
\end{gather}
are equivalent to the original graph norms of the spaces 
 $\HH^{-1/2} (\curl_{\pa \Om})$ and $\HH^{-1/2} (\Div_{\pa \Om})$, respectively.

We will employ systematically the natural identifications between the Hilbert factor-space $\HH (\curlm, \Om) / \HH_0 (\curlm, \Om)$ and  the graph factor-space $\Gr \curlm / \Gr \curln$.
The operators 
\begin{multline*}
\text{
$\pi_\top^{-1}:\HH^{-1/2} (\curl_{\pa \Om}) \to \HH (\curlm, \Om) / \HH_0 (\curlm, \Om)$}
\\ \text{  and 
$\ga_\top^{-1}: \HH^{-1/2} (\Div_{\pa \Om}) \to \HH (\curlm, \Om) / \HH_0 (\curlm, \Om)$ are homeomorphisms.} 
\end{multline*}
Consequently, the sesquilinear forms 
\cite{BCS02,M04} 
$_\pi\!\<\cdot|\star\>_\ga : \HH^{-1/2} (\curl_{\pa \Om}) \times \HH^{-1/2} (\Div_{\pa \Om}) \to \CC$ and 
$_\ga\!\<\cdot|\star\>_\pi :  \HH^{-1/2} (\Div_{\pa \Om}) \times \HH^{-1/2} (\curl_{\pa \Om}) \to \CC$
defined by 
\begin{gather*}
_\pi\!\<\uu|\vv\>_\ga := (\curlm \pi_\top^{-1} \uu | \ga_\top^{-1} \vv )_{\LL^2 (\Om)}  - ( \pi_\top^{-1} \uu | \curlm \ga_\top^{-1} \vv )_{\LL^2 (\Om)}, \label{e:pi<>gaMax1} \\
_\ga\!\<\uu|\vv\>_\pi :=  -  ( \curlm \ga_\top^{-1} \uu | \pi_\top^{-1} \vv )_{\LL^2 (\Om)} + ( \ga_\top^{-1} \uu | \curlm \pi_\top^{-1} \vv )_{\LL^2 (\Om)}
\label{e:pi<>gaMax2}
\end{gather*}
are bounded.
Then the following integration by parts holds automatically
\begin{gather} \label{e:IntByPpiga2}
(\curlm \uu | \vv )_{\LL^2 (\Om)} - ( \uu | \curlm \vv )_{\LL^2 (\Om)}   =_\pi\!\< \pi_\top (\uu) | \ga_\top (\vv) \>_\ga 
= - _\ga\!\< \ga_\top (\uu) | \pi_\top (\vv) \>_\pi 
\end{gather}
for all $ \uu,\vv \in  \HH (\curlm, \Om)$. 
Comparing (\ref{e:IntByPpiga}) with (\ref{e:IntByPpiga2}) one sees
that  
$_{\pi(\ga)}\< \uu | \vv \>_{\ga(\pi)} = ( \uu | \vv )_{\LLT}  $ for all 
$\uu \in \VV_{\pi(\ga)}$, $ \vv  \in \VV_{\ga(\pi)}$. Since $\HH^1 (\Om) \imb \HH (\curlm,\Om)$ densely \cite{C96}, one sees that $\VV_\pi \imb \HH^{-1/2} (\curl_{\pa \Om})$ densely and $\VV_{\ga} \imb \HH^{-1/2} (\Div_{\pa \Om}) $ densely.
This proves the following statement (which, in the case of a connected $\pa \Om$ of topological genus 0, is contained implicitly in \cite{BCS02}).

\begin{prop}[cf. \cite{BCS02}]  \label{p:IntPart} The sesquilinear form $_{\pi}\< \cdot | \star \>_{\ga}$ is a unique bounded extension  to $\HH^{-1/2} (\curl_{\pa \Om}) \times \HH^{-1/2} (\Div_{\pa \Om})$ of the restriction $ ( \cdot | \star )_{\LLT} \uph_{\VV_{\pi} \times \VV_{\ga}} $ (an analogous statement is valid for the sesquilinear form $(-1) _{\ga}\< \cdot | \star \>_{\pi}$).
\end{prop}

In Section \ref{s:RT},  we extend the simplified notation  $\<\cdot| \star\>_{\LLT}$ of Remark \ref{r:pairing} to the both pairings  $_\ga\!\<\cdot|\star \>_\pi $ and $_\pi\!\<\cdot|\star\>_\ga $ of the tangential trace spaces  (for the justification see Remark \ref{r:<>H}).




Let $A$ be a  symmetric closed densely defined operator in an abstract  Hilbert space $\Hc$.

\begin{defn}[\cite{C39}] \label{d:CT}
Assume that an auxiliary Hilbert space $\Hs$, a closed operator $\Gb':\dom \Gb' \subseteq \Gr A^* \to \Hs$, and a unitary operator $W:\Hs \to \Hs$ satisfy the following conditions:
\item[(C1)] 
$\dom \Gb'$ is dense in the Hilbert space $\Gr A^*$ (equipped with the graph norm $\| \cdot \|_{\Gr A^*}$), 
\item[(C2)] 
the validity of 
\begin{gather*} 
(A^* f |g)_{\Hc} - (f,g')_{\Hc} = ( \Gb' \{f,A^*f\}| h )_{\Hs}  \quad \forall \{f,A^*f \} \in \dom \Gb' 
\end{gather*}
for a certain $\{g,g',h\} \in \Hc \times \Hc \times \Hs$ is equivalent to  
$\{g,g'\} \in \dom \Gb'$ and $h = W \Gb' \{g,g'\} $.

Then $\Gb'$ is called a \emph{reduction operator for} $A^*$, $\Hs$ is its \emph{target space}, $W$ is called a \emph{rotation}.
\end{defn}

This is an equivalent reformulation of the original Calkin definition  \cite{C39} of reduction operator. We refer to \cite{HW12,DM17} for the  contemporary point of view on this theory.

\begin{rem} \label{r:Ga}
Following \cite{C39},  one can associate with $\Gb'$ the operator $\Gb f := \Gb' \{f,A^*f\} $
defined on the set of all $f \in \dom A^*$ such that $\{f,A^*f\} \in \dom \Gb'$. This leads to a shorter notation, which we will systematically use for various operators with roles similar to that of $\Gb'$ in (C2).
In this notation the operator $\Gb:\dom \Gb \subseteq \dom A^* \to \Hs$ is closed and 
densely defined in the Hilbert space $\dom A^*$ equipped with the graph norm $\| \cdot\|_{\dom A^*}$. Condition (C2) takes the following form: the validity of
$(A^* f |g)_{\Hc} - (f|g')_{\Hc} = ( \Gb f| h )_{\Hs}$  \quad $\forall f \in \dom \Gb $
for a certain $\{g,g',h\} \in \Hc \times \Hc \times \Hs$ is equivalent to  
$g \in \dom \Gb$, $g' = A^* g$, and $h = W \Gb g$.
\end{rem}

Taking into account the aforementioned identification of $\Gb'$ and $\Gb$,
each of the triples $(\Hs,\Gb,W)$ and $(\Hs,\Gb',W)$ we will be called a \emph{Calkin triple for $A^*$}.

The definition of a Calkin triple $(\Hs,\Gb,W)$ implies  \cite{C39}
the following version of `abstract  integration by parts' 
$(A^*f |g)_\Hc - (f |A^*g)_\Hc  =  ( \Gb f | W \Gb g)_{\Hs}$ 
 for all $f,g \in \dom \Gb$,
and in turn, the following properties:
$\ran \Gb$ is dense in $\Hs$,   $\dom A = \ker \Gb \subseteq \dom \Gb$, 
$W^* = - W = W^{-1}$,
and the fact that $(\Hs, U W \Gb, U W U^{-1})$ is a Calkin triple for every unitary operator $U$ in $\Hs$. 

We will say that \emph{the Calkin triples $(\Hs,\Gb,W)$ and 
$(\Hs, \ii W \Gb,W)$ are dual to each other}.

Section \ref{s:GIBCa} shows that Leontovich-type boundary conditions are written in terms of two mutually dual Calkin triples constructed in Corollary \ref{c:curlRT} for the operator $\curlm$.

\subsection{Mixed-norm duality and $\#$-adjoint operators}
\label{s:MDAbstract}

In this subsection we give an abstract version of the duality connecting the trace spaces  $\HH^{-1/2} (\curl_{\pa \Om})$ and 
$ \HH^{-1/2} (\Div_{\pa \Om}) $ (see also Corollary \ref{c:DevDuality}).
(Under additional assumptions that $\pa \Om$ is $C^{1,1}$-boundary  \cite[Proposition 2.3]{C96} or that $\pa \Om$ is a connected Lipschitz boundary of topological genus $0$ \cite{BCS02}, the duality of these trace spaces was described in a specific way via the Hodge decompositions. For arbitrary Lipschitz boundary and in more general manifold settings, another approach was taken by M.~Mitrea \cite{M04}, namely, an analogue of the  space $\HH^{-1/2} (\curl_{\pa \Om})$ was essentially defined as the space of linear functionals $(\HH^{-1/2} ( \Div_{\pa \Om}))'$.)

Let $\Hs_j$, $j=1,2$, be Hilbert spaces.
A sesquilinear form $\af (\cdot , \star)$
defined on $\Hs_1 \times \Hs_2$ is said to be a \emph{perfect pairing} of $\Hs_1$ with $\Hs_2$ if the following conditions are fulfilled:
(i) for each $h_1 \in \Hs_1 \setminus\{0\}$ there exists $h_2 \in \Hs_2$ such that $\af (h_1,h_2) \neq 0$,
(ii) $ \|h_2\|_{\Hs_2} = \sup \{ |\af(h_1,h_2)| \ : \ \|h_1\|_{\Hs_1} \le 1\} $. Note that conditions (i) and (ii)
are equivalent to the existence of a unitary operator $U:\Hs_2 \to \Hs_1$ such that $\af (h_1,h_2) = (h_1|Uh_2)_{\Hs_1}$. If $\af (\cdot , \star)$ is a perfect pairing of $\Hs_1$ with $\Hs_2$, then  $\overline{\af (h_1 , h_2)}=(h_2|U^{-1}h_1)_{\Hs_2}$ is a perfect pairing of $\Hs_2$ with $\Hs_1$, which is called an \emph{adjoint pairing} to $\af (\cdot , \star)$.

Let $\Hs$ be an abstract pivot Hilbert space. Through this subsection we assume that:
\begin{gather}
\text{$\Hs_{-,+}$ is a Hilbert space such that 
$\wt \Hs_{-,+}:=\Hs \cap \Hs_{-,+}$ is dense in both $\Hs$ and $\Hs_{-,+}$,} 
\label{e:tHs0} 
\\
\text{and the (quadratic) form $\af_{-,+} [h]:=(h|h)_{\Hs_{-,+}}$,
$h \in \wt \Hs_{-,+}$, is closed in $\Hs$ } \label{e:tHs2}
\end{gather}
(see \cite{Kato13} for the basic facts concerning sesquilinear and quadratic forms). 

\begin{defn} \label{d:m-n}
If (\ref{e:tHs0})-(\ref{e:tHs2}) are satisfied, we say that  $\Hs_{-,+}$ 
is a \emph{Hilbert space with a mixed-order norm} (in short, \emph{an m-order space}) associated with the pivot space $\Hs$.
\end{defn}

Let us consider the linear subspace $\wt \Hs_{+,-} \subseteq \Hs$ consisting of all $h \in \Hs$ such that 
\begin{gather}\label{e:defH+-}
\| h \|_{\Hs_{+,-}} < +\infty , \quad \text{ where $\| h \|_{\Hs_{+,-}}:= \sup \{ |(g | h)_{\Hs} | \ : \ g \in \wt \Hs_{-,+}  , \ \| g \|_{\Hs_{-,+}}  \le 1 \} $,}\\
\text{and define the \emph{dual space $(\Hs_{+,-}, \| \cdot \|_{ \Hs_{+,-}})$ to
$\Hs_{-,+}$  w.r.t.} $\Hs$ as \hspace{10em}} \notag \\
\text{\hspace{12em}
the completion of the normed space $(\wt \Hs_{+,-}, \| \cdot \|_{\Hs_{+,-}})$. }\label{e:dual+-}
\end{gather}



The next proposition proves that the  dual space $\Hs_{+,-}$ is an m-order  space associated with $\Hs$ and that the dual to $\Hs_{+,-}$ is the original m-order space $\Hs_{-,+}$.


\begin{prop} \label{p:MixOrder} 

(i) There exist a unique positive selfadjoint operator $\Sop$ in $\Hs$
such that 
$\wt \Hs_{\mp,\pm} = \dom \Sop^{\pm 1} = \ran \Sop^{\mp 1}$ and $\| \Sop^{\pm 1} h \|_{\Hs} = \|h\|_{\Hs_{\mp,\pm}}$ for all $h \in \wt \Hs_{\mp,\pm}$.
In particular, $(\Hs_{+,-}, \| \cdot \|_{\Hs_{+,-}})$ is a Hilbert space.


\item[(ii)] The inner product $(\cdot|\star)_\Hs$ of $\Hs$ restricted to $\wt \Hs_{-,+} \times \wt \Hs_{+,-}$ has a unique  extension to a bounded sesquilinear form $_{\Hs_{-,+}}\<\cdot|\star\>_{\Hs_{+,-}}$ on $\Hs_{-,+} \times \Hs_{+,-} $. 
The form $_{\Hs_{-,+}}\<\cdot|\star\>_{\Hs_{+,-}}$ is a perfect pairing  of $\Hs_{-,+}$ with $\Hs_{+,-} $. 
 
\item[(iii)] The quadratic form $\af_{+,-} [\cdot]:= \| \cdot \|_{\Hs_{+,-}}^2$ defined on $\wt \Hs_{+,-}$ is closed in $\Hs$ and the associated sesquilinear form $\af_{+,-} (\cdot|\star)$ is an inner product in $\wt \Hs_{+,-}$. The m-order space dual to $\Hs_{+,-}$ w.r.t.  $\Hs$ can be identified with $\Hs_{-,+}$ (in the sense of the spaces of fundamental sequences in $(\wt \Hs_{-,+} , \| \cdot \|_{\Hs_{-,+}})$). The associated perfect pairing $_{\Hs_{+,-}}\<\cdot|\star\>_{\Hs_{-,+}}$
is adjoint to $_{\Hs_{-,+}}\<\cdot|\star\>_{\Hs_{+,-}}$.

\item[(iv)] There exist a unique unitary operator $\Uop_{\Hs \to \Hs_{\mp,\pm}} $ from $\Hs$ to $\Hs_{\mp,\pm}$ such that $\Uop_{\Hs \to \Hs_{\mp,\pm}} f = \Sop^{\mp 1}  f$ for all $f \in \wt \Hs_{\pm,\mp}$. The inverse unitary operator $\Uop_{\Hs_{\mp,\pm} \to \Hs} := \Uop_{\Hs \to \Hs_{\mp,\pm}}^{-1}$ \linebreak satisfies 
$\Uop_{\Hs_{\mp,\pm} \to \Hs} f = \Sop^{\pm 1} \ f$ for all $f \in \wt \Hs_{\mp,\pm}$.
 The operator $\Uop_{\Hs_{\mp,\pm} \to \Hs_{\pm,\mp}} :=  \Uop_{\Hs \to \Hs_{\pm,\mp}} \Uop_{\Hs_{\mp,\pm} \to \Hs}$ is a unitary operator from $\Hs_{\mp,\pm}$ to $\Hs_{\pm,\mp}$.
\end{prop}

\begin{proof}
First, we apply  the second representation theorem of the theory of Friedrichs' extensions \cite[Theorem VI.2.23]{Kato13} to the closed positive form $\af_{-,+}[\cdot] = (\cdot|\cdot)_{\Hs_{-,+}}$ in $\Hs$. This produces a positive selfadjoint operator $\Sop:\dom \Sop \subseteq \Hs \to \Hs$ such that $\wt \Hs_{-,+} = \dom \Sop$ and $(\Sop^2 f|f)_\Hs = (f|f)_{\Hs_{-,+}}$ for $f \in \dom \Sop^2$. Applying then the spectral decomposition theorem \cite{AG} to $\Sop$
one obtains 
all the desired statements (cf. \cite[Appendix to IX.4, Example 3]{RSII}).
\end{proof}

\begin{rem} \label{r:<>H}
When this does not lead to a confusion,
we use the notation $\<\cdot|\star\>_\Hs$ for each of the pairings $_{\Hs_{\mp,\pm}}\<\cdot|\star\>_{\Hs_{\pm,\mp}}$ constructed in the above statements (ii)-(iii) as well as for the inner product $(\cdot|\star)_\Hs$.
So the map $\<\cdot|\star\>_\Hs: (\Hs_{-,+} \times \Hs_{+,-}) \cup (\Hs_{+,-} \times \Hs_{-,+}) 
\cup \Hs^2 
\to \CC$, which will be called \emph{$\Hs$-pairing}, has the sesquilinear property.
\end{rem}

Proposition \ref{p:MixOrder} justifies the following definition. 

\begin{defn}\label{d:mutualduality}
Two m-order spaces  $\Hs_{-,+}$ and $\Hs_{+,-}$ associated with a pivot Hilbert space $\Hs$ are called \emph{m-order spaces dual to each other w.r.t.} $\Hs$ if 
$\| h \|_{\Hs_{\pm,\mp}} = \sup_{g \in \wt \Hs_{\mp,\pm} \setminus \{0\} }  \frac{|(g | h)_{\Hs}|}{\| g \|_{\Hs_{\mp,\pm}}}$  \ , \ 
where $\wt \Hs_{\mp,\pm} = \Hs \cap \Hs_{\mp,\pm}$.
In this case, we also say that \emph{the norms $\| \cdot\|_{\Hs_{-,+}}$ and $\| \cdot\|_{\Hs_{+,-}}$ are dual to each other w.r.t.} $\Hs$.

\end{defn}

The definition of duality of norms is justified by the following simple observation.
Assume that two spaces $\Hs_{\mp,\pm}$ are m-order spaces dual to each other w.r.t. $\Hs$. Let $\|\cdot\|'_{\Hs_{-,+}}$ be a Hilbertian norm in $\Hs_{-,+} $ equivalent to the original norm $\|\cdot\|_{\Hs_{-,+}}$. Then  $(\Hs_{-,+}, \|\cdot\|'_{\Hs_{-,+}})$ is an m-order space associated with $\Hs$. Moreover, there exists a unique Hilbertian norm $\|\cdot\|'_{\Hs_{+,-}}$ in $\Hs_{+,-} $ equivalent to $\|\cdot\|_{\Hs_{+,-}}$ such that the two spaces $(\Hs_{\mp,\pm}, \|\cdot\|'_{\Hs_{\mp,\pm}})$ are m-order spaces dual to each other w.r.t. $\Hs$.

In the particular case, where $\Hs_{-,+} = \wt \Hs_{-,+}$ (and so $\Hs_{-,+}$ is continuously imbedded into $\Hs$), the imbeddings 
$\Hs_{-,+} \imb \Hs \imb \Hs_{+,-}$ form the standard  rigged Hilbert space. Formally, we do not exclude such a case $\Hs_{-,+} \imb \Hs $ from our abstract considerations,
and so the standard  rigged Hilbert space is a particular case of Definition \ref{d:mutualduality}. 
However, we are interested primarily in the situation where $\Hs_{-,+} \not \imb \Hs$ and $\Hs_{+,-} \not \imb \Hs$, which arises for the trace spaces of $\HH (\curlm,\pa\Om)$ (see Corollary \ref{c:DevDuality}).

Let  two Hilbert spaces $\Hc_{\mp,\pm}$ (two Hilbert spaces $\Yc_{\mp,\pm}$) be m-order spaces dual to each other w.r.t. a pivot space $\Hc$ (resp., w.r.t. a pivot space $\Yc$). Let  
$T:\dom T \subseteq \Hc_{-,+}  \to \Yc_{+,-}$ be a densely defined in $\Hc_{-,+}$ operator. Then there exists a unique operator $T^\#: \dom T^\# \subseteq \Yc_{-,+}  \to \Hc_{+,-} $ defined by  the equivalence (the notation of Remark \ref{r:<>H} is used)
\begin{multline} \label{e:Am}
\text{$g \in \dom T^\#$ and $g' = T^\# g$ if and only if } \\ 
\{g,g'\} \in \Yc_{-,+} \times \Hc_{+,-} \text{ satisfies }\< T f|g\>_\Yc = \<  f|g'\>_\Hc \quad \text{ for all } f \in \dom T .
\end{multline}

Note that, if $T$ is closable, then $(T^\#)^\#$ is the closure $\overline T$ of $T$.

We  use $T^\#$  in several different situations where
the two triples of spaces $(\Hc_{-,+},\Hc,\Hc_{+,-})$ and $(\Yc_{-,+},\Yc,\Yc_{+,-})$ are connected with each other. Namely, we use $T^\#$ in the following cases: 
\begin{itemize}
\item[(A1)] $(\Hc_{-,+},\Hc,\Hc_{+,-}) = (\Yc_{-,+},\Yc,\Yc_{+,-}) = (\Hs_{-,+},\Hs,\Hs_{+,-})$; 
\item[(A2)] $(\Hc_{-,+},\Hc,\Hc_{+,-}) = (\Hs,\Hs,\Hs)$ and $(\Yc_{-,+},\Yc,\Yc_{+,-}) = (\Hs_{+,-},\Hs,\Hs_{-,+})$;
\item[(A3)] $(\Hc_{-,+},\Hc,\Hc_{+,-}) = (\Hs_{-,+},\Hs,\Hs_{+,-})$ and $(\Yc_{-,+},\Yc,\Yc_{+,-}) = (\Hs,\Hs,\Hs)$;
\item[(A4)] $(\Hc_{-,+},\Hc,\Hc_{+,-}) = (\Hs_{-,+},\Hs,\Hs_{+,-})$ and $(\Yc_{-,+},\Yc,\Yc_{+,-}) = (\Hs_{+,-},\Hs,\Hs_{-,+})$.
\end{itemize}

In the cases (A1)-(A4),  $T^\#$ is called \emph{$\Hs$-pairing-adjoint operator to} $T$. In the case (A1), an operator $T$ is called \emph{$\Hs$-pairing-selfadjoint} if $T = T^\#$.  (In the `trivial' case $(\Hc_{-,+},\Hc,\Hc_{+,-}) = (\Hs,\Hs,\Hs) = (\Yc_{-,+},\Yc,\Yc_{+,-}) $, one has $T^\# = T^*$.)

The $\Hs$-pairing-adjoint operator $T^\#$ is invariant under a replacement of the norms of $\Hs_{\mp,\pm}$ with a pair of equivalent mutually dual norms $\|\cdot \|'_{\Hs_{\mp,\pm}}$.
In the case (A1),  $T^\#$  is an operator from $\Hs_{-,+}$ to $\Hs_{+,-}$ (similarly to $T$), while the standard adjoint operator $T^*$  acts conversely from $\Hs_{+,-}$ to $\Hs_{-,+}$.


\section{Riesz bases in trace spaces of $\HH (\curlm,\Om)$}
\label{s:Upi}


Recall that $\De^{\pa \Om}:H^1 (\pa \Om) \to H^{-1} (\pa \Om)$ is the continuous Laplace-Beltrami operator with a compact resolvent
and that $\ker \De^{\pa \Om} = \KK_0$ is a finite-dimensional space.

Let us introduce the Hilbert factor-spaces 
$H^0_{\pa \Om}:= L^2 (\pa \Om) / \KK_0$ and $H^s_{\pa \Om}:= H^s (\pa \Om) / \KK_0$, $s \in (0,1]$,
and the Hilbert spaces $H^{-s}_{\pa \Om}$ as the duals to $H^s_{\pa \Om}$ w.r.t. the pivot space $H^0_{\pa \Om}$.

The operator $\De^{\pa \Om}$ can be considered as a homeomorphism from 
$H^1_{\pa \Om}$ to $H^{-1}_{\pa \Om}$. Indeed, it is easy to see from its definition that $\De^{\pa \Om}$ is an $L^2 (\pa \Om)$-paring selfadjoint operator, i.e., $\De^{\pa \Om} = (\De^{\pa \Om})^\#$ in the sense of (\ref{e:Am}) .
Hence, $\KK_0\perp \De^{\pa \Om} H^1 (\pa \Om) $.

The quadratic form 
$  (f|f)_1 = \| \grad_{\pa \Om} f \|_{\LLT} $, $ f \in \dom (\cdot|\cdot)_1 :=  H^1_{\pa \Om} $, is closed in $H^0_{\pa \Om}$. It defines a unique nonpositive selfadjoint operator $\De_{\pa \Om}$ in $H^0_{\pa \Om}$ by the requirements that 
$H^1_{\pa \Om} = \dom (- \De_{\pa \Om})^{1/2}$ and $(-\De_{\pa \Om} f|f)_{H^0_{\pa \Om}} =  (f|f)_1$ for all 
$f \in \dom \De_{\pa \Om}$
(this follows, e.g., from the Friedrichs representation theorem and its refinement \cite[Theorems VI.2.1 and VI.2.23]{Kato13}).

It is obvious that $\De_{\pa \Om}$ is invertible and 
$\De_{\pa \Om} = \De^{\pa \Om} \uph_{\dom \De_{\pa \Om}} $. Consequently,  $\De_{\pa \Om}$ has a compact resolvent  and a purely discrete spectrum $\si (\De^{\pa \Om})$. Thus, there exists 
\begin{multline}
\text{an orthonormal basis $\{u_k\}_{k=1}^{\infty}$ in $H^0_{\pa \Om}$ such that}
\\
\text{ $-\De_{\pa \Om} u_k = \la_k^2 u_k$ and $\si (\De_{\pa \Om}) = \{-\la_k^2 \}_{k=1}^{\infty}$, \  where  \ $\la_k >0$, \ $k \in \NN$.} \label{e:basisDe}
\end{multline}

Consider the scale $H^s_{\De_{\pa \Om}}$, $s \in \RR$, of Hilbert spaces associated with $\De_{\pa \Om}$, i.e., we define $| \cdot |_s:  H^0_{\pa \Om} \to [0,+\infty]$ by
$| f |_s^2 := \sum_{k=1}^\infty  \la_k^{2s} | (f|u_k)_{H^0_{\pa \Om}} |^2 $
and define Hilbert spaces 
$(H^s_{\De_{\pa \Om}}, | \cdot |_s)$ by
$H^s_{\De_{\pa \Om}} :=
\{ f \in H^0_{\pa \Om} \ : \ | f |_s < \infty \}$  for $s \ge 0$, and Hilbert spaces  
$H^{-s}_{\De_{\pa \Om}}$ as the completions of 
$H^0_{\pa \Om}$ w.r.t. $ | \cdot |_{-s} $  for $s>0$.

For $s \in [-1,1]$, it follows  from \cite{GMMM11} that
 $H^s_{\De_{\pa \Om}}$ can be identified with $H^s_{\pa \Om}$ up to equivalence of the norms.
The orthogonal Hodge decomposition (\ref{e:HDL2}) of $\LLt$ can be written as 
\begin{gather} \label{e:HDL22}
\LLt =  \grad_{\pa \Om} H^1_{\De_{\pa \Om}} \oplus \KK_1  (\pa \Om) \oplus \curlm_{\pa \Om} H^1_{\De_{\pa \Om}} ,
\end{gather}
where  $ \KK_1 (\pa \Om)$ is a space of dimension $b_1 (\pa \Om) < \infty$. 
The Hodge decompositions of the trace spaces of $\HH (\curlm,\Om)$ \cite{BHPS03} can be written as the following direct sums 
\begin{gather} \label{e:HDcurl}
\HH^{-1/2} (\curl_{\pa \Om}) = \grad_{\pa \Om} H^{1/2} _{\De_{\pa \Om}} \dot + \KK_1  (\pa \Om) \dot + \curlm_{\pa \Om} H^{3/2}_{\De_{\pa \Om}} 
,  \\
\HH^{-1/2} (\Div_{\pa \Om}) = \grad_{\pa \Om} H^{3/2}_{\De_{\pa \Om}} \dot + \KK_1  (\pa \Om) \dot + \curlm_{\pa \Om} H^{1/2} _{\De_{\pa \Om}}  
 ,
\label{e:HDdiv}
\end{gather}
where $ \grad(\curlm)_{\pa \Om} H^{1/2}_{\De_{\pa \Om}}$ and $ \grad (\curlm)_{\pa \Om} H^{3/2}_{\De_{\pa \Om}}$ are closed subspaces of the corresponding Hilbert spaces (see also \cite{BCS02} for the detailed treatment of the case $b_1 (\pa \Om) = 0$, $b_0 (\pa \Om) = 1$).

Taking into account $|\cdot|_1^2 = (\cdot|\cdot)_1$, one sees that the restrictions 
\begin{gather} \label{e:GradCurl}
\text{$\Grad_{\pa \Om}:=\grad_{\pa \Om} \uph_{H^1_{\De_{\pa \Om}}}$ and 
$\Curl_{\pa \Om}:= \curlm_{\pa \Om} \uph_{H^1_{\De_{\pa \Om}}}$}
\end{gather}
are isometric operators from $H^1_{\De_{\pa \Om}}$ to $\LLt$.
The $\LLt$-pairing adjoints $\Grad_{\pa \Om}^\# = - \Div_{\pa \Om} \uph_{\LLT}$ and $\Curl_{\pa \Om}^\# = \curl_{\pa \Om} \uph_{\LLT}$ are continuous operators from  $\LLt$ to $H^{-1}_{\De_{\pa \Om}}$, see Section \ref{s:curl} for the definitions of $ \Div(\curl)_{\pa \Om}$.
We consider in $\LLt$ the operator 
\[
\Dev := \Grad_{\pa \Om} \ (\Div_{\pa \Om}\!\uph_{\LLT})  - \Curl_{\pa \Om} \ (\curl_{\pa \Om} \!\uph_{\LLT}) .
\]
Theorem \ref{t:basis} shows that $\Dev$ is an analogue of the Laplace-de Rham operator, cf. \cite{C96}.

Starting from the orthonormal basis $\{u_k\}_{k=1}^{\infty}$ 
of eigenfunctions of $\De_{\pa \Om}$ in $H^0_{\pa \Om}$, we put
\begin{gather}
\text{$\vv_{-j} := \la_j^{-1} \grad_{\pa \Om} u_j$, \quad $\vv_j := \la_j^{-1} \curlm_{\pa \Om} u_j $, \qquad $j \in \NN$,} \label{e:vj}\\
\text{$\la_{-j}:=\la_j^{-1}$ for $j \in \NN$ and  $\ww_j^{\mp,\pm} := \la_j^{\pm 1/2} \vv_j $ for $j \in \ZZ \setminus \{0\}$,} \label{e:wj}\\
\text{and fix an arbitrary orthonormal basis $\{\vv_{0,j}\}_{j=1}^{b_1 (\pa \Om)}$  in $(\KK_1  (\pa \Om), \| \cdot\|_{\LLT})$}. \label{e:v0j}
\end{gather}


\begin{thm} \label{t:basis} (i) 
$\{\vv_j \}_{j\in \ZZ\setminus\{0\}} \cup \{\vv_{0,j}\}_{j=1}^{b_1(\pa \Om)} $ is an orthonormal basis in $\LLt$.

\item[(ii)] 
$\{\ww_j^{-,+} \}_{j\in \ZZ\setminus\{0\}} \cup \{\vv_{0,j}\}_{j=1}^{b_1(\pa \Om)} $ is a Riesz basis in $\HH^{-1/2} (\curl_{\pa \Om})$.

\item[(iii)] 
$\{\ww_j^{+,-} \}_{j\in \ZZ\setminus\{0\}} \cup \{\vv_{0,j}\}_{j=1}^{b_1(\pa \Om)} $ is a Riesz basis in $\HH^{-1/2} (\Div_{\pa \Om})$.

\item[(iv)] The operator $\Dev$ is selfadjoint in $\LLt$, and $\{\vv_j \}_{j\in \ZZ\setminus\{0\}} \cup \{\vv_{0,j}\}_{j=1}^{b_1(\pa \Om)} $ is a complete system of its  eigenvectors. Namely, $\Dev \vv_j = - \la_{|j|}^2 \vv_j$ , $j \in \ZZ \setminus \{0\}$, and $\ker \Dev = \KK_1 (\pa \Om) $.

\item[(v)] The operator
$\U_{\LLT\to\HH^{-1/2} (\curl_{\pa \Om})}  := 
\grad_{\pa \Om} \De_{\pa \Om}^{1/4} \Grad_{\pa \Om}^{-1} \dot + I_{\KK_1}\dot + \curlm_{\pa \Om} \De_{\pa \Om}^{-1/4} \Curl_{\pa \Om}^{-1}$ \linebreak
(operator $\U_{\LLt \to \HH^{-1/2} (\Div_{\pa \Om})}  := 
\grad_{\pa \Om} \De_{\pa \Om}^{-1/4} \Grad_{\pa \Om}^{-1} \dot +  I_{\KK_1} \dot + \curlm_{\pa \Om} \De_{\pa \Om}^{1/4} \Curl_{\pa \Om}^{-1}$)  
is a ho\-meomorphism from $\LLt$ to  $\HH^{-1/2} (\curl_{\pa \Om}) $ (to $\HH^{-1/2} (\Div_{\pa \Om})$, respectively). 
\end{thm}
\begin{proof}
We prove first (iv) and obtain (i) as a by-product. Then (v) will follow from the Hodge decompositions (\ref{e:HDcurl})-(\ref{e:HDdiv}) and yield immediately (ii) and (iii).

Let us prove (iv). Since $\Grad_{\pa \Om}^\# = - \Div_{\pa \Om} \uph_{\LLT}$ and $\Curl_{\pa \Om}^\# = \curl_{\pa \Om} \uph_{\LLT}$, one sees that
\[
\text{$
\ker (\Div_{\pa \Om} \uph_{\LLT})  = \KK_1  (\pa \Om) \oplus \curlm_{\pa \Om} H^1_{\De_{\pa \Om}} $ \  and  \ $\ker (\curl_{\pa \Om} \uph_{\LLT}) = \grad_{\pa \Om} H^1_{\De_{\pa \Om}} \oplus \KK_1  (\pa \Om)$. }
\]
This implies $0= \curl_{\pa \Om}  \grad_{\pa \Om}= \Div_{\pa \Om} \curlm_{\pa \Om}$ and, due to (\ref{e:GradCurl}), implies 
\[ \text{
$
\dom \Dev = \{ \vv \in \LLt \ : \ \Div_{\pa \Om} \vv \in H^1_{\De_{\pa \Om}}, \ \curl_{\pa \Om} \vv \in H^1_{\De_{\pa \Om}} \}
$}
\]
(recall that $\dom (A+B) := \dom (A) \cap \dom (B)$). 

Hence 
$\Dev = 
\Grad_{\pa \Om}  \ \De_{\pa \Om} \!\uph_{H^3_{\De_{\pa \Om}}} 
 \Grad_{\pa \Om} ^{-1} \ \oplus \ \ 0_{\KK_1} \ \oplus \ \Curl_{\pa \Om} \ \De_{\pa \Om}\! \uph_{H^3_{\De_{\pa \Om}}} \Curl_{\pa \Om}^{-1} .
$
Now statement (iv) follows from the following  facts: (a) the operators $\Grad_{\pa \Om}^{ \pm 1}$ (operators $\Curl_{\pa \Om}^{ \pm 1}$) can be seen as unitary transformations between $H^1_{\De_{\pa \Om}}$ and $\grad_{\pa \Om} H^1_{\De_{\pa \Om}}$ (resp., $\curlm_{\pa \Om} H^1_{\De_{\pa \Om}}$), (b) 
the selfadjoint in $H^1_{\De_{\pa \Om}}$ operator $\De_{\pa \Om} \uph_{H^3_{\De_{\pa \Om}}}$
has the orthonormal in $H^1_{\De_{\pa \Om}}$ basis of eigenfunctions $\{\la_j^{-1}u_j\}_{j\in \NN}$. This completes the proof of (iv)
and of the theorem.
\end{proof}

\begin{cor}
 \label{c:DevDuality}
(i) There exists in $\HH^{-1/2} (\curl_{\pa \Om})$  (in $\HH^{-1/2} (\Div_{\pa \Om})$) a unique Hilbertian norm $\| \cdot \|_{\pi(\ga)} $ that is equivalent to the norm $|\cdot|_{\pi(\ga)}$  of (\ref{e:NoH-1/2}) and satisfy 
\begin{gather*} \label{e:||piDe0}
\textstyle \| \uu \|_{\pi}^2 =  \sum_{j \in \ZZ\setminus\{0\}} \la_{j}^{-1} | (\uu|\vv_j)_{\LLT}|^2  + \sum_{j=1}^{b_1 (\pa \Om)} | (\uu|\vv_{0,j})_{\LLT} |^2 , \quad \uu \in \LLT \cap \HH^{-1/2} (\curl_{\pa \Om}),
 \\ 
 \label{e:||gaDe0}
\textstyle \Bigl(\text{resp., } \| \uu \|_{\ga}^2 = 
  \sum_{j \in \ZZ\setminus\{0\}} \la_{j} | (\uu|\vv_j)_{\LLT} |^2 + \sum_{j=1}^{b_1 (\pa \Om)} | (\uu|\vv_{0,j})_{\LLT} |^2 ,   \uu \in \LLT \cap \HH^{-1/2} (\Div_{\pa \Om}) \Bigr).
\end{gather*}
\item[(ii)] The spaces $\Hs_{-,+} = \left(\HH^{-1/2} (\curl_{\pa \Om}) ,  \| \cdot \|_{\pi}\right)$ and $\Hs_{+,-}=\left(\HH^{-1/2} (\Div_{\pa \Om}), \| \cdot \|_{\ga} \right)$ are m-order spaces dual to each other w.r.t. $\Hs=\LLt$  (see Definition \ref{d:mutualduality}).

\item[(iii)] $\{\ww_j^{\mp,\pm} \}_{j\in \ZZ\setminus\{0\}} \cup \{\vv_{0,j}\}_{j=1}^{b_1(\pa \Om)} $ is an orthonormal basis in the space $\Hs_{\mp,\pm}$ of statement (ii).

\item[(iv)] The systems of vectors $\{\ww_j^{-,+} \}_{j\in \ZZ\setminus\{0\}} \cup \{\vv_{0,j}\}_{j=1}^{b_1(\pa \Om)} $ and $\{\ww_j^{+,-} \}_{j\in \ZZ\setminus\{0\}} \cup \{\vv_{0,j}\}_{j=1}^{b_1(\pa \Om)} $ are bi-orthogonal w.r.t. the $\LLT (\Om)$-pairing. 
\end{cor}
\begin{proof}
Since $\{\ww_j^{\mp,\pm} \}_{j\in \ZZ\setminus\{0\}} \cup \{\vv_{0,j}\}_{j=1}^{b_1(\pa \Om)} $
are Riesz bases by Theorem \ref{t:basis}, there is 
a unique Hilbertian norm $\| \cdot \|_{\pi} $ in $\HH^{-1/2} (\curl_{\pa \Om})$  (resp., $\| \cdot \|_{\ga} $ in $\HH^{-1/2} (\Div_{\pa \Om})$) that is equivalent to $|\cdot|_{\pi(\ga)}$ and makes the corresponding basis orthonormal. It follows from (\ref{e:vj})-(\ref{e:v0j}) and Theorem \ref{t:basis} that
$\{\ww_j^{-,+} \}_{j\in \ZZ\setminus\{0\}} \cup \{\vv_{0,j}\}_{j=1}^{b_1(\pa \Om)} $ and $\{\ww_j^{+,-} \}_{j\in \ZZ\setminus\{0\}} \cup \{\vv_{0,j}\}_{j=1}^{b_1(\pa \Om)} $ are bi-orthogonal w.r.t.  $(\cdot,\star)_{\LLT}$. 
This and (\ref{e:wj}) imply the statements (i) and (iii). It is easy to see from formulae of statement (i) that the conditions of m-order duality in Definition \ref{d:mutualduality}  are satisfied in statement (ii). This also justifies the replacement of the inner product  $(\cdot|\star)_{\LLT}$ with the pairing
$\<\cdot|\star\>_{\LLT}$ in the aforementioned bi-orthogonality statement, proving in this way statement (iv).
\end{proof}

\begin{rem}  \label{r:Interpolation}
Corollary \ref{c:DevDuality} (iii) implies that (\ref{e:HDcurl}) and (\ref{e:HDdiv}) are orthogonal w.r.t. $\|\cdot\|_{\pi}$ and, resp.,  $\|\cdot\|_{\ga}$.
Corollary \ref{c:DevDuality} (ii) and \cite[Appendix to IX.4, Example 3]{RSII} yield that $\LLt = [\HH^{-1/2} (\curl_{\pa \Om}), \HH^{-1/2} (\Div_{\pa \Om})]_{1/2} $  up to equivalence of the norms, where $[\cdot, \star ]_\th$ denotes the complex interpolation of Hilbert spaces (in the sense described in  \cite{RSII}).
\end{rem}

\begin{lem} \label{l:propertiesHadj}
Let $\Hs_{\mp,\pm}$ be a pair of m-order spaces dual to each other w.r.t. $\Hs$ (see Section \ref{s:MDAbstract}).
Let $\Uop_{\Hs_{\mp,\pm} \to \Hs}$, $\Uop_{\Hs \to \Hs_{\pm,\mp}}$, and $\Uop_{\Hs_{-,+} \to \Hs_{+,-}}$ be the unitary operators defined in Proposition \ref{p:MixOrder}. Then $\Uop_{\Hs_{\mp,\pm} \to \Hs}^\# = \Uop_{\Hs\to \Hs_{\pm,\mp}}$
and $\Uop_{\Hs_{-,+} \to \Hs_{+,-}} = \Uop_{\Hs_{-,+} \to \Hs_{+,-}}^\#$.
Moreover, $T^\# = \Uop_{\Hs_{-,+} \to \Hs_{+,-}} T^* \Uop_{\Hs_{-,+} \to \Hs_{+,-}}$  for densely defined  operators $T: \dom T \subseteq \Hs_{-,+}  \to \Hs_{+,-} $.
\end{lem} 
\begin{proof} 
Let $\Sop:\wt \Hs_{-,+} \subseteq \Hs \to \Hs$ be the injective selfadjoint operator defined in Proposition \ref{p:MixOrder}.  
Recall that  $\wt \Hs_{\mp,\pm}=\Hs \cap \Hs_{\mp,\pm}$ and 
$\Sop^{\pm1} = \Uop_{\Hs_{\mp,\pm} \to \Hs} \uph_{\wt \Hs_{\mp,\pm}} = \Uop_{\Hs \to \Hs_{\pm,\mp}} \uph_{\wt \Hs_{\mp,\pm}}$. 
So the obvious for $h_{\mp,\pm} \in \wt \Hs_{\mp,\pm}$ equality 
$( h_{-,+} |  h_{+,-} )_\Hs = 
\<\Uop_{\Hs_{-,+} \to \Hs} h_{-,+} | \Uop_{\Hs_{+,-} \to \Hs} h_{+,-} \>_\Hs  
$
and (\ref{e:Am}) imply the desired statements. 
\end{proof}

\begin{cor} \label{c:Upiga}
Let $\Hs_{\mp,\pm}$ and $\Hs$ be as in Corollary \ref{c:DevDuality} (ii).
Then 
\[
\text{$\Uop_{\Hs \to \Hs_{-,+}} = \U_{\LLT\to\HH^{-1/2} (\curl_{\pa \Om})} $ and  $\Uop_{\Hs \to \Hs_{+,-}} = \U_{\LLt \to \HH^{-1/2} (\Div_{\pa \Om})} $ \quad (cf. Theorem \ref{t:basis}). }
\]
\end{cor}
\begin{proof}
The choice of the norm $\|\cdot \|_{\Hs_{-,+}} = \| \cdot \|_{\pi} $, 
implies that $\Sop:\wt \Hs_{-,+} \subseteq \Hs \to \Hs$ 
is diagonal w.r.t. the basis $\{\vv_j \}_{j\in \ZZ\setminus\{0\}} \cup \{\vv_{0,j}\}_{j=1}^{b_1(\pa \Om)} $,
 $\Sop \vv_j = \la_j^{-1/2} \vv_j$ for $j \neq 0$ and $S_{-,+} \vv_{0,j} = \vv_{0,j}$ for 
 $j =1$, \dots, $b_1(\pa \Om)$. So the unique extension $\Uop_{\Hs \to \Hs_{+,-}}$ of $\Sop $ to a unitary operator
 from $\Hs$ to $\Hs_{+,-}$ takes the form of $\U_{\LLT\to\HH^{-1/2} (\curl_{\pa \Om})}$.
The proof for $\Uop_{\Hs \to \Hs_{-,+}}$ is analogous.
\end{proof}
\begin{rem} \label{r:Upi}
Since $\Uop_{\Hs_{-,+} \to \Hs}^{-1}=\Uop_{\Hs \to \Hs_{-,+}}$ by definition, and $\Uop_{\Hs_{-,+} \to \Hs}^\# = \Uop_{\Hs \to \Hs_{+,-}}$ by Lemma \ref{l:propertiesHadj},  one sees that  Corollary \ref{c:Upiga} 
proves Theorem \ref{t:Upi} with $\Upi= \U_{\LLT\to\HH^{-1/2} (\curl_{\pa \Om})}^{-1}$.
\end{rem}

\section{Reduction tuples and associated Calkin triples} \label{s:RT}

Let $A$ be a  symmetric closed densely defined operator in a  Hilbert space $\Hc$, and let 
$\Hc^2 = \Hc \oplus \Hc$. We consider the graph $\Gr A^* = \{ \{f,A^*f \} \in \Hc^2 : f \in \dom A^*\}$ of the adjoint operator $A^*$ and the domain $\dom A^*$ of $A^*$ as Hilbert spaces equipped with the graph norm and  use systematically the identification of the normed spaces $(\Gr A, \| \cdot \|_{\Gr A})$ and $(\dom A, \| \cdot \|_{\dom A})$.  


\begin{defn} \label{d:rt} 
Let $\Hs_{\mp,\pm}$ be m-order spaces dual  to each other w.r.t. $\Hs$ and 
let $\wt \Hs_{\mp,\pm} := \Hs_{\mp,\pm} \cap \Hs$
(see Section \ref{s:MDAbstract}).
We  say that $(\Hs_{-,+},\Hs,\Gb', W)$ is a \emph{reduction tuple for} $A^*$ if 
the following conditions hold:
\item[(R1)] $ \Gb':\Gr A^* \to \Hs_{-,+}$ is a surjective  operator,
\item[(R2)] $W$ is a linear homeomorphism from $\Hs_{-,+}$ onto $\Hs_{+,-}$,
\item[(R3)] $(A^*f|g)_{\Hc}-(f,A^*g)_{\Hc} =  \< \Gb' \{f,A^*f\} | W \Gb' \{g,A^*g\} \>_\Hs$ for all $f,g \in \dom A^*$.
\end{defn}

\begin{rem}
 Similar to Remark \ref{r:Ga}, we will use  the shortened notations $(\Hs_{-,+},\Hs,\Gb, W)$ with $\Gb $ instead of $\Gb'$,
where the operator $\Gb f:=\Gb' \{f,A^*f\}$ acts from from $\dom A^*$ to $\Hs_{-,+}$. So 
\begin{gather} 
\label{e:R3'}
\text{(R3) takes the shorter form }
(A^*f|g)_{\Hc}-(f,A^*g)_{\Hc} =  \< \Gb f | W \Gb g \>_\Hs , \quad f,g \in \dom A^* .
\end{gather}
\end{rem}

\begin{defn} \label{d:urt} 
A reduction tuple $(\Hs_{-,+},\Hs, \Gb, W)$ will be called unitary if it satisfies the following conditions:
\item[(U1)] $W:\Hs_{-,+} \to \Hs_{+,-}$ is a unitary operator, 
\item[(U2)] $ W$ maps $\wt \Hs_{-,+} $ onto $\wt \Hs_{+,-} $ and $\| W h\|_\Hs = \| h\|_\Hs$ for all $h \in \wt \Hs_{-,+}$.
\end{defn}

\begin{rem}\label{r:iCT} For every operator  adjoint to a  certain symmetric operator, there exists a unitary reduction tuple. A trivial example can be obtained from the example of the surjective Calkin reduction operator $\Gb'$ in \cite[Theorem 2.8]{C39}, which obviously produces a unitary reduction tuple $(\Hs,\Hs,\Gb',W)$. 
\end{rem}

\begin{defn} \label{d:c} 
We say that a reduction tuple $(\Hs_{-,+},\Hs, \Gb, W)$ and a Calkin triple $(\Hs,\wt \Gb,\wt W)$ for $A^*$ are associated with each other if
$\Gb f = \wt \Gb f$ for all $f \in \dom \wt \Gb $, $W h = \wt W h$ for all $h \in \wt \Hs_{-,+} $, and  $ W^{-1} h = \wt W^{-1} h$ for all $h \in \wt \Hs_{+,-} $.
\end{defn}

\begin{prop} \label{p:RT}
Let $(\Hs_{-,+},\Hs,\Gb,W)$ be a reduction tuple for $A^*$. Then:
\item[(i)] The operator $\Gb':\Gr A^* \to \Hs_{-,+}$ is closed and  bounded.
\item[(ii)] The tuple $( \Hs_{+,-},\Hs, c W  \Gb,\, - W^{-1})$ is a reduction tuple for $A^*$ for every $c \in 
\{ z \in \CC : |z|=1\}$.
The reduction tuple $( \Hs_{+,-},\Hs, \ii W  \Gb,\, - W^{-1})$ will be called dual to $(\Hs_{-,+},\Hs,\Gb,W)$.

\item[(iii)] $\Gr A = \ker \Gb'$, $\dom A = \ker \Gb$, and $W^\# = - W$.
\end{prop}

\begin{proof}
(i) can be obtained from (\ref{e:R3'}) by standard arguments based on the surjectivity of $\Gb'$ and $W$ (cf. the case of boundary triple in \cite[Lemma 7.2]{DM17}). 
(ii)-(iii) The skew-symmetry of the form $[f|g]_{A^*} = (A^*f|g)_{\Hc}-(f,A^*g)_{\Hc}$ and  (\ref{e:R3'}) imply   
\[ \< c W  \Gb f | (- W^{-1}) c W  \Gb g \>_\Hs  = - \<  W  \Gb f |  \Gb g \>_\Hs 
 =  \< \Gb f | W  \Gb g \>_\Hs \] for $ f,g \in \dom A^*$.
This gives $W^\# = - W$ and (\ref{e:R3'}) for the tuple $\mathfrak{t} = ( \Hs_{+,-},\Hs, c W  \Gb,\, - W^{-1})$.  Now (R1) and (R2) for $\mathfrak{t}$ and $\ker \Gb = \dom A$ are obvious.
\end{proof}

\begin{cor} \label{c:RTtoCT}
Let $(\Hs_{-,+},\Hs,\Gb,W)$ be a unitary reduction tuple for $A^*$. Let $\wt \Gb  := \Gb \uph_{\dom \wt \Gb}$ with 
$\dom \wt \Gb :=\{f \in \dom A^* \ : \ \Gb f \in \Hs \}$.
Then there exists a unique unitary operator $\wt W:\Hs \to \Hs$
such that $\wt W h =  W h$ for all $h \in \wt \Hs_{-,+} $ and 
$(\Hs,\wt \Gb,\wt W)$ is a Calkin triple for $A^*$. Besides, 
$(\Hs_{-,+},\Hs,\Gb,W)$ and $(\Hs,\wt \Gb,\wt W)$ are  associated  in the sense of Definition \ref{d:c}.
\end{cor}
\begin{proof}
Condition (U2) of Definition \ref{d:c} implies that one can consider $W_0 := W \uph_{\wt \Hs_{-,+}}$ as a bounded densely defined  operator from $\Hs$ to $\Hs$. Moreover, 
its continuous extension $\wt W:\Hs \to \Hs$ to 
$\Hs$ is isometric and has the image $\wt W \Hs \supseteq \wt \Hs_{+,-} $. Since $\wt \Hs_{+,-}$ is dense in $\Hs$, the operator $\wt W$ is unitary. Clearly, $\wt W$ is a unique unitary operator in $\Hs$
such that $\wt W h =  W h$ for $h \in \wt \Hs_{-,+} $.
We prove in several steps that conditions (C1)-(C2) of Definition \ref{d:CT} hold for $(\Hs,\wt \Gb,\wt W)$.

\emph{Step 1.} We shall prove (C1) which says that $\dom \wt \Gb$ is dense in $(\dom A^* , \| \cdot \|_{\dom A^*})$. By definition, 
$\dom \wt \Gb = \{f \in \dom A^* \ : \ \Gb f \in \wt \Hs_{-,+} \}$, and so $\ran \wt G = \wt \Hs_{-,+} $.
By property (R1) of  $(\Hs_{-,+},\Hs,\Gb,W)$ and by Proposition \ref{p:RT} (iii), we see that $\dom A = \ker \Gb = \ker \wt \Gb$ and that $ \Gb$ maps  $\dom A^* \ominus \dom A$ bijectively and continuously onto $\Hs_{-,+}$. Since $\wt \Hs_{-,+}$ is dense both in  $\Hs_{-,+}$ and in $\Hs$, the set 
$
\{f \in \dom A^* \ominus \dom A \ : \   \Gb f \in \wt \Hs_{-,+}\}
$
is dense in $\dom A^* \ominus \dom A$ and 
is  contained in $\dom \wt \Gb$. On the other hand, $\dom \wt \Gb$
contains $\dom A = \ker \wt \Gb$. This implies (C1) for the triple $(\Hs,\wt \Gb,\wt W)$.

\emph{Step 2.} The implication $\Leftarrow$ of (C2) takes the form $(A^*f |g)_\Hc - (f |A^* g)_\Hc =  ( \wt \Gb f | \wt W \wt \Gb g)_{\Hs}$, $f,g \in \dom \wt \Ga$, and follows immediately from (R3).

\emph{Step 3.} Let us prove the implication $\Rightarrow$ of (C2), i.e., we assume $(A^* f |g)_{\Hc} - (f,g')_{\Hc} = (\wt \Gb f| h )_{\Hs} $ for all 
 $f \in \dom \wt \Gb$,  and have to prove that
$g \in \dom \wt \Gb$, $g' = A^* g$, and $h =\wt W \wt \Gb g$.
For arbitrary $f \in \dom A$, we have $\wt \Gb f =0$, and in turn,
 $(A f |g)_{\Hc} = (f,g')_{\Hc} $, which implies $\{g,g'\} \in \Gr  A^*$. Now it follows from (R3) that 
$
(A^* f |g)_{\Hc} - (f,g')_{\Hc} = \<  \Gb f| W  \Gb g \>_{\Hs} $ for all $f \in \dom A^*$. So $\< h' |  W  \Gb g \>_{\Hs} =  ( h' | h )_{\Hs}$ for all $h' \in \ran \wt \Gb$. It follows from Step 1 that $\ran \wt \Gb = \wt \Hs_{-,+}$, and so $\| h \|_{\Hs_{+,-}}<\infty$ (see (\ref{e:defH+-}) for the definition of $\| \cdot \|_{\Hs_{+,-}}$). Thus, $h \in \wt \Hs_{+,-}$ and $h =  W  \Gb g$. By (U2), 
 $ \Gb g = W^{-1} h \in \wt \Hs_{-,+}$, and so, $g \in \dom \wt \Gb$ and 
$\wt W \wt \Gb g =  W  \Gb g = h$. This completes the proof of (C2).

The fact that the tuples $(\Hs_{-,+},\Hs,\wh \Gb,\wh W)$ and $(\Hs, \Gb, W)$ are associated is now obvious.
\end{proof}


In the rest of the section, we fix in  $\HH^{-1/2} (\curl_{\pa \Om})$ (in $\HH^{-1/2} (\Div_{\pa \Om})$) the norm $\|\cdot\|_{\pi (\ga)}$ defined by Corollary \ref{c:DevDuality}. By Corollary \ref{c:DevDuality} (ii),
\begin{gather*} \label{e:Hpmpiga}
\text{$\Hs_{-,+} = \left(\HH^{-1/2} (\curl_{\pa \Om}) ,  \| \cdot \|_{\pi}\right)$ and $\Hs_{+,-}=\left(\HH^{-1/2} (\Div_{\pa \Om}), \| \cdot \|_{\ga} \right)$ } 
\end{gather*}
are m-order spaces dual to each other w.r.t. $\Hs=\LLt$, see Definition \ref{d:mutualduality}.

Let us define the linear homeomorphism
$\n_\times^{\pi \to \ga} :=- \ga_\top \pi_\top^{-1}$ from $\HH^{-1/2} (\curl_{\pa \Om})$ onto $\HH^{-1/2} (\Div_{\pa \Om})$ and put $\n_\times^{\ga \to \pi} := -  (\n_\times^{\pi \to \ga})^{-1}$.
Note that
\begin{gather} \label{e:npiga}
\n_\times^{\pi(\ga) \to \ga(\pi)} \uu = \n_\times \uu  \quad \text{  if $\uu \in \LLt \cap \HH^{-1/2} (\curl_{\pa \Om})$ (if 
$\uu \in \LLt \cap \HH^{-1/2} (\Div_{\pa \Om})$)}.
\end{gather}

\begin{lem} \label{l:npiga}
With the choice of norms as in Corollary \ref{c:DevDuality} (ii), $\n_\times^{\pi \to \ga}$ and $\n_\times^{\ga \to \pi}$ are   unitary operators. 
\end{lem}
\begin{proof}
It follows from  
$\curlm_{\pa \Om} \uph_{H^1 (\pa \Om)}  = - \n_\times   \grad_{\pa \Om} \uph_{H^1 (\pa \Om)}$ and (\ref{e:HDcurl})-(\ref{e:HDdiv}) that
the operators $\n_\times^{\pi(\ga) \to \ga(\pi)}$ coincide on $\KK_1 (\Om)$ with the unitary operator $\n_\times: \LLt \to \LLt$. 
Corollary \ref{c:DevDuality} (iii) and  
$\n_\times^{\pi \to \ga} \ww_j^{-,+}= - \ww_{-j}^{+,-}$, $j \in \ZZ \setminus \{0\}$, complete the proof.
\end{proof}


\begin{cor} \label{c:curlRT} 
(i) In the settings of Corollary \ref{c:DevDuality} (ii), the tuples
\[\text{
$(\HH^{-1/2} (\curl_{\pa \Om}),\LLt, \pi_\top, \  -\n_\times^{\pi \to \ga})$ and $(\HH^{-1/2} (\Div_{\pa \Om}) ,\LLt, \ii \ga_\top, \  -\n_\times^{\ga \to \pi})$ 
}
\]
are unitary reduction tuples  for $\curlm$ dual to each other in the sense of Proposition \ref{p:RT} (ii).

\item[(ii)] $(\LLt, \pi_{\top,2},  \ -\n_\times)$ and 
$(\LLt,  \ii \ga_{\top,2},  \ -\n_\times)$ are mutually dual Calkin triples for $\curlm$. They are associated with the respective  reduction tuples of statement (i).
\end{cor}
\begin{proof}
(i) The conditions of Definition \ref{d:rt} follow from the definitions of $\n_\times^{\pi(\ga) \to \ga(\pi)}$ and (\ref{e:TrTh})--(\ref{e:IntByPpiga}). The conditions 
of Definition \ref{d:urt} follow from (\ref{e:npiga}) and the proof of 
Lemma \ref{l:npiga}. 

Statement (ii) follows from (i)   and  Corollary \ref{c:RTtoCT}.
\end{proof}

The connection between these unitary reduction tuples and Calkin triples  will be crucial for the rigorous handling of Leontovich-type boundary conditions, see Section \ref{s:GIBCa}.

\section{M-boundary tuples and abstract Maxwell operators}
\label{s:mrt}

The aim of this section is the description of all m-dissipative extensions of an ´abstract version´ of the Maxwell operator  $\M$ using m-boundary tuples of Definition \ref{d:MBT}.
We  introduce ´abstract Maxwell operators´ in the following way.
Let $A$ be a closed densely defined symmetric operator in a Hilbert space $\Hc$. Let $S$ be a bounded uniformly positive selfadjoint operator in $\Hc^2 = \Hc \oplus \Hc$ (so $S^{-1}$ is also bounded). One can define on $\Hc^2  \times \Hc^2$ another  inner product $(\cdot|\star)_{\Hc^{2,S}} := (S \cdot|\star)_{\Hc^2}$,
which generates an equivalent norm $\| \cdot \|_{\Hc^{2,S}}$. This defines the `weighted' Hilbert space $\Hc^{2,S} := (\Hc^2, \| \cdot \|_{\Hc^{2,S}})$.
As an \emph{abstract symmetric Maxwell operator}, we consider in $\Hc^{2,S}$ the `$S$-weighted'  operator 
\begin{align} \label{e:M}
M  \psi 
:=   
S^{-1} 
\begin{pmatrix}
 0 & \ii A \\  -\ii A & 0 
\end{pmatrix} \begin{pmatrix} \psi_1 \\ \psi_2 \end{pmatrix}, 
\quad \psi = \begin{pmatrix} \psi_1 \\ \psi_2 \end{pmatrix} \in \dom M = (\dom A)^2,
\end{align}
which is obviously symmetric, closed, and has the adjoint 
$M^* = S^{-1} \left( \begin{smallmatrix}
 0 & \ii A^* \\  -\ii A^* & 0 
\end{smallmatrix} \right).$

If $A = \curln$, the operator $M$ with a suitable $S$ becomes the Maxwell operator
$\M$ of (\ref{e:iM}). 

\subsection{M-boundary tuples: properties and connections with reduction tuples}
\label{s:MBT}


Reduction tuples (see Definition \ref{d:rt}) and m-boundary tuples (see Definition \ref{d:MBT}) are connected by the following statement.

\begin{prop} \label{p:M*MBT} 
Let two closed densely defined symmetric  operators, $A$ in $\Hc$ and $M$ in $\Hc^{2,S}$, be connected by (\ref{e:M}).
Let $(\Hs_{-,+},\Hs,\Gb, W)$ be a reduction tuple for $A^*$.
Let
\[
\text{ $\Ga_0 \psi := \Gb \psi_2 $ \quad and \quad $\Ga_1 \psi := \ii W  \Gb \psi_1$ \quad for $\psi = \{ \psi_1,\psi_2\} \in (\dom A^*)^2$.}
 \]
Then $(\Hs_{-,+},\Hs, \Ga_0 ,\Ga_1)$ is an m-boundary tuple for $M^*$.
\end{prop}
\begin{proof} The verification of conditions (M1)-(M2) of Definition \ref{d:MBT} is straightforward.
\end{proof}

Let $V$ be a linear homeomorphism from $\Hs$ to $\Hs_{-,+}$.
Then it is obvious from the definitions of Section \ref{s:MDAbstract} that $V^\# $ is a linear homeomorphism from $\Hs_{+,-}$ to $\Hs$ and 
\begin{gather} \label{e:VV++}
( V^\# h_{+,-} | V^{-1} h_{-,+})_\Hs = \< h_{+,-}|h_{-,+}\> _\Hs \quad \text{ for all } h_{\mp,\pm} \in \Hs_{\mp,\pm}.
\end{gather}

\begin{prop}\label{p:BT}
 Let $(\Hs_{-,+},\Hs,\Ga_0,\Ga_1)$ be an m-boundary tuple for the adjoint $\A^*$ to a closed densely defined symmetric operator $\A$. Let $V$ be an arbitrary linear homeomorphism from $\Hs$ to $\Hs_{-,+}$. Then 
$(\Hs, V^{-1} \Ga_0, V^\# \Ga_1) $ is a boundary triple for $\A^*$ (see (\ref{e:BTr}) for the definition).
\end{prop}
\begin{proof}
The surjectivity of $\wh \Ga : \dom \A^* \to \Hs^2$,  $\wh \Ga: f \mapsto \{  V^{-1} \Ga_0,V^\# \Ga_1\}$, follows from (M1) of Definition \ref{d:MBT} and the surjectivity of $V$ and $V^\#$. Using (M2) of Definition \ref{d:MBT} and  (\ref{e:VV++}) 
one gets $
(\A^*f|g)_\Hc - (f|\A^*g)_\Hc =  (V^\# \Ga_1 f | V^{-1} \Ga_0 g )_\Hs  - (V^{-1} \Ga_0 f | V^\# \Ga_1 g )_\Hs $. 
\end{proof}

In abstract settings, the unitary operator $\Uop_{\Hs \to \Hs_{-,+}}$ defined in Proposition \ref{p:MixOrder} provides a convenient choice of a homeomorphism $V$. In this case, $\Uop_{\Hs \to\Hs_{-,+}}^\#=\Uop_{\Hs_{+,-} \to \Hs } = \Uop_{ \Hs \to\Hs_{+,-}}^{-1}$.

Proposition \ref{p:BT} allows one to translate known  results about boundary triples (see e.g. \cite{K75,GG91,DM95,DM17,BHdS20}) to the language of m-boundary tuples, which, at least for Maxwell operators, are better adjusted to the trace maps. In particular, Proposition \ref{p:BT} and well-known results on boundary triples (e.g. \cite[Section 7.1]{DM17}) imply the following.

\begin{prop} \label{p:PropMBT}
A closed densely defined symmetric operator $\A$ has equal deficiency indices $n_+ (\A) = n_- (\A)$ exactly when there exists an m-boundary tuple $(\Hs_{-,+},\Hs,\Ga_0,\Ga_1)$ for $\A^*$. If this is the case, then $n_\pm (\A) = \dim \Hs$ and 
the operators $\Ga_0$ and $\Ga_1$ are bounded.
\end{prop}

\subsection{Extensions of symmetric operators in terms of m-boundary tuples} \label{s:MBT-BT}

Let $\A$ be a densely defined symmetric operator in $\Hc$ with equal deficiency indices.
Let $(\Hs_{-,+},\Hs, \Ga_0 ,\Ga_1)$ be a certain m-boundary tuple for $\A^*$.
So $\Ga:f \mapsto \{ \Ga_0 f, \Ga_1 f\}$ is a surjective 
bounded operator from
$(\dom \A^*,\| \cdot \|_{\dom \A^*})$ onto $\Hs_{-,+} \oplus \Hs_{+,-}$ (see Proposition \ref{p:PropMBT}). 

A \emph{linear relation} from a Hilbert space $\Hs_1$ to a Hilbert space $\Hs_2$ is by definition a linear subspace in $\Hs_1 \oplus \Hs_2$, e.g., a graph $\Gr B$ of an operator $B:\dom B \subseteq \Hs_1 \to \Hs_2$ is a linear relation. Identifying operators with their graphs, one can consider operators as particular cases of linear relations, i.e., a linear relation $\Th$ can be considered as an operator if and only if its \emph{multivalued part} $\Th (0):=\left\{h_2 \in \Hs_2 \ : \ \{0,h_2\} \in \Th \right\}$ is equal to $\{0,0\}$. Then the notions of closed relation and closed operator are consistent (see \cite{K75,DM95,HW12,DM17,BHdS20} for the basics on linear relations). If $\Th$ is a linear relation from $\Hs$ to $\Hs$, $\Th$ is said to be a linear relation in $\Hs$.

Recall that an operator $\wt \A$ is called an \emph{admissible extension of} $\A$ if $\Gr \A \subseteq \Gr \wt \A \subseteq \Gr \A^*$. It follows from Definition \ref{d:MBT} that $\Gr \A = \ker \Ga$. Consequently, there exists 1--to--1 correspondence between: (i) the family of admissible extensions $\wt \A$ of $\A$, 
(ii) the family of subspaces $\Gr (\wt \A) /\Gr (\A)$ of the factor space $\Gr (\A^*) / \Gr (\A)$, 
(iii) and the family of 
linear relations $\Theta = \left\{ \{ \Ga_0 u , \Ga_1 u \} \ : \ u \in \Gr \wt \A \right\}$ from $\Hs_{-,+}$ to $\Hs_{+,-}$.
Namely, for the correspondence (i) $\leftrightarrow$ (iii), the restriction $\A_\Theta$ of $\A^*$ associated with a linear relation $\Th$ is defined by
\begin{gather} \label{e:ATh}
\text{$\A_\Theta := \A^*\uph_{\dom \A_\Th}$, \ where $\dom \A_\Th := \Ga^{-1} \Th = \left\{ u \in \dom \A^* \ : \ \{ \Ga_0 u , \Ga_1 u \} \in \Th \right\}$}.
\end{gather}
Then $\Gr (\A_\Th) / \Gr (\A)$ is the corresponding subspace of the factor space $\Gr (\A^*) / \Gr (\A)$.

\begin{defn} \label{d:Th}
Let $\Th$ be a linear relation from $\Hs_{-,+}$ to $\Hs_{+,-}$.

\item[(i)] A numerical cone of $\Th$ is defined by 
$
\Ncone (\Th) := \{ \< h_1| h_0 \>_\Hs \ : \ \{ h_0 , h_1 \} \in \Th\}  .
$

\item[(ii)] A linear relation $\Th$ is said to be symmetric, nonnegative, dissipative,  or accretive if, resp., $\Ncone (\Th) \subseteq \RR$, $\Ncone (\Th) \subseteq [0,+\infty) $, $\Ncone (\Th) \subseteq \overline{\CC}_- $, or $\Ncone (\Th) \subseteq \ii \overline{\CC}_-$, where $\overline{\CC}_\pm := \{z \in \CC \ : \ \pm \im z \ge 0\}$.  Besides, a linear relation $\Th$ in each of these classes is called maximal if it cannot be extended to another linear relation of the same class. 

\item[(iii)] The $\Hs$-pairing-adjoint linear relation $\Th^\#$ consist of all $\{g,g'\} \in \Hs_{-,+} \oplus \Hs_{+,-}$ such that 
$\<f' | g \>_\Hs  =  \<f | g'\>_\Hs$
 \quad 
 for all  $\{f,f'\} \in \Th$.

\item[(iv)] $\Th$ is called $\Hs$-pairing-selfadjoint if $\Th =\Th^\#$.

\item[(v)] A linear operator $T:\dom T \subseteq \Hs_{-,+} \to \Hs_{+,-}$ is called symmetric, nonnegative, dissipative,   or accretive if the linear relation $\Gr T$ is so.
\end{defn}

These are mixed-order duality analogues of standard definitions, which are well-known in the case $\Hs_{\mp,\pm} = \Hs$. In particular, (iii) is a generalization of the standard definition of the adjoint linear relation $\Th^*$ in $\Hs$. The names for the above classes of linear relations and operators are often interchanged, cf. \cite{Kato13,LL04,DM17}. We use  the terminology of mathematical physics \cite{E12} which places numerical ranges of dissipative operator to $\overline{\CC}_-$. 

\begin{rem} \label{r:AThvsTh} Obviously,
$\Gr \A_{\Th_1} \subseteq  \Gr \A_{\Th_2}$ if and only if $\Th_1 \subseteq \Th_2$.
Note that $\Th^\#$ is a closed linear relation from $\Hs_{-,+}$ to $\Hs_{+,-}$ and that $(\Th^\#)^\#$ is the closure $\overline{\Th}$ of $\Th$. 
The definitions of $\Hs$-pairing-adjoint linear relations and operators are obviously consistent  in the sense that $\Gr (T^\#) = (\Gr T)^\#$, see Section \ref{s:MDAbstract}. For the operator (\ref{e:ATh}), $(\A_\Th)^* = \A_{\Th^\#}$ and $\overline{\A_\Theta} = \A_{\overline{\Th}}$. It is clear that 
$\A_\Theta$ is closed if and only if $\Theta$ is so.
\end{rem}

\begin{rem} \label{r:m-D}
A maximal dissipative linear relation is closed. A graph $\Gr T$ of a maximal dissipative operator $T$ is not necessarily a maximal dissipative linear relation (see \cite{P59} for an example of nonclosed maximal dissipative operator $T$). A densely defined operator $T:\dom T \subseteq \Hc \to \Hc$ is maximal dissipative (maximal accretive) if only if $\Gr T$ is so \cite{P59}.
\end{rem}

\begin{cor} \label{c:MBTm-D} Let an operator $\wt \A$ be a dissipative extension of $\A$. Then:
\item[(i)] there exists a dissipative linear relation $\Th$ from $\Hs_{-,+}$ to $\Hs_{+,-}$ such that $\wt \A=\A_\Th$, where $\A_\Th$ is defined by (\ref{e:ATh}). 

\item[(ii)] The operator $\A_\Th$ is m-dissipative if and only if the linear relation $\Th$ is maximal dissipative.
\end{cor}

\begin{proof} 
(i) Let $\wt \A$ be a dissipative extension of $\A$. Since every dissipative extension of $\A$ is admissible (see \cite{K77} or \cite[Lemma 7.13]{DM17}), one sees that $\wt \A$ has the form $\A_\Th$. The dissipativity of $\Th$  follows from the dissipativity of $\A_\Th$ and condition (M2) of Definition \ref{d:MBT}.

(ii) It is obvious from the definition of m-boundary tuple and Definition \ref{d:Th} that $\A_\Th$ is a maximal dissipative operator exactly when $\Th$ is a maximal dissipative linear relation from $\Hs_{-,+}$ to $\Hs_{+,-}$. 
By the result of Phillips \cite{P59}, the densely defined operator $\A_\Th$ is maximal dissipative exactly when it is m-dissipative. This completes the proof of (ii).
\end{proof}

\begin{rem} \label{r:OpRel} 
(i)  $\A_\Th = (\A_\Th)^*$  if and only if $\Th = \Th^\#$ (due to $
(\A_\Th)^* = \A_{\Th^\#} 
$).

(ii) It is easy to see from Definitions \ref{d:MBT} and \ref{d:Th}, Remark \ref{r:m-D}, and Corollary \ref{c:MBTm-D}
that an operator $\A_\Th$ is (maximal) dissipative, symmetric, accretive,  if and only if $\Th$ is so as a linear relation from $\Hs_{-,+}$ to $\Hs_{+,-}$.
\end{rem}

Let $\Xi$ be a dissipative linear relation in $\Hs$. Then it is easy to see that there exists  a unique operator $\Cay_\Xi:\dom \Cay_\Xi \subseteq \Hs \to \Hs$ such that 
\begin{gather} \label{e:Cay}
\dom \Cay_\Xi = \{ h' - \ii h \ : \ \{h,h'\} \in \Xi\} , \quad \Cay_\Xi (h' - \ii h) = h' + \ii h \text{ for all } \{h,h'\} \in \Xi .
\end{gather}
Indeed, the dissipativity of $\Xi$ implies the following statement:
\begin{gather} \label{e:f=hj-hj}
\text{if $f=h'_j - \ii h_j$ for $\{h_j,h'_j\} \in \Xi$, $j=1,2$, then $h_1=h_2$ and $h_1'=h_2'$.}
\end{gather}

The operator $\Cay_\Xi$ is called \emph{the Cayley transform of $\Xi$},  cf. \cite{AG,DM17,BHdS20}. It is easy to see that $\Cay_\Xi$ is a contractive operator  in $\Hs$ for every dissipative $\Xi$   \cite{K75,K77} (see also \cite{DM17,BHdS20}). 
If $T$ is a dissipative operator in $\Hs$, then the Cayley transform of $T$ is defined as $\Cay_{\Gr T}$ and one has $\Cay_{\Gr T} = (T+\ii I_\Hs)(T-\ii I_\Hs)^{-1}$.

Let $T$ be an arbitrary contractive operator in $\Hs$. Then it is easy to see that 
\begin{multline} \label{e:inverseCayley}
\text{the linear relation  $\Xi = \Bigl\{ \{(T-I_\Hs)f, \ii (T+I_\Hs)f \} \ : \ f \in \dom T \Bigr\}$ is dissipative } \\ 
\text{and that $\Xi$ is \emph{the inverse Cayley transform of $T$} in the sense $\Cay_\Xi=T$.}
\end{multline}

\begin{lem} \label{l:OpExt} 
Let $V$ be a linear homeomorphism from $\Hs$ onto $\Hs_{-,+}$,
let $\Th$ be a linear relation from $\Hs_{-,+}$ to $\Hs_{+,-}$.
Let $\Th_V = \bigl( \begin{smallmatrix} V^{-1} & 0 \\
0 & V^\# \end{smallmatrix} \bigr) \Th$ be the linear relation in $\Hs$ defined by
$\Th_V := 
\{ \{V^{-1} h_{-,+}, V^\# h_{+,-} \} \ : \ \{ h_{-,+},h_{+,-}\} \in \Th \} .$
Then:
\item[(i)] $(V^\#)^{-1} = (V^{-1})^\#$ (note that this is a  homeomorphism from $\Hs$ onto $\Hs_{+,-}$).
\item[(ii)] $\Th$ is $\Hs$-pairing-selfadjoint, symmetric, nonnegative,  dissipative,  closed exactly when $\Th_V$ in $\Hs$  is  selfadjoint, symmetric, nonnegative, dissipative, closed, respectively.
\item[(iii)] An operator $T:\dom T \subseteq \Hs_{-,+} \to \Hs_{+,-}$ is $\Hs$-pairing-selfadjoint, nonnegative,  dissipative,
closed, closable exactly when the operator $V^\# T V$ in $\Hs$ is so.
\item[(iv)] $(\Th^\#)_V =(\Th_V)^*$
\item[(v)] $\Th$ is symmetric if and only if $\Th \subseteq \Th^\#$.
\item[(vi)] The following three statements are equivalent:
{
\subitem(vi.a) $\Th$ is maximal dissipative;
\subitem(vi.b) $\Th_V$ is maximal dissipative;
\subitem(vi.c) there exists a contraction $K$ on $\Hs$ 
with the property that 
}
\begin{gather} \label{e:abcK0}
\text{$\Th$ consists of $\{h_{-,+},h_{+,-}\} $ such that }  (K+I_\Hs) V^{-1} h_{-,+}   + \ii (K-I_\Hs) V^\# h_{+,-} = 0 .
\end{gather}
\item[(vii)] If (vi.a)-(vi.c) hold true, then $K = \Cay_{\Th_V}$. Besides, $\Th$ is $\Hs$-pairing-selfadjoint if and only if $K$ is a unitary operator on $\Hs$.
\item[(viii)] Assume that $\Th$ is dissipative. Then $\Psi$ is a maximal dissipative extension of $\Th$ if and only if $\Psi_V$ is an inverse Cayley transform of a contraction $K$ on $\Hs$ satisfying 
$\Gr K \supseteq \Gr \Cay_{\Th_V}$. 
\end{lem}
\begin{proof}
The equalities
$
\< h_{-,+}| h_{+,-}\>_\Hs  
= \< h_{-,+}| (V^{-1})^\# V^\# h_{+,-}\>_\Hs 
$
and $(V^{-1} h_{-,+}| V^\# h_{+,-} )_\Hs = \< h_{-,+}| h_{+,-}\>_\Hs $ are valid for all $h_{\mp,\pm} \in \Hs_{\mp,\pm}$. They imply statements (i) and (ii).
Statements (iii) and (iv) follow from (i) and (ii) combined with Definition \ref{d:Th} and Remark \ref{r:AThvsTh}. Statement (v) follows from (ii), (iv), and the fact that $\Th_V$ is symmetric if and only if $\Th_V \subseteq (\Th_V)^*$. 
The equivalence (vi.a)$\Leftrightarrow$(vi.b) follows from (iii). The equivalence (vi.b)$\Leftrightarrow$(vi.c) is the first part of \cite[Theorem 1]{K75} (the proof of \cite[Theorem 1]{K75} can be found in \cite{DM17}).

(vii) Let (vi.a)-(vi.c) be true. Then $K$ is a contraction on $\Hs$ (in particular, $\dom K = \Hs$). By (\ref{e:inverseCayley}), the inverse Cayley transform $\Xi$ of $K$ has the properties that $\Xi$ is dissipative and that $(K+I_\Hs) h = - \ii (K-I_\Hs) h'$ for all $\{h,h'\} \in \Xi$. This and (vi.c) implies $\Th_V \supseteq \Xi$. However,
(\ref{e:f=hj-hj}) and $\dom K = \Hs$ imply that $\Xi$ is maximal dissipative. Thus, $\Xi = \Th_V$ and 
$K=\Cay_{\Th_V}$.

(viii) It follows from (\ref{e:f=hj-hj}) that two dissipative linear relations $\Xi_1$ and $\Xi_2$ in $\Hs$
satisfy $\Xi_1 \subset \Xi_2$ if and only if $\Gr \Cay_{\Xi_1} \subset \Gr \Cay_{\Xi_2}$. This, (ii), and (vii) imply (viii).
\end{proof}

\begin{cor} \label{c:Cayley} 
Let $\A_\Psi$ be a dissipative extension of $\A$. Then $\A_\Th$ is an m-dissipative extension of $\A_\Psi$ if and only if $\Th$ consists of all $\{h_{-,+},h_{+,-}\} \in \Hs_{-,+} \oplus \Hs_{+,-}$ satisfying (\ref{e:abcK0}) with a certain contraction $K$ on $\Hs$ such that 
$\Gr K \supseteq \Gr \Cay_{\Psi_V}$.
\end{cor}
\begin{proof} The statement follows from Remark \ref{r:AThvsTh}, Corollary \ref{c:MBTm-D}, and Lemma \ref{l:OpExt} (vi)-(viii).
\end{proof}

\begin{rem} \label{r:m-D2}
An operator $T:\dom T \subseteq \Hs \to \Hs$ is m-dissipative (m-accretive) if and only if it is closed and maximal dissipative (resp., maximal accretive) \cite{P59,Kato13}. 
\end{rem}

\begin{rem} \label{r:m-D3} 
A densely defined operator $T:\dom T \subseteq \Hs_{-,+} \to \Hs_{+,-}$ is maximal dissipative (maximal accretive) if only if $\Gr T$ is so. This follows from Lemma \ref{l:OpExt} (ii)-(iii) and Remark \ref{r:m-D} applied to $V^\# T V$.
\end{rem}

\subsection{M-dissipative boundary conditions for abstract Maxwell operators} \label{s:ThGen}

Let $\Hs_{0}$, $\Hs_{1}$, and $\Hs$ be Hilbert spaces.
Operators $T_j:\dom T_j \subseteq \Hs_j \to \Hs$, $j=0,1$, and the condition  
\begin{gather} \label{e:Trel}
\text{ $T_0 f_0 +  T_1 f_1 = 0$}
\end{gather}
define a linear relation $\ker (T_0 \ T_1)$ from $\Hs_0$ to $\Hs_1$ consisting of all $\{ f_0,f_1\} \in \Hs_0 \oplus \Hs_1$ such that 
$f_j \in \dom T_j$, $j=0,1$, and (\ref{e:Trel}) is fulfilled.
Assume that $\A$ is a symmetric operator satisfying the settings of Section \ref{s:MBT-BT}.  Putting $\Hs_0 = \Hs_{-,+}$ and $\Hs_1 = \Hs_{+,-}$, we see that a linear relation $\Th=\ker (T_0 \ T_1)$ defined by (\ref{e:Trel}) is associated with the admissible extension $\A_\Th$ of $\A$. Using   (\ref{e:ATh}), one can rewrite the domain of $\A_\Th$ as
the set of all $f \in \dom \A^*$ satisfying the  condition
$T_0 \Ga_0 f +  T_1 \Ga_1 f = 0 .$
This motivates the following definition: 
\begin{multline} \label{e:Athbc}
\text{the operator $\A_\Th$ with $\Th = \ker  (T_0 \ T_1)$ is called} \\ \text{the restriction of $\A^*$ defined by  the condition $T_0 \Ga_0 f +  T_1 \Ga_1 f = 0$.}
\end{multline}
This gives an abstract analogue of boundary conditions (cf. \cite{K75} and \cite[Chapter 6]{DM17}).

Let a closed densely defined symmetric  operator $A$ in a Hilbert space $\Hc$ and a symmetric operator 
$
M \bigl(\begin{smallmatrix} \psi_1 \\ \psi_2 \end{smallmatrix} \bigr)
=   
S^{-1} 
\bigl(\begin{smallmatrix}
 0 & \ii A \\  -\ii A & 0 
\end{smallmatrix} \bigr) \bigl(\begin{smallmatrix} \psi_1 \\ \psi_2 \end{smallmatrix} \bigr)
$
in the space $\Hc^{2,S}$ be connected by (\ref{e:M}).
Let $(\Hs_{-,+},\Hs, \Ga_0 ,\Ga_1)$
be an m-boundary tuple for $M^*$ (e.g., the m-boundary tuple of  Proposition \ref{p:M*MBT}).
Let $V$ be a certain fixed linear homeomorphism from $\Hs$ onto $\Hs_{-,+}$.


\begin{thm} \label{t:absM-dis} 
(i)  An extension $\wh M$ of $M$ is an m-dissipative operator  if and only if it is a restriction of $M^*$ defined by  an `abstract boundary condition' of the form
\begin{gather} \label{e:K+IGa}
(K+I_\Hs) V^{-1} \Ga_0 \psi_2 + \ii (K-I_\Hs) V^\#  \Ga_1 \psi_1 = 0 ,
\end{gather}
where $K:\Hs  \to \Hs$ is a certain contraction on $\Hs$.

\item[(ii)] Statement (i) establishes one-to-one correspondence between 
contractions $K$ on $\Hs$ and m-dissipative extensions $\wh M$ of $M$.
Besides, $\wh M$ is selfadjoint if and only if $K$ is unitary.
\end{thm}
\begin{proof}
By Corollary \ref{c:MBTm-D}, $\wh M$ is m-dissipative if and only if $\wh M = M_\Th$ for a certain 
maximal dissipative linear relation $\Th$ from $\Hs_{-,+}$ to $\Hs_{+,-}$. By Lemma \ref{l:OpExt} (vi), such linear relations $\Th$ have the form (\ref{e:abcK0}) with a certain contraction $K$ on $\Hs$.
Moreover, this establishes 1-to-1 correspondence between the family of maximal dissipative relations $\Th$ and 
the family of contractions $K$ on $\Hs$. 
Combining  (\ref{e:abcK0}), (\ref{e:ATh}), and  (\ref{e:Athbc}), we see that the extension $\wh M$ is defined by (\ref{e:K+IGa}). The correspondence between the case $\wh M = \wh M^*$ and unitary operators $K$ follows  from Remark \ref{r:OpRel} (i) and Lemma \ref{l:OpExt} (vii) .
\end{proof}

\section{Abstract impedance-type boundary conditions}
\label{s:AbstractITC}


For a linear relation $\Phi$ from $\Hs_1$ to $\Hs_2$, let us define 
$\big(\begin{smallmatrix}
   1 & 0 \\
0 & \mp \ii  
\end{smallmatrix}\big) \Phi := \{ \{h_1, \mp \ii h_2 \} \ : \ \{ h_1,h_2\} \in \Phi \}. $
The domain of $\Phi$ is defined by $\dom \Phi := \left\{h_1 \in \Hs_1 \ : \ \{h_1,h_2\} \in \Phi \right\} $, the inverse linear relation $\Phi^{-1}$ from $\Hs_2$ to $\Hs_1$ by 
$\Phi^{-1}:= \left\{ \{h_2,h_1\}  \ : \ \{h_1,h_2\} \in \Phi \right\} $. Let $\Phi (h_1) := \left\{ h_2 \ : \ \{h_1,h_2\} \in \Phi \right\}$. So  $\Phi (0) := \left\{ h_2 \ : \ \{0,h_2\} \in \Phi \right\}$ is the multivalued part of $\Phi$.

Using the abstract settings of Section \ref{s:mrt}, we consider 
a closed densely defined symmetric operator $A$ in a Hilbert space $\Hc$ and 
an associated symmetric  abstract Maxwell operator $M$ defined in the `weighted' Hilbert space $\Hc^{2,S}$ by 
$
 M  \psi 
:=   
S^{-1} 
\left(\begin{smallmatrix}
 0 & \ii A \\  -\ii A & 0 
\end{smallmatrix}\right) \bigl(\begin{smallmatrix} \psi_1 \\ \psi_2 \end{smallmatrix}\bigr).
$

\subsection{Boundary conditions with impedance-type operators} \label{s:I-T}

Let $(\Hs_{-,+},\Hs,  \Ga_0 , \Ga_1)$ be the m-boundary tuple for $M^*$ associated by Proposition \ref{p:M*MBT} with a reduction tuple $(\Hs_{-,+},\Hs,\Gb, W)$ for $A^*$ (recall that $\Ga_0 \psi := \Gb \psi_2 $ and $\Ga_1 \psi := \ii W  \Gb \psi_1$). 

Consider an abstract boundary condition of the form 
\begin{gather} \label{e:I-TC}
\ii Z \Ga_0 \psi + \Ga_1 \psi = 0 
\end{gather}
with an accretive operator $Z:\dom Z \subseteq \Hs_{-,+} \to \Hs_{+,-}$. Recall that, by definition, $Z$ is accretive if $\re \<Z h|h\>_\Hs \ge 0$, $h \in \Hs_{-,+}$ 
(we do not assume that the operator $Z$ is closable or densely defined in $\Hs_{-,+}$). According to (\ref{e:ATh}), the operator $M_{\Gr (-\ii Z)}$ is a restriction of $M^*$ defined by condition (\ref{e:I-TC}). 

\begin{defn} \label{d:ImpOp}
 Accretive operators $Z:\dom Z \subseteq \Hs_{-,+} \to \Hs_{+,-}$ (accretive linear relations $\Phi$ from $\Hs_{-,+}$ to $\Hs_{+,-}$) will be called \emph{impedance-type operators (resp., impedance-type relations)} associated with the m-boundary tuple $(\Hs_{-,+},\Hs,  \Ga_0 , \Ga_1)$. 
\end{defn}

 If $Z$ is an impedance-type operator, we say that condition (\ref{e:I-TC}) and the restriction $M_{\Gr (-\ii Z)}$ of $M^*$ are \emph{generated by the impedance-type operator} $Z$. 
If $\Phi$ is an impedance-type relation, we say that 
\emph{the restriction $M_{\left(\begin{smallmatrix}
   1 & 0 \\
0 & - \ii  
\end{smallmatrix}\right) \Phi}$ of $M^*$ is generated by the impedance-type relation $\Phi$}.

The class of conditions (\ref{e:I-TC}) generated by impedance-type operators  includes the class of the generalized impedance boundary conditions of  \cite[Sections 1.6.1]{ACL17} and the class of Leontovich-type boundary conditions, see Sections \ref{s:Mresults} and \ref{s:mdM}.

In this subsection, we assume that $\wt M=M_{\Gr (-\ii Z)}$ is a restriction of $M^*$ generated by a certain impedance-type operator $Z$ and that $V$ is a certain linear homeomorphism from $\Hs$ to $\Hs_{-,+}$. 
Note that $\Gr Z$ is an impedance-type relation, and that the linear relation $\Gr (-\ii Z) = \left(\begin{smallmatrix}
   1 & 0 \\
0 & - \ii  
\end{smallmatrix}\right) \Gr Z$ is dissipative. So the operator $\wt M$ is a dissipative (but not necessarily m-dissipative). 

\begin{lem} \label{l:I-Ta}
Let $\Phi$ be a impedance-type relation, let $\Phi_V:= \big(\begin{smallmatrix}
   V^{-1} & 0 \\
0 & V^\#  
\end{smallmatrix}\big) \Phi$ (see Lemma \ref{l:OpExt}),
and $\Th := \big(\begin{smallmatrix}
   1 & 0 \\
0 & - \ii  
\end{smallmatrix}\big) \Phi$. Then  the following statements are equivalent:
(i) $M_\Th $ is m-dissipative, (ii)~$\Th$ is maximal dissipative,
(iii) $\Phi$ is maximal accretive, (iv) $\Phi_V$ is maximal accretive, (v)~$\Th_V$ is maximal dissipative.
\end{lem}
\begin{proof} Lemma \ref{l:OpExt} (ii) and Corollary \ref{c:MBTm-D} give  (i) 
$\Leftrightarrow$ (ii) $\Rightarrow$ (iii) $\Rightarrow$ (iv) $\Rightarrow$ (v) $\Rightarrow$ (ii).
\end{proof}

\begin{rem} \label{r:I-T}
If, in the settings of Lemma \ref{l:I-Ta}, $\Phi$ is additionally nonnegative, then each of the statements (i)-(v) is equivalent to each of the following statements:
(vi) $\Phi_V = (\Phi_V)^*$, (vii) $\Phi=\Phi^\#$.

Indeed, (vi) $\Leftrightarrow$ (vii) by Lemma \ref{l:OpExt} (ii). The equivalence (iv) $\Leftrightarrow$ 
(vi) for nonnegative $\Phi_V$ is well known and can be obtained from the equivalence (iv) $\Leftrightarrow$ (v) and the consideration of the domain of the Cayley transform of $\big(\begin{smallmatrix}
   1 & 0 \\
0 & - \ii  
\end{smallmatrix}\big) \Phi_V = \Th_V$. Note that $\Th_V$ is maximal dissipative if and only if $\dom \Cay_{\Th_V} = \Hs$ (see e.g. Lemma \ref{l:OpExt} (vi)-(vii)).
\end{rem}

\begin{cor} \label{c:I-T}
Consider an arbitrary operator $Z:\dom Z \subseteq \Hs_{-,+} \to \Hs_{+,-}$.  Then the following statements are equivalent: 
(i) $\wt M = M_{\Gr (-\ii Z)}$ is m-dissipative,
(ii) $Z$ is maximal accretive and closed, (iii) $V^\# Z V$ is m-accretive in $\Hs$.
\end{cor}
\begin{proof}
The corollary follows from Corollary \ref{c:MBTm-D}, Lemma \ref{l:I-Ta}, and Remark \ref{r:m-D2}.
\end{proof}
\begin{rem} \label{r:essMDI-T}
For an arbitrary operator $Z:\dom Z \subseteq \Hs_{-,+} \to \Hs_{+,-}$, the following statements are equivalent: 
(i) the operator $M_{\Gr (-\ii Z)}$ is essentially m-dissipative,
(ii) the closure $\overline{\Gr Z}$ is a maximal accretive linear relation from $\Hs_{-,+}$ to $\Hs_{+,-}$,
(iii) $\overline{\Gr (V^\# Z V)}$ is a maximal accretive linear relation in $\Hs$. Indeed, let $\Th = \Gr (-\ii Z)$. Then, by Remark \ref{r:AThvsTh} and Corollary \ref{c:MBTm-D}, $ M_{\Gr (-\ii Z)}$ is essentially m-dissipative if and only if the closure $\overline{\Th}$ is maximal dissipative. Lemma \ref{l:I-Ta} completes the proof of the desired equivalences.
\end{rem}

If $\overline{\Gr Z} \neq  (\overline{\Gr Z})^\# $ or if it is difficult to check the condition 
$\overline{\Gr Z} =  (\overline{\Gr Z})^\# $ for a particular example, one can construct m-dissipative extensions of $M_{\Gr (-\ii Z)}$.


\begin{thm} \label{t:I-Text}
For an impedance-type operator $Z$, the following statements are equivalent:
\item[(i)] An operator $\wh M$ is an m-dissipative extension 
of $\wt M = M_{\Gr (-\ii Z)}$.
\item[(ii)] $\wh M = M_\Th$, where $\Th = \left(\begin{smallmatrix}
   1 & 0 \\
0 & - \ii  
\end{smallmatrix}\right) \Phi$ for a certain maximal accretive extension $\Phi$ of $\Gr Z$.
\item[(iii)]  $\wh M$ is the restriction of $M^*$ defined by a condition  
$(I_{\Hs}+K) V^{-1} \Ga_0 \psi  - \ii (I_{\Hs}-K) V^\#  \Ga_1 \psi = 0 $
with a certain contraction $K$ on $\Hs$ such that $K$ is an extension of the Cayley transform $\Cay_{\Gr (-\ii V^\#ZV)}$ associated with the dissipative operator $(-\ii) V^\#ZV$.
\end{thm}
\begin{proof}
The equivalence (i) $\Leftrightarrow$ (ii) follows from Corollary \ref{c:MBTm-D} and Lemma \ref{l:I-Ta}.
The equivalence (ii) $\Leftrightarrow$ (iii) follows from Lemmata \ref{l:I-Ta} and \ref{l:OpExt} (viii).
\end{proof}

According Definition \ref{d:contraction}, it is assumed in Theorem \ref{t:I-Text} (iii) that $\dom K = \Hs$, while the contractive in $\Hs$ operator $ \Cay_{\Gr (-\ii V^\#ZV)}$ is not necessarily defined on the whole $\Hs$. Actually $\dom  \Cay_{\Gr (-\ii V^\#ZV)}= \Hs$ if and only if $M_{\Gr (-\ii Z)}$ is m-dissipative (see Corollary \ref{c:I-T} (iii)).

\subsection{Leontovich-type boundary conditions in abstract settings}
\label{s:GIBCa}

The Leontovich-type boundary condition (\ref{e:GIBC}) is written in  terms of Calkin's triple \linebreak $(\LLt, \pi_{\top,2},  \ -\n_\times)$, see Corollary \ref{c:curlRT}.  To apply the results of Section \ref{s:I-T} to (\ref{e:GIBC}), we  first provide a transition to the settings of associated reduction tuples and m-boundary tuples.  

 Let $(\Hs_{-,+},\Hs, \Gb, W)$ be a unitary reduction tuple for $A^*$ and $(\Hs,\wt \Gb,\wt W)$ be the associated Calkin triple constructed in Corollary  \ref{c:RTtoCT}. Recall that this means that 
$\dom \wt \Gb :=\{f \in \dom A^* \ : \ \Gb f \in \Hs \} $ and $\wt \Gb f = \Gb f$ for all $f \in \dom \wt \Gb
$, 
where $\wt \Gb$ is understood as a densely defined operator from $\dom A^*$ to $\Hs$.
The unitary operator $\wt W:\Hs \to \Hs$ is the closure in $\Hs$ of the restriction $ W \uph_{\wt \Hs_{-,+}}$ (recall that  
$\wt \Hs_{\mp,\pm} = \Hs_{\mp,\pm} \cap \Hs$).

Let us associate with the Calkin triple $(\Hs,\wt \Gb,\wt W)$ the `restricted  boundary maps' 
\[
\text{$ \wt \Ga_0 \psi = \wt \Gb \psi_2 $ and $\wt \Ga_1 \psi = \ii \wt W  \wt \Gb \psi_1$,}
\]
where $\wt \Ga_j$, $j=0,1$, are defined on all 
$\psi = \{\psi_1,\psi_2\} \in \dom M^*$ such that $\psi_{2-j} \in \dom \wt \Gb$. 
Taking the m-boundary tuple $(\Hs_{-,+},\Hs,  \Ga_0 , \Ga_1)$ associated by Proposition \ref{p:M*MBT} with the aforementioned reduction tuple, we get 
$
 \Ga_0 \psi =  \Gb \psi_2 , \   \Ga_1 \psi = \ii  W   \Gb \psi_1.
$ 
So 
$\wt \Ga_j$ are restrictions of $\Ga_j$, $j=0,1$.

Through this subsection we assume that $\wh Z$ is  an accretive operator in $\Hs$. 
An abstract version of a Leontovich-type boundary condition can be introduced as 
\begin{gather} \label{e:aGIBC}
\ii \wh Z \wt \Ga_0 \psi + \wt \Ga_1 \psi = 0  \qquad \text{(it can be rewritten as $ \wh Z \wt \Gb \psi_2 + \wt W  \wt \Gb \psi_1 = 0$).} 
\end{gather}

\begin{defn} \label{d:MimpZ}
The  \emph{abstract version $M_{\imp,\wh Z}$ of Leontovich-type operator associated with $\wh Z$} is the restriction  of $M^*$ to 
$
\dom M_{\imp,\wh Z} := \{ \psi\in \dom M^* \cap \dom \wt \Ga_0 \cap  \dom \wt\Ga_1 \ : \ \text{ (\ref{e:aGIBC}) is satisfied} \}. 
$
\end{defn}

In the case of Calkin's triple $(\LLt,\pi_{\top,2},-\n_\times)$ for $\curlm$ and the multiplication operator $\wh Z = \Mul_{\Z}$, one sees that (\ref{e:aGIBC}) becomes the Leontovich-type boundary condition (\ref{e:GIBC}).

Our aim now is to write (\ref{e:aGIBC}) in the form $\ii Z \Ga_0 \psi + \Ga_1 \psi = 0 $ of Section \ref{s:I-T} with a certain impedance-type operator 
$Z$.  
We define $Z$ by 
\begin{gather} \label{e:domZ}
\text{$Z h = \wh Z h $ for all $h$ belonging to } \dom Z := \{ h \in \wt \Hs_{-,+} \cap \dom \wh Z \ : \ 
\wh Z h \in \wt \Hs_{+,-}\} .
\end{gather}
Since $\wh Z$ is accretive in $\Hs$, we see that $Z$ is accretive as an operator from 
$ \Hs_{-,+} $ to $ \Hs_{+,-}$. So $Z$  is an impedance-type operator in the sense of Definition \ref{d:ImpOp}. 
Recall that  the operator $M_{\Gr (-\ii Z)}$ is the restriction of $M^*$ defined  by the  condition $\ii Z \Ga_0 \psi + \Ga_1 \psi = 0 $.

\begin{lem}  \label{l:M=M}
Condition (\ref{e:aGIBC}) is equivalent to the condition $\ii Z \Ga_0 \psi + \Ga_1 \psi = 0 $ (i.e., $M_{\imp, \wh Z} = M_{\Gr (-\ii Z)} $). In particular, $M_{\imp, \wh Z}$ is a dissipative extension of $M$.
\end{lem}
\begin{proof} The lemma follows from $\dom M_{\imp, \wh Z} = \dom M_{\Gr (-\ii Z)}$. Indeed, the restrictions imposed in (\ref{e:aGIBC}) by $\psi_j \in \dom \wt \Gb = \Gb^{-1} \wt \Hs_{-,+}$, $j=1,2$,  are imposed in $Z \Gb \psi_2 +  W \Gb \psi_1 = 0 $ via the more narrow domain of $Z$, i.e., 
$\Gb \psi_2 \in \dom Z \subset \wt \Hs_{-,+}$ and $ \Gb \psi_1 \in W^{-1} \ran Z  \subset W^{-1} \wt \Hs_{+,-} = \wt \Hs_{-,+}$ (we use the facts that $(\Hs_{-,+},\Hs,\Gb,W)$ is a unitary reduction tuple and, consequently, $W^{-1} \wt \Hs_{+,-} = \wt \Hs_{-,+}$).
Thus, $M_{\imp, \wh Z} = M_{\Gr (-\ii Z)} $ and this operator is  dissipative (see Section \ref{s:I-T}). 
\end{proof}

It follows from Proposition \ref{p:MixOrder}   that there exists a unique positive selfadjoint operator 
$\Soin$ in $\Hs$ with $\dom \Soin = \wt \Hs_{+,-} $ such that $\| h \|_{\Hs_{+,-}} = \| \Soin h\|_{\Hs}$ for all $h \in \wt \Hs_{+,-}$. Since $\Uop_{\Hs_{+,-} \to \Hs} = \Uop_{\Hs \to \Hs_{-,+}}^\#$ (by Lemma \ref{l:propertiesHadj}),
the operator $\Soin$ satisfies
\begin{gather} \label{e:Spm}
\Soin \ = \ \Sop^{-1} \ = \ \Uop_{\Hs \to \Hs_{-,+}} \! \uph_{\wt \Hs_{+,-}} \ = \ \Uop_{\Hs_{+,-} \to \Hs} \uph_{\wt \Hs_{+,-}} \ = \ \Uop_{\Hs \to \Hs_{-,+}}^\# \! \uph_{\wt \Hs_{+,-}} .
\end{gather}
 
\begin{thm}[criterion of m-dissipativity] \label{t:aIBCm-d}
Let $\wh Z$  be an accretive operator in $\Hs$. Then:
\item[(i)] The operator $M_{\imp, \wh Z}$ is m-dissipative (is essentially m-dissipative) if and only if the operator $\Soin \wh Z \Soin$ is m-accretive in $\Hs$ (resp., essentially m-accretive in $\Hs$). 
\item[(ii)] Assume additionally that $\wh Z $ is nonnegative. Then $M_{\imp, \wh Z}$ is m-dissipative (is essentially m-dissipative) if and only if $\Soin \wh Z \Soin$ is selfadjoint (resp., essentially selfadjoint) in $\Hs$.
\end{thm}

Theorem \ref{t:aIBCm-d} follows immediately from Corollary \ref{c:I-T} and Remarks \ref{r:I-T}--\ref{r:essMDI-T} applied to the operator $M_{\imp, \wh Z} = M_{\Gr (-\ii Z)} $ with the use of the following lemma. 

\begin{lem}  \label{l:SZS=UZU}
$\Soin \wh Z \Soin =  \Uop_{\Hs_{+,-} \to \Hs} Z \Uop_{\Hs \to \Hs_{-,+}}$ 
\end{lem}
\begin{proof} 
Since $\dom Z \subseteq \wt \Hs_{-,+} $ and $\ran Z \subseteq \wt \Hs_{+,-}$, the mapping $h \mapsto Zh$ can be treated both as an operator from $\Hs_{-,+}$ to $\Hs_{+,-}$ and as an operator in $\Hs$. In turn, $ \Uop_{\Hs_{+,-} \to \Hs} Z \Uop_{\Hs \to \Hs_{-,+}} = \Soin Z \Soin$. Since $Z$ is a restriction of $\wh Z$ to the domain given by  (\ref{e:domZ}), the fact that $\Soin$ is a bijection from $\dom \Soin = \wt \Hs_{+,-}$ onto $\wt \Hs_{-,+}$ (see Proposition \ref{p:MixOrder}) implies $\Soin Z \Soin = \Soin \wh Z \Soin$. This completes the proofs of Lemma \ref{l:SZS=UZU} and 
Theorem \ref{t:aIBCm-d}.
\end{proof}

\begin{cor} \label{c:sectorialL}
Assume that a bounded accretive operator $\wh Z:\Hs \to \Hs$ and a constant $C_1>0$ satisfy 
$C_1 \|h\|^2 \le |(\wh Z h|h)_\Hs| $ for all $h\in \Hs$.  
Then $M_{\imp, \wh Z}$ is m-dissipative.
\end{cor}
\begin{proof} 
Since $\Soin = \Soin^*$ and $\wh Z$ is accretive, $\Soin \wh Z \Soin$ is also accretive. Due to Theorem \ref{t:aIBCm-d}, it is enough to prove that $\Soin \wh Z \Soin + I_\Hs$ is boundedly invertible in $\Hs$, which can be done applying the Lax–Milgram theorem (e.g.,  \cite{ACL17}) to 
the form 
$\bfr (u,v) := (\wh Z \Soin u | \Soin v)_\Hs + ( u |  v)_\Hs$ on the Hilbert space $(\wt \Hs_{+,-}, \|\cdot\|_{\wt \Hs_{+,-}})$ with $\| h \|_{\wt \Hs_{+,-}}^2 = \| \Soin h\|_\Hs^2 +\|h\|_\Hs^2 $. 
\end{proof}

\section{Applications to m-dissipative Maxwell operators}
\label{s:mdM}

Let $\M$ be the symmetric Maxwell operator  in $ \LL^2_{\epb,\mub} (\Om)$ defined by
(\ref{e:iM}). Then the operator $\M^* = \left( \begin{smallmatrix} 0 & \ii \epb^{-1} \curlm \\
-\ii \mub^{-1} \curlm & 0 \end{smallmatrix} \right) $ is defined on $\dom \M^* = (\HH (\curlm,\Om))^2$.

\begin{thm} \label{t:MaxwellBT}
(i) Let us define  $\pi_\top^{\hf} \{\E,\H\} := \pi_\top \H$ and $\ga_\top^{\ef} \{\E,\H\} := \ga_\top \E$. Then \begin{gather} 
(\Hs_{-,+},\Hs, \Ga_0 ,\Ga_1) = (\HH^{-1/2} (\curl_{\pa \Om}),\LLt , \pi_\top^{\hf}, \   \ii \ga_\top^\ef ) \text{ is an
m-boundary tuple for $\M^*$.}\label{e:MBTM*}
\end{gather}

\item[(ii)] In the settings of (i), an operator $\wt \M$ is an m-dissipative extension of the symmetric Maxwell operator $\M$ if and only if $\wt \M = \M_\Th$ for a  certain maximal dissipative linear relation $\Th$ from $\Hs_{-,+} = \HH^{-1/2} (\curl_{\pa \Om})$ to $\Hs_{+,-} = \HH^{-1/2} (\Div_{\pa \Om})$. (According to (\ref{e:ATh}), this means that $\M_\Th := \M^* \uph_{\dom \M_\Th}$ with $\dom \M_\Th := \{ \{\E,\H\} \in \dom \M^* \ : \  \{\pi_\top \H,  \ii \ga_\top \E\} \in \Th \}$). 
\end{thm}
\begin{proof} (i) follows from Corollary \ref{c:curlRT} and Proposition \ref{p:M*MBT},  (ii) from  Corollary \ref{c:MBTm-D}.\end{proof}

 
An equivalent representation of m-dissipative extensions $\wt M$ in terms of boundary conditions is provided by Theorem \ref{t:m-dMax}, which is an immediate corollary of  Theorem \ref{t:absM-dis}.

According to Section \ref{s:I-T}, an operator $\Zc$ is called an \emph{impedance-type operator associated with the m-boundary tuple} (\ref{e:MBTM*})
if $\Zc$ is an accretive operator from $\HH^{-1/2} (\curl_{\pa \Om})$ to $\HH^{-1/2} (\Div_{\pa \Om})$.
Then the boundary conditions (\ref{e:I-TC}) can be written as $
\Zc \pi_\top \H  +  \ga_\top \E = 0 
$
and all the results of Section \ref{s:I-T} are applicable to 
the Maxwell operator $\wt \M = \M_{\Gr (-\ii \Zc)}$ defined by this boundary condition.
In particular, Corollary \ref{c:I-T} implies Corollary \ref{c:Z}. 

The Leontovich-type boundary conditions or, more generally, the boundary condition 
\begin{gather} \label{e:hZBC2}
\wh \Zc \pi_{\top,2} \H + \ga_{\top,2} \E  = 0 , 
\end{gather}
with an arbitrary accretive operator $\wh \Zc:\dom \wh \Zc \subseteq \LLt \to \LLt$ can be put in the above framework in the way shown by Section \ref{s:GIBCa}. This means that 
we choose the reduction tuple $(\Hs_{-,+},\Hs, \Gb, W) = (\HH^{-1/2} (\curl_{\pa \Om}),\LLt, \pi_\top, \  -\n_\times^{\pi \to \ga})$ for the operator $\curlm$
and consider the associated Calkin triple 
$(\Hs,\wt \Gb,\wt W) = (\LLt, \pi_{\top,2},  \ -\n_\times)$
constructed in Section \ref{s:RT}. Proposition \ref{p:M*MBT} connects them with the m-boundary tuple (\ref{e:MBTM*}).
Then the m-dissipativity and essential m-dissipativity of the Maxwell operator $\M_{\imp,\wh \Zc}$ defined by (\ref{e:hZBC2}) are characterized by Theorem \ref{t:aIBCm-d}. Theorem \ref{t:aIBCm-d}  contains  
Theorem \ref{t:CritGIMC} (i) as its particular case since 
$\Soin = \Sdiv$ in the present settings. To get Theorem \ref{t:CritGIMC} (ii), one takes into account Remark \ref{r:I-T}.  

As an application, let us consider the following special class of operators $\wh \Zc$.

\begin{cor} 
\label{c:f(De)}
Assume that an accretive operator $\wh \Zc:\dom \wh \Zc \subseteq \LLt \to \LLt$ is diagonal w.r.t. the orthonormal basis $\{\vv_j \}_{j\in \ZZ\setminus\{0\}} \cup \{\vv_{0,j}\}_{j=1}^{b_1(\pa \Om)} $ of Theorem \ref{t:basis} and has a maximal possible domain of definition. That is, there exist sequences $\{\mu_j\}_{j\in \ZZ\setminus\{0\}} \subset \ii \overline{\CC}_-$ and 
$\{\mu_{0,j}\}_{j=1}^{b_1(\pa \Om)} \subset \ii \overline{\CC}_-$ such that $\wh \Zc \vv_j = \mu_j \vv_j$, $\wh \Zc \vv_{0,j} = \mu_{0,j} \vv_{0,j}$, and 
\[ \textstyle
\dom \wh \Zc = 
\left\{  \uu  = \sum_{j \in \ZZ\setminus\{0\}} \al_j \vv_j + \sum_{j=1}^{b_1 (\pa \Om)}
  \al_{0,j} \vv_{0,j} : \sum_{j \in \ZZ\setminus\{0\}} 
  | \al_j|^2 <\infty, \ \sum_{j \in \ZZ\setminus\{0\}} 
  |\mu_j \al_j|^2 <\infty
\right\}.
\]
Then the Maxwell operator $\M_{\imp,\wh \Zc}$ defined by (\ref{e:hZBC2}) is essentially m-dissipative. 
\end{cor}
\begin{proof}
We apply Theorem \ref{t:aIBCm-d} in the settings described above. In particular, the selfadjoint in $\Hs$ operator $\Soin = \Sdiv$ is diagonal w.r.t. the basis $\{\vv_j \}_{j\in \ZZ\setminus\{0\}} \cup \{\vv_{0,j}\}_{j=1}^{b_1(\pa \Om)} $, see (\ref{e:Sga0}) and Theorem \ref{t:basis}. Hence, the operator $\Tc = \Soin \wh \Zc \Soin = \Sdiv \wh \Zc \Sdiv$ is densely defined and its closure $\overline{\Tc}$ is a normal operator. Since $\wh \Zc$ and $\Tc$ are accretive, we see that $\overline{\Tc}$ is  m-accretive. Thus, Theorem \ref{t:aIBCm-d} implies that
$\M_{\imp,\wh \Zc}$ is essentially m-dissipative.
\end{proof}

This corollary and Theorem \ref{t:basis} imply the statement of Example \ref{ex:De}.

All m-dissipative extensions of the dissipative operator $\M_{\imp, \wh Z} $ defined by (\ref{e:hZBC2}) are described by Theorem \ref{t:I-Text} and Lemma \ref{l:M=M}. This and Theorem   \ref{t:Upi} prove Theorem \ref{t:Kz}.

Let us consider in more detail a particular case of Leontovich-type boundary condition  
with a nonnegative measurable function $\Z : \pa \Om \to [0,+\infty)$ 
as the impedance coefficient, see (\ref{e:GIBC}). 
Recall that by 
$\Lc_\Z$ we denote the associated Leontovich-type operator, i.e., $\Lc_\Z$ is a restriction of the  Maxwell operator 
$\M^* $ to the set of $\{\E,\H\} \in \dom \M^*$ satisfying 
$\Z (\x) \pi_{\top,2} \H (\x)+ \ga_{\top,2} \E (\x) = 0 $  a.e. on $ \pa \Om$. By (\ref{e:ATh}), there exists a linear relation $\Th$ from $\HH^{-1/2} (\curl_{\pa \Om})$ to $\HH^{-1/2} (\Div_{\pa \Om})$ such that $\Lc_\Z = \M_\Th$. It is easy to see that $\Th = \Gr (-\ii \wt \Mul_{\Z})$, where 
\begin{gather*} 
\wt \Mul_{\Z} := \Mul_\Z \uph_{\dom \wt \Mul_{\Z}}, \quad
 \dom \wt \Mul_{\Z} := \{ \uu \in \HH^{-1/2} (\curl_{\pa \Om}) \cap \dom \Mul_\Z  \ : \  
 \Z \uu \in \HH^{-1/2} (\Div_{\pa \Om}) \},
\end{gather*}
i.e., $\wt \Mul_{\Z} : \dom \wt \Mul_{\Z} \subset  \HH^{-1/2} (\curl_{\pa \Om}) \to \HH^{-1/2} (\Div_{\pa \Om}) $ is the restriction to $\dom \wt \Mul_{\Z}$ of the selfadjoint in $\LLt$ multiplication operator $\Mul_\Z: \uu \mapsto \Z \uu $.
Since $\< \wt \Mul_{\Z} \uu | \uu \>_{\LLT} = \int_{\pa \Om} \Z \|\uu (\x)\|_{\CC^3}^2 \ge 0$ for all $\uu \in \dom \wt \Mul_{\Z}$, Definition \ref{d:Th}  imply that 
$\wt \Mul_{\Z}$ is a nonnegative operator from $\HH^{-1/2} (\curl_{\pa \Om})$ to $\HH^{-1/2} (\Div_{\pa \Om})$, and so $\Gr (-\ii \wt \Mul_{\Z})$ is a dissipative  linear relation. By Remark \ref{r:OpRel} (ii), we see that $\Lc_\Z=\M_{\Gr (-\ii \wt \Mul_{\Z})} = \M_{\imp, \Mul_{\Z}}$ is dissipative.
 
It is known (see \cite{LL04}) that $\Lc_\Z$ is an m-dissipative operator if the impedance coefficient $\Z (\cdot)$ is non-degenerate (i.e., if $\Z$ and $1/\Z$ are $L^\infty$-functions).
Let us show that for a wide class of degenerate impedance coefficients $\Lc_\Z$ is not m-dissipative.

For an open set $\Es$ in the Lipschitz manifold $\pa \Om$, we introduce the set $H^1_{\comp} (\Es)$ that consists of all $H^1 (\pa \Om)$-functions with the support compactly included in $\Es$. The function space $H^s_0 (\Es)$ is defined for $0<s\le 1$ as the closure of $H^1_{\comp} (\Es)$ in $H^s (\pa \Om)$. 

\begin{prop}\label{p:non-m-dis}
Let $\Es$ be an open set in the Lipschitz manifold $\pa \Om$. Assume that $\Z (\x) = 0$ for a.a. 
$\x \in \Es$. Then the Leontovich-type operator $\Lc_\Z$ is not closed and is not m-dissipative. 
\end{prop}
\begin{proof}
We prove that $\ker \wt \Mul_{\Z}$ is not closed in $\HH^{-1/2} (\curl_{\pa \Om})$.  
Using the Hodge decompositions (\ref{e:HDL22}) and 
(\ref{e:HDcurl}),
one sees that $\ker \wt \Mul_{\Z}$ contains $ \grad_{\pa \Om} H^1_{\comp} (\Es) $. It follows from (\ref{e:HDcurl}) that the closure of $ \grad_{\pa \Om} H^1_{\comp} (\Es) $ in $\HH^{-1/2} (\curl_{\pa \Om})$ is  $\grad_{\pa \Om} H^{1/2}_0 (\Es)$. 
Using (\ref{e:HDL22}) and $ H^{1/2}_0 (\Es) \setminus H^1 (\pa \Om) \neq \emptyset$, we get  $\left(\grad_{\pa \Om} H^{1/2}_0 (\Es) \right)\setminus \LLt \neq \emptyset$. 
The definition of $\wt \Mul_{\Z}$ gives  $\ker \wt \Mul_{\Z} \subset \LLt$.
Thus, $\ker \wt \Mul_{\Z}$  is not closed in $\HH^{-1/2} (\curl_{\pa \Om})$. Hence, the operator $(-\ii)\wt \Mul_{\Z} $ is not closed and so its graph $\Gr (-\ii \wt \Mul_{\Z})$ is not closed.
The equality $\Lc_\Z = \M_{\Gr (-\ii \wt \Mul_{\Z})}$ and Remark \ref{r:AThvsTh} imply that  $\Lc_\Z$ is not closed and is not m-dissipative.
\end{proof}

In the next section, we consider specific m-dissipative extensions of 
Leontovich-type operators $\Lc_\Z$ with degenerate nonnegative impedance coefficients $\Z$.

\section{Discussion: K/F-extensions,  randomization, and 
fat fractals}
\label{s:discussion}

Let us compare various approached to the study of  the mix of the conservative condition $\n \times \E =0$ on a part $\pOmp$ of $\pa \Om$ and  a Leontovich condition on $\pOma = \pa \Om \setminus \pOmp$.

Kapitonov \cite{K94}  extends the  boundary mapping  
$\Z (\x) \pi_{\top} \uph_{C^1 (\overline{\Om})} \H (\x)+ \ga_{\top} \uph_{C^1 (\overline{\Om})} \E (\x)$ defined originally on $(C^1 (\overline{\Om}))^2$ to a continuous mapping $\ka: (\HH (\curlm,\Om))^2 \to \HH^{-1/2}(\pa \Om)$ and showed that the boundary condition $\ka (\{\E,\H\}) = 0$ is m-dissipative under the assumptions that  $\pa \Om$ is of $C^\infty$-regularity, $\Z (\cdot) \in C^1 (\pa \Om)$, and $\re \Z (\x) \ge 0$ a.e. in $\pa \Om$. 

The condition $\n \times \E =0$ on $\x \in \pOmp$ corresponding to $\Z (\x)=0$ is combined often with the assumption $0<c_1 \le \Z (\x) \le c_2$
of the type of \cite{LL04} but imposed only on $\pOma$ instead of whole $\pa \Om$. The well-posedness of  Maxwell systems with  such mixed boundary conditions was addressed for polyhedral domains $\Om$ in \cite{M03,ACL17} (see also references therein)  
assuming that the impedance coefficient $\Z(\cdot)$ and the part $\pOma$ of $\pa \Om$ possess certain good enough properties, e.g., in \cite[Sections 5.1.2 and  7]{ACL17} it is assumed that  $\Z$ is a positive constant on $\pOma$ and 
 $\pOma$ has a piecewise smooth boundary and trivial topology. 

Below we give several examples, where $\Z (\cdot)$, $\pOmp$, and $\pOma$ do not fit to the above assumptions, but we are able to associate with corresponding Leontovich-type boundary conditions  specific m-dissipative Maxwell operators $\wh M$ using Theorem \ref{t:Kz} and special extensions of the boundary operator  $\Sdiv \Mul_\Z \Sdiv$ (the m-dissipativity then immediately implies the well-posedness of the evolution Maxwell system $\pa_t \mathbf{\Psi} = - \ii \wh M \mathbf{\Psi}$).

Let $\chi_\Es$ be the indicator function of $\Es$, i.e.  $\chi_\Es (\x) = 1$ if $\x \in \Es$, and $\chi_\Es (\x) = 0$ if $\x \not \in \Es$.  For an $\RR$-valued function $f$, let 
$[f (\x)]_+ := f (\x) \chi_{\{\x:f(\x)> 0\}} (\x)$.

\begin{ex}[fat-fractal boundary impedance] \label{ex:FF}
Let $\Ff$ be a certain ``fat fractal'' set  \cite{UF85} on the boundary $\pa \Om$ of the 3-D cube $ \Om = [0,1]^3$, i.e., $\Ff \subset \pa \Om$ is a fractal-like structure having positive surface measure. Assume also that $\pa \Om \setminus \Ff $ has a nonempty interior as a subset of the manifold $\pa \Om$. For example, to be concrete, we can denote by $\Ff_0 \subset [0,1]^2$ the Wallis sieve \cite{R93} and assume that  
$\Ff $ is the union of 6 copies of $\Ff_0$ placed to each of the facets of $ \Om = [0,1]^3$.
Then the Leontovich-type operator $\Lc_{\chi_\Ff}$ with the impedance coefficient $\Z (\x) =  \chi_\Ff (\x) $, $\x \in \pa \Om$, is dissipative, but, by Proposition \ref{p:non-m-dis}, is  not m-dissipative.
\end{ex}

\begin{ex}[randomized mixed boundary condition] \label{ex:random}
Let $(\bbOmega,\FF,\PP)$ be a probability space. We  introduce on the Lipschitz manifold $\pa \Om$ an $\RR$-valued random field $f_\om (\x) = f (\x,\om)$ with the property that $f_\om (\cdot) \in L^2 (\pa \Om) \ominus \KK_0$ almost surely (i.e., for a.a. $\om \in \bbOmega$ w.r.t. the measure $\PP$) using the following construction.
Recall that $b_0 = b_0 (\pa \Om) \in \NN $ is  the dimensionality of the space $\KK_0$ of locally constant functions, and let $\{v_{0,k}\}_{k=1}^{b_0} \cup \{v_k\}_{k=1}^{\infty}$ be an orthonormal basis in $L^2 (\pa \Om)$ composed of real-valued eigenfunctions of the Laplace-Beltrami operator $\De^{\pa \Om}$ in such a way that 
$\De^{\pa \Om} v_{0,k} = 0$ for $1 \le k \le b_0 $ and 
$\De^{\pa \Om} v_k = - \la_k^2 v_k$ for all $\la_k$, $k \in \NN$, defined as in (\ref{e:basisDe}).
Let $ \{\xi_k\}_{j\in \NN}$ be a sequence of independent centered random variables with finite variances $\Var (\xi_j)$ satisfying $0<\sum_{j\in \NN} \Var (\xi_j) < \infty$. We construct a random field with desired properties  
taking $ f_{\om} (\x) = 
\sum_{k \in \NN} \xi_j v_k (\x) $,
where the series is a.s.  
strongly convergent in $L^2 (\pa \Om)$ due to \cite[Theorem 3.2]{K93}. For modeling of the leakage of energy to an uncertain outer medium $\RR^3 \setminus \Om$, let us introduce the random impedance coefficient
$\Z_\om  := [f_{\om} (\x)]_+$. Then a dissipative Leontovich-type operator $\Lc_{\Z_\om}$ is defined for a.a. $\om \in \bbOmega$.
In the particular case of (nontrivial)
cut‐off fractional Gaussian fields $f_{\om} $ on $C^\infty$-manifold $\pa \Om$ \cite{R19}  
$f_{\om} (\cdot)$ is a.s. smooth,  an so the a.s. orthogonality to $\KK_0$ and Proposition \ref{p:non-m-dis} imply that 
$\Lc_{\Z_\om}$ is a.s. non-m-dissipative. Thus, an applied use of  $\Z_\om  := [f_{\om} (\x)]_+$ require a construction of m-dissipative extensions of $\Lc_{\Z_\om}$. 
\end{ex}

We conjecture that 
$\Lc_{\Z_\om}$ is non-m-dissipative with positive probability 
for arbitrary Lipschitz domain $\Om$ and arbitrary $ f_{\om}  = 
\sum_{k \in \NN} \xi_j v_k  $ with $\{\xi_k\}_{j\in \NN}$ satisfying the above conditions.
Indeed, Theorem \ref{t:CritGIMC} shifts the difficulty of the question of m-dissipativity of $\Lc_{[f_{\om}]_+}$ to the selfadjointness of nonnegative boundary operator $\Sdiv\Mul_{\Z} \Sdiv$ with $\Z = [f_{\om}]_+$, which requires an additional study. However, since $f_\om (\x)$ has a.s. zero average on every connected component $[\pa \Om]_j$ of $\pa \Om$, the zero set
of $\Z_\om (\cdot) = [f_{\om} (\cdot)]_+$ has a.s. positive measure on each $[\pa \Om]_j$,
and so is degenerate from points of view of  the Lax–Milgram theorem and of Proposition  \ref{p:non-m-dis}. Therefore, it is difficult to expect that  $\Lc_{[f_{\om}]_+}$ is m-dissipative with positive probability.

Let us show how to apply our approach to the construction of 
m-dissipative Maxwell operators associated with the impedance conditions of Examples \ref{ex:FF}-\ref{ex:random}. 
For $\Lc_{\chi_\Ff}$ and for $\Lc_{[f_{\om}]_+}$, all m-dissipative extensions are theoretically described by  Theorem \ref{t:Kz}. However, the applied modeling requires  concrete m-dissipative operators. Operator theoretical results \cite{Kato13,CdS78} allow us to construct concrete m-dissipative Maxwell operators using either the
 Friedrichs extension $[\Sdiv\Mul_{\Z} \Sdiv]_\Fr$, or the Krein-von Neumann extension $[\Sdiv\Mul_{\Z} \Sdiv]_\KN$ of the nonegative boundary operator $\Sdiv \Mul_{\Z} \Sdiv$ in $\LLt$. 

We consider here a rigorous procedure of the application of Friedrichs and Krein-von Neumann extensions to the abstract settings of Section \ref{s:I-T}. Let a dissipative extension $\wt M$ of the  abstract symmetric Maxwell operator $M$ be defined via the condition $ \ii Z \Ga_0 \psi + \Ga_1 \psi = 0 $ with a nonnegative impedance-type operator $Z:\dom Z \subseteq \Hs_{-,+} \to \Hs_{+,-}$, i.e., $\wt M = M_{\Gr (-\ii Z)}$. Let $V$ be a certain fixed linear homeomorphism from $\Hs$ to $\Hs_{-,+}$ (for Maxwell operators we use $V=\Upi^{-1}$). 
Then $V^\# Z V$ is a nonnegative operator (possibly nondensely defined) in the pivot space $\Hs$. Its graph 
$\Psi := \Gr (V^\# Z V)$ is a nonnegative linear relation in $\Hs$. The class of selfadjoint nonnegative linear relations in $\Hs$ that are extensions of $\Psi$ can be described as a closed interval w.r.t. the partial ordering  associated with corresponding quadratic forms \cite{DM95,HMS04,BHdS20}. 
The Friedrichs extension $[\Psi]_\Fr$ and the Krein-von Neumann extension $[\Psi]_{\KN}$ of $\Psi$ \cite{CdS78} are the greatest and, resp., the smallest elements of this interval.
 So $[\Psi]_{\Fr(\KN)}$  are nonnegative selfadjoint linear relations in $\Hs$. Then $\big(\begin{smallmatrix}
   1 & 0 \\
0 &  -\ii  
\end{smallmatrix}\big) [\Psi]_{\Fr (\KN)} $ are m-dissipative linear relations in $\Hs$ containing $\Gr (-\ii V^\# Z V)$,
see Remark \ref{r:I-T}. 
 
\begin{defn} \label{d:FK} 
If $\wh M $ is a restriction of $M^*$ defined by the condition 
\begin{gather} \label{e:K+IGa5}
(I_{\Hs}+K) V^{-1} \Ga_0 \psi  - \ii (I_{\Hs}-K) V^\#  \Ga_1 \psi = 0 
\end{gather}
with the operator $K$ equal to the Cayley transform of $\big(\begin{smallmatrix}
   1 & 0 \\
0 &  -\ii  
\end{smallmatrix}\big) [\Psi]_{\Fr } $ (of $\big(\begin{smallmatrix}
   1 & 0 \\
0 &  -\ii  
\end{smallmatrix}\big) [\Psi]_{\KN} $), we shall say that $\wh M$ is 
\emph{the F-extension  of} $\wt M$ (resp., \emph{the K-extension  of} $\wt M$) associated with the m-boundary tuple $(\Hs_{-,+},\Hs,  \Ga_0 , \Ga_1)$.
\end{defn}

The F- and K- extensions of  $\wt M$ are well-defined. Indeed, they can be written as 
$M_\Th$ with $\big(\begin{smallmatrix}
   1 & 0 \\
0 &  \ii  
\end{smallmatrix}\big)\Th = \left(\begin{smallmatrix} V & 0 \\
0 & (V^\#)^{-1} \end{smallmatrix} \right) [\Psi]_{\Fr(\KN)}$  and do not depend on the choice of $V$ (they do depend on the choice the m-boundary tuple $(\Hs_{-,+},\Hs,  \Ga_0 , \Ga_1)$). Theorem \ref{t:I-Text} shows that these F- and K- extensions are indeed extensions of 
$\wt M$ and proves also the following.

\begin{cor} \label{c:FK}
 The F-extension and  K-extension of $\wt M$ are m-dissipative  operators. 
\end{cor}

The F-extension and  the K-extension coincide if and only if $[\Psi]_\Fr = [\Psi]_{\KN}$. 
This equality can be checked with the use of the generalized Krein uniqueness criterion \cite[Theorem 4.14]{HMS04} 
which however reduces the question to the spectral analysis of the boundary operator $V^\# Z V$.

For introduced in Section \ref{s:GIBCa} abstract Leontovich-type operators $M_{\imp,\wh Z}$ with nonnegative $\wh Z:\dom \wh Z \subseteq \Hs \to \Hs$, the F- and K-extensions can be constructed by application of the above procedure to the nonnegative impedance-type operator 
$Z:\dom Z \subseteq \Hs_{-,+} \to \Hs_{+,-}$ defined by (\ref{e:domZ}). This means that we start from Friedrichs and Krein-von Neumann extensions of the operator $\Soin \wh Z \Soin$, see Theorem \ref{t:aIBCm-d}.
Thus, we can define the  F-extension $\Lc_{\Z,\Fr}$ and the K-extension $\Lc_{\Z,\KN}$ for arbitrary Leontovich-type operator $\Lc_\Z$ and, in particular, for the ``fat-fractal'' operator 
$\Lc_{\chi_\Ff}$ of Example \ref{ex:FF}. Since F- and K-extensions are m-dissipative, in this way we associate with an arbitrary measurable impedance coefficient $\Z : \pa \Om \to \ii \overline{\CC}_- $ two concrete Maxwell contraction semigroups (possibly coinciding).   

In the case of randomized impedance coefficient $\Z_\om$ of Example \ref{ex:random},
$\Lc_{\Z_\om,\Fr}$ and $\Lc_{\Z_\om,\KN}$ are operator-valued functions on the probability space $\bbOmega$ taking a.s. m-dissipative values. 

The above procedure can be generalized to the case of sectorial $\wh Z$ by application of the results of \cite{Kato13} and \cite{A00}.

\vspace{2ex}
\noindent
\emph{Acknowledgments.}  IK is grateful to the VolkswagenStiftung project   "From Modeling and Analysis to Approximation" for the financial support of his participation in the workshop ``Analytical Modeling and Approximation Methods'',  Humboldt-Universität Berlin, 4-8.03.2020.

\vspace{2ex}
\begingroup
\small
\renewcommand{\section}[2]{}

\endgroup
\end{document}